\documentclass[10pt,preprint]{amsart}
\usepackage{amssymb,amsfonts,amsthm,amsmath,amscd,relsize}
\usepackage{hyperref}
\usepackage{enumerate}
\usepackage{comment}

 \usepackage{epsfig}
 \usepackage{graphicx}
 \usepackage{lipsum}
\usepackage{xcolor}
\usepackage{float}
\usepackage{soul}
\usepackage{adjustbox}

\usepackage{enumitem}

\usepackage{pgf,tikz,pgfplots}
\pgfplotsset{compat=1.13}
\usepackage{mathrsfs}
\usetikzlibrary{arrows}

\definecolor{qqqqff}{rgb}{0,0,1}
\definecolor{qqwuqq}{rgb}{0,0.39215686274509803,0}
\definecolor{uuuuuu}{rgb}{0.266,0.266,0.266}

\def\cont{\Rightarrow\!\Leftarrow}

\def\tr{\operatorname{tr}}
\def\id{\operatorname{id}}
\def\N{\mathbb N}
\def\Z{\mathbb Z}

\theoremstyle{plain}
\newtheorem{theorem}{Theorem}
\newtheorem*{op}{Open Problem}
\newtheorem{proposition}{Proposition}
\newtheorem{lemma}{Lemma}
\newtheorem{corollary}{Corollary}

\theoremstyle{definition}
\newtheorem{definition}{Definition}

\theoremstyle{remark}
\newtheorem{remark}{Remark}
\newtheorem{example}{Example}

\newtheorem*{notation}{Notation}

\makeatletter
\@addtoreset{footnote}{page}
\makeatother

\definecolor{ttzzqq}{rgb}{0.2,0.6,0}

\begin{document}

\title[Lattice Equable Quadrilaterals]{Lattice Equable Quadrilaterals III:\\ tangential and extangential cases}

\author{Christian Aebi and Grant Cairns}

\address{Coll\`ege Calvin, Geneva, Switzerland 1211}
\email{christian.aebi@edu.ge.ch}
\address{Department of Mathematics, La Trobe University, Melbourne, Australia 3086}
\email{G.Cairns@latrobe.edu.au}

\begin{abstract} A lattice equable quadrilateral is a  quadrilateral in the plane whose vertices lie on the integer lattice  and which is equable in the sense that its area equals its perimeter.
This paper treats the tangential and extangential cases. We show that up to Euclidean motions, there are only 6 convex tangential lattice equable quadrilaterals, while the concave ones are arranged in 7 infinite families, each being given by a well known diophantine equation of order 2 in 3 variables. On the other hand, apart from the kites, up to Euclidean motions there is only one concave extangential lattice equable quadrilateral, while there are infinitely many convex ones.
\end{abstract}

\maketitle

\section*{Introduction}

A \emph{lattice equable quadrilateral} (LEQ for short) is a  quadrilateral whose vertices lie on the integer lattice $\Z^2$ and which is equable in the sense that its area equals its perimeter.
This paper is a continuation of the work \cite{AC1}, which treated lattice equable parallelograms, and \cite{AC2}, which treated lattice equable kites, trapezoids and cyclic quadrilaterals, 
but this paper can be read independently of the previous two. Here we examine LEQs that are \emph{tangential}, i.e., they have an incircle, or \emph{extangential}, i.e., they have an excircle.

Before stating our main results, let us make some general remarks about the importance and occurrence of tangential and extangential LEQs, up to Euclidean motions. Remarkably, 
tangential and extangential LEQs  apparently constitute a large component of the overall set of LEQs. For example, apart from parallelograms and trapezoids, we know of only one convex LEQ that is neither tangential nor extangential. This is the LEQ with vertices $(0,0), (2,0), (8,8), (8,15)$ and side lengths $2, 10, 7, 17$. There seems to be significantly more tangential LEQs than extangential LEQs, within a ball of any given radius of sufficient size. 
The tangential LEQs are mainly concave; indeed, as we show in Corollary~\ref{C:convex}, there are only 6 convex tangential LEQs. 
    The extangential LEQs are mainly convex;  we show in Corollary~\ref{C:main} that there  is only one concave non-kite extangential LEQ. Kites which aren't parallelograms are both tangential and extangential, and they are the only LEQs with this property.

Consider a tangential LEQ $OABC$ 
whose sides $OA,AB,BC,CO$ have length $a,b,c,d$ respectively, which therefore are integers \cite[Remark 2]{AC1}.
The key to our results on tangential LEQs is the observation that a certain pair of  functions of the side lengths take a very restricted range of possible values. The functions are:
\begin{definition}\label{D:tang}
 For a tangential LEQ $OABC$, let 
 \[
 \sigma=\frac{ad+bc+2\delta\sqrt{a b c d - 4 (a + c)^2}}{ 16 + (a-b)^2 },\quad\tau= \frac{ab+cd-2\delta\sqrt{a b c d - 4 (a + c)^2}}{ 16 + (a-d)^2 },
\]
 where $\delta=1$ if $B$ lies within the circumcircle of the triangle $OAC$, and $\delta=-1$ otherwise.
 \end{definition}
In fact, as we show in Section~\ref{S:lem}, these functions can only take the seven possible values $2,3,5,9,9/8,5/4,3/2$, and moreover $\frac1{\sigma}+\frac1{\tau}=1$.  In particular, in each case at least one of $\sigma,\tau$ is an integer and belongs to $\{2,3,5,9\}$.

For each of the  seven possibilities for the pair $(\sigma,\tau)$, we show that the side lengths satisfy a certain corresponding diophantine equation, and conversely, solutions to the equation, along with some auxiliary conditions, lead to the existence of a corresponding tangential LEQ. There is a certain redundancy both in the statement of the  seven results and their proofs, so we have been at pains to present the results in as compact a form as possible. The statements of the resulting theorem and its converse are rather cumbersome, but considerable saving is attained in the long run. Before stating the results, note that for a tangential  LEQ $OABC$ with consecutive sides $a,b,c,d$, we see in  Remark~\ref{R:choice}  that by making a reflection if necessary, we may assume that $a$ and $c$ are even in the case $\sigma=\tau=2$. Our classification result for tangential LEQs is then as follows.

\begin{theorem}\label{T:sigmatau}
Suppose that $OABC$ is a tangential  LEQ with consecutive sides $a,b,c,d$ and vertices $O,A,B,C$ in positive cyclic order. Suppose also that if $OABC$ is concave, then its reflex angle is at $B$.  
Without loss of generality we also assume that $a$ and $c$ are even in the case $\sigma=\tau=2$.
 Then the following conditions hold:
\begin{enumerate}
\item[\rm(i)] $|c-b|\tau<a+c$,
\qquad{\rm(ii)} \ $(a+d)\tau>a+c$,
\qquad{\rm(iii)}  \   $(b+c)\tau \not=a+c $.
\end{enumerate}
Moreover, $OABC$ is convex if and only if $(b+c)\tau >a+c $.
Furthermore, there are two cases:
\begin{enumerate}[leftmargin=*]
\item[\rm(I)] If $\tau \in\{2,3,5,9\}$, then $a,\tau  b$ have the same parity and setting $u=\frac{\tau  b-a}2,v=\frac{\tau  b+a}2$, we have
\begin{equation}\label{E:gentau}
(2\tau )^2 + u^2=v^2-\left(v-\frac{\tau -1}2c\right)^2.
\end{equation}

\item[\rm(II)] If $\sigma\in\{3,5,9\}$, then $a,\sigma  d$ have the same parity and setting $u=\frac{\sigma   d-a}2,v=\frac{\sigma  d+a}2$, we have
\begin{equation}\label{E:gensigma}
(2\sigma  )^2 + u^2=v^2-\left(v-\frac{\sigma  -1}2c\right)^2.
\end{equation}
 \end{enumerate}
\end{theorem} 

We now state the converse result.

\begin{theorem}\label{T:conversegen}
Let $x\in\{2,3,5,9\}$ and suppose we have an integer solution $(u,v,c)$ of the diophantine equation
\begin{equation}\label{E:gen}
(2x)^2 + u^2=v^2-\left(v-\frac{x-1}2c\right)^2
\end{equation}
for which $u+v\equiv 0\pmod x$ and $c>0$, and further that $c$ is even when $x=2$ and that $c$ is not divisible by $3$ if $x=3$.
Then we have:
\begin{enumerate}[leftmargin=*]
\item[\rm(I)] Let $t=x$, $a=v-u, b=(v+u)/t,d=a+c-b$, and suppose the following  conditions hold:
\begin{enumerate}
\item[\rm(i)] $|c-b|t<a+c$,
\qquad{\rm(ii)} \ $(a+d)t>a+c$,
\qquad{\rm(iii)}  \   $(b+c)t \not=a+c $.
\end{enumerate}
Then there is a tangential LEQ $OABC$ with successive side lengths $a,b,c,d$ for which $(\sigma,\tau)=(\frac{t}{t-1},t)$.

\item[\rm(II)] Let $s=x$, $a=v-u, d=(v+u)/s,b=a+c-d$ and suppose that  the above  conditions (i) -- (iii) hold for $t=\frac{s}{s-1}$
and that $b>0$.
Then there is a tangential LEQ $OABC$ with successive side lengths $a,b,c,d$ for which $(\sigma,\tau)=(s,t)$.
\end{enumerate}
Furthermore, in both of the above cases, if $OABC$ is concave, then the reflex angle is at $B$.
\end{theorem}

\begin{corollary}\label{C:convex} Up to Euclidean motions,  there are only six convex tangential LEQs: 
\begin{itemize}
\item the $4\times4$ square,
\item  the  isosceles trapezoid of side lengths 5,2,5,8,
\item  the right trapezoid of side lengths 5,3,4,6,
\item the equable rhombus of side length 5,
\item the equable kite of side lengths 3 and 15, 
\item the LEQ with vertices $(0,0),(40,9),(36,12),(35,12)$, and sides  37,1,5,41.
\end{itemize}
\end{corollary}

\begin{corollary}\label{C:incenter} The incenter of a tangential LEQ is an integer lattice point in the cases where $\sigma,\tau\in\{2,3,5,5/4,3/2\}$.
\end{corollary}

Examples where $\sigma,\tau\in\{9,9/8\}$ and the incenter is not an integer lattice point are given in Example~\ref{E:noninttang}.


We now turn to our results on extangential LEQs. 
Consider an extangential LEQ $OABC$ 
whose sides $OA,AB,BC,CO$ have length $a,b,c,d$ respectively.
We introduce  functions analogous to those of Definition~\ref{D:tang}. More precisely, it is convenient to define functions $\Sigma,T$ analogous to $8\sigma,8\tau$, as follows.

\begin{definition}
 For an extangential LEQ $OABC$, let 
 \begin{align*}
 \Sigma&=8\cdot\frac{ad+bc+2\delta\sqrt{a b c d - 4 (a + b)^2}}{ 16 + (a-c)^2 },\\
T&=8(a + b)^2\cdot\frac{ a b + c d +  2\delta\sqrt{a b c d - 4 (a + b)^2}}{16(a+b)^2 +  (a - c)^2 (a - d)^2},
\end{align*}
 where $\delta=1$ if $B$ lies within the circumcircle of the triangle $OAC$, and $\delta=-1$ otherwise.
 \end{definition}

The functions $\Sigma,T$ are not constrained to take only a finite number of possible values, as was the case with $\sigma,\tau$. So the study of extangential LEQs is somewhat more complicated than that of tangential LEQs.
Our main result is as follows.

\begin{theorem}\label{T:main} If a non-kite extangential  LEQ $OABC$ has consecutive sides $a,b,c,d$, then $\Sigma,T$ are integers and one the following holds:
\begin{enumerate}
\item  $(\Sigma,T)=(9,18)$ or $(18,50)$,
\item  $(\Sigma,T)=(5m^2,5m^2+5)$ for some integer $m$ for which there exists integers $n,Y,Z$ such that $m^2-10n^2=-1$ and $ (5m^2-8) Y^2= 5+8Z^2$.
\item  $(\Sigma,T)=(m^2,m^2+1)$ for some integer $m$ for which there exists integers $n,Y,Z$ such that $m^2-2n^2=-1$ and $ (m^2-8) Y^2= 1+8Z^2$.
\end{enumerate}
\end{theorem}

The situation concerning case (a) of the above theorem is very satisfactory.
We examine  the  two possibilities for $(\Sigma,T)$ in Section~\ref{S:extanleqs}, and study the corresponding extangential  LEQs up to Euclidean motions. We explicitly classify all LEQs with $(\Sigma,T)=(9,18)$; there is a single infinite family corresponding to solutions of the  negative Pell equation
$x^2-2 y^2 = -1$. For $(\Sigma,T)=(18,50)$, we prove that there is precisely one extangential  LEQ; this isolated example has sides $(a,b,c,d)=(13,2,5,10)$ and is shown on the right of Figure~\ref{F:extans}. 

We don't give a complete classification for case (b) of the above theorem. However, in Section~\ref{S:extanleqs} we consider $m=3$, which is the smallest value of $m$  for which $m^2-10n^2=-1$ has a solution. Here $(\Sigma,T)=(45,50)$, and we give explicit formulas for infinitely many such LEQs. The side lengths of the first three members of this family are given in Table~\ref{F:extanfam2a}. One sees that the lengths grow very rapidly. The next possible value of $m$ is $m=117$; see Remark~\ref{R:mnos}. Here $(\Sigma,T)=(5\cdot 117^2,5\cdot 117^2+5)$. In Example~\ref{Ex:expl}, we exhibit the smallest possible extangential  LEQ with this $(\Sigma,T)$ pair; it has perimeter $\cong 3\cdot 10^{27}$.

We do not know if there are any LEQs satisfying  condition (c) of the above theorem. Indeed, we have:

\begin{op}
Does there exist an integer solution $(m,n)$ of the negative Pell equation $m^2-2n^2=-1$, for which the diophantine equation
$ (m^2-8) Y^2= 1+8Z^2$ has an integer solution for $(Y,Z)$.
\end{op}
Even if there were such a solution, it would still be necessary to prove that there are lattice vertices that realise the corresponding side lengths.
We show at the very end of the paper that if there is
an extangential LEQ corresponding to case (c) of Theorem~\ref{T:main}, then its perimeter is at least $10^{718}$. 

As a consequence of our study,  we have the following.

\begin{corollary}\label{C:main} Up to Euclidean motions, there is only one concave non-kite extangential  LEQ; it is the LEQ with vertices $(0,0), (12,5), (10,5), (6,8)$ and side lengths $(13,2,5,10)$.
\end{corollary}

Theorem \ref{T:main}  is proved by reducing it to the following number theoretic result.

\begin{theorem}\label{T:nt}
Let $y,z,k\in \N$ with  $k>16$ and $k>yz$. Suppose that 
\begin{enumerate}
\item $\Sigma:=\frac{8 (z^2 + k)}{k-16}$ is an integer,
\item $\Sigma':=\frac{y^2  \Sigma}{k}$ is an integer,
\item $x:=\sqrt{\frac{k(\Sigma+\Sigma')}{ 8}}$ is an integer.
\end{enumerate}
Then either
\begin{enumerate}
\item  $(\Sigma,\Sigma')=(9,9),(12,24),(16,16),(24,12),(10,40),(40,10)$ or $(18,32)$,
\item  $(\Sigma,\Sigma')=(5m^2,5)$ for some integer $m$ for which there exists integers $n,Y,Z$ such that $m^2-10n^2=-1$ and $ (5m^2-8) Y^2= 5+8Z^2$,
\item  $(\Sigma,\Sigma')=(m^2,1)$ for some integer $m$ for which there exists integers $n,Y,Z$ such that $m^2-2n^2=-1$ and $ (m^2-8) Y^2= 1+8Z^2$.
\end{enumerate}
\end{theorem}

The proof of this theorem is established by writing the ratio $\frac{\Sigma'}{\Sigma}$ as $\frac{u}v$, with $\gcd(u,v)=1$, and considering the 6 cases according to whether the pair $(u,v)$ is respectively (odd,even), (odd,odd), or (even,odd), and whether the 2-adic order of the even number (respectively $v,u+v$ or $u$) is even or odd. Each of the six cases is conducted by a series of contradiction arguments.


The paper is organised in two parts. Part 1 covers tangential LEQs. Section~\ref{S:tan} develops some general results true for all tangential quadrilaterals. Section~\ref{S:tanleqs} gives explicit examples: we present calculations of the incenters of LEQs that are kites, and we give an infinite nested family of non-dart concave tangential LEQs. Section~\ref{S:lem} gives a series of lemmas on tangential LEQs leading to the definition of the key functions $\sigma$ and $\tau$, and their properties. In Section~\ref{S:proof} we give the proof of Theorem~\ref{T:sigmatau} and Corollary~\ref{C:convex}.
Section~\ref{S:converse} is the most substantial part of Part 1. Here we prove Theorem \ref{T:conversegen} and  Corollary~\ref{C:incenter}. The final section of Part 1, Section~\ref{S:disc}, gives more examples. In particular, we show that there are infinitely many LEQs for each of the seven possible choices of $(\sigma,\tau)$. 

Part 2 treats extangential LEQs. Sections~\ref{S:extan} and \ref{S:exlem} follow the general plan adopted in Sections~\ref{S:tan} and \ref{S:lem} of Part 1; Section~\ref{S:extan} presents some general results for all extangential quadrilaterals, and Section~\ref{S:exlem} gives a series of lemmas  leading to the definition of the functions $\Sigma$ and $T$, and their properties. 
Section~\ref{S:extanleqs} treats extangential LEQs in the cases where $(\Sigma,T)=(9,18),(18,50)$ and $(45,50)$.
Section~\ref{S:thms} shows how Theorem~\ref{T:main} can be deduced from Theorem~\ref{T:nt}.
Section~\ref{S:pfs} is the longest section in the paper; here we prove Theorem~\ref{T:nt}. This section also contains the proof of Corollary~\ref{C:main}, see Remark~\ref{R:coro}.
Finally, in Section~\ref{S:open} we discuss the Open Problem presented above.

\tableofcontents

\begin{notation} In this paper, a quadrilateral $OABC$ is defined by four vertices $O,A,B,C$, no three of which are colinear, such that the line segments $OA,AB,BC,CO$ have no interior points of intersection; that is, our quadrilaterals have no self-intersections. We always write the vertices $O,A,B,C$ in positive (counterclockwise) cyclic order, and if $O,A,B,C$ is concave, then the labelling is chosen so that the reflex angle is at $B$. We use the notation  $K(OABC)$ for area and  $P(OABC)$ for perimeter. 
Throughout this paper, for ease of expression, we often simply write $K$ for $K(OABC)$, and  $P$ for $P(OABC)$, and we abbreviate the triangle areas $K(COA),K(OAB),K(ABC),K(BCO)$ as $K_O,K_A,K_B,K_C$ respectively. 
By abuse of notation, we write  $OA,AB,BC,CO$  for both the sides, and their lengths; the meaning should be clear from the context. 
We also usually denote  $OA,AB,BC,CO$ by the letters $a,b,c,d$. The lengths of the diagonals $OB,AC$ are denoted $p,q$, respectively.  
We use vector notation, such as  $\overrightarrow{AB}$.  But we use the same symbol, $A$ say, for the vertex $A$ and its position vector $\overrightarrow{OA}$. Finally, by \emph{Euclidean motions}, we mean both the orientation preserving and orientation reversing kinds; that is, we consider the group generated by translations, rotations and reflections.
In this paper, we employ the term \emph{positive} in the strict sense. So $\N=\{n\in\Z \ |\ n>0\}$.
Many of the arguments in this paper are proofs by contradiction. To avoid having to write ``which gives a contradiction'' countless times, we use the contradiction symbol $\cont$ at the end of certain displayed formulas.

We used Mathematica and Maple for many of the calculations and algebraic manipulations in this paper.
The  factorizations of large numbers conducted at the end of the paper were performed using Dario Alpern's  integer factorization calculator \cite{Alpern1}. We remark that Alpern has a very nice continued fraction calculator, and a quadratic diophantine equation solver that we also found useful \cite{Alpern}.
\end{notation}

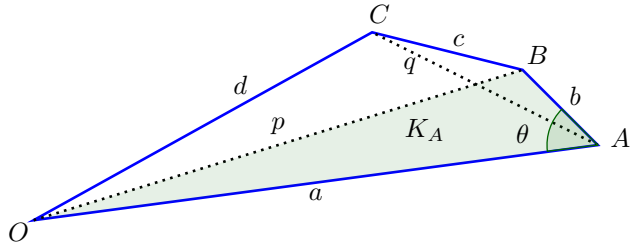
\begin{figure}[h]
\definecolor{bluee}{rgb}{0,0,1}
\definecolor{qqwuqq}{rgb}{0,0.4,0}
\begin{tikzpicture}[scale=.5][line cap=round,line join=round,>=triangle 45,x=1cm,y=1cm]
\draw[line width=1pt,color=bluee] (0,0) -- (15,2)--(13,4) --(9,5) --  cycle; 
\draw[dotted,line width=1pt,color=black]   (15,2)-- (9,5);
\draw[dotted,line width=1pt,color=black]   (0,0)-- (13,4);
\draw [shift={(15,2)},line width=.5pt,color=qqwuqq,fill=qqwuqq,fill opacity=0.1] (0,0) -- (135:1.35) arc (135:187:1.35) -- cycle;
\fill[line width=0.1pt,color=qqwuqq,fill=qqwuqq,fill opacity=0.1] (0,0) -- (15,2) -- (13,4)  -- cycle;
\draw[color=black] (-.4,-.3) node {$O$};
\draw[color=black] (15.6,2.2) node {$A$};
\draw[color=black] (13.4,4.4) node {$B$};
\draw[color=black] (9.2,5.5) node {$C$};
\draw[color=black] (5.5,3.6) node {$d$};
\draw[color=black] (7.5,.7) node {$a$};
\draw[color=black] (14.4,3.3) node {$b$};
\draw[color=black] (11.3,4.8) node {$c$};
\draw[color=black] (6.5,2.5) node {$p$};
\draw[color=black] (10.0,4.1) node {$q$};
\draw[color=black] (10.4,2.4) node {$K_A$};
\draw[color=black] (13.0,2.3) node {$\theta$};
\end{tikzpicture}
\caption{Illustration of some of the notation used}\label{F:nota}
\end{figure}

\newpage
\part{Tangential quadrilaterals}

\section{Basic notions for tangential LEQs}\label{S:tan}

It is well known and easy to see that a triangle is  equable  if and only if its incircle has radius 2. A quadrilateral that has an incircle is said to be \emph{tangential}, or \emph{circumscriptible} \cite{Mi,Jo1,Jo3,JD}. Obviously, a tangential quadrilateral is  equable  if and only if its incircle has radius 2. Pitot's theorem says that a quadrilateral with consecutive side lengths $a,b,c,d$ is tangential if and only if  the following equation holds; see \cite{Sau}, \cite[p.~62--64]{AB}  and \cite{Jo4}.
\begin{equation}\label{E:tan}
a+c=b+d.
\end{equation}
While Pitot's Theorem is usually stated only for convex quadrilaterals, it also holds in the concave case. Indeed, consider a concave quadrilateral $OABC$ with reflex angle at $B$. Let $A'$ denote the point of intersection of the side $OA$ and the extension of side $BC$. Similarly, let $C'$ denote the point of intersection of the side $OC$ and the extension of side $AB$.  Let $a,b,c,d$ denote the lengths of $OA,AB,BC,CO$ respectively, and similarly, let $a',b',c',d'$ denote the lengths of $OA',A'B,BC',C'O$.  Then it is easy to see that \eqref{E:tan} holds if and only if $a'+c'=b'+d'$; see \cite[Problem~262]{AB}. That is, $OABC$ is tangential if and only if $OA'BC'$ is tangential.

Figure \ref{F:con} gives an example of a concave tangential LEQ.
Note that, as this example shows, for a concave tangential LEQ $OABC$, while the associated convex tangential quadrilateral $OA'BC'$ is equable, it may fail to have integer sides or have its vertices on lattice points.

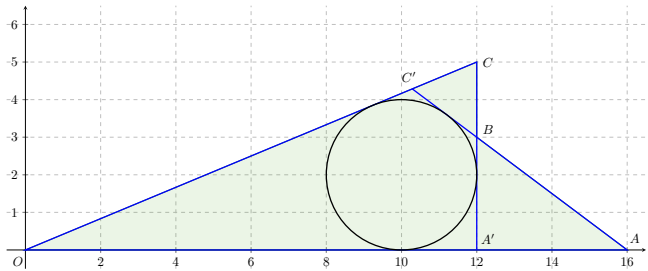
\begin{figure}[H]
\begin{tikzpicture}[scale=.5][line cap=round,line join=round,>=triangle 45,x=1cm,y=1cm]
\begin{axis}[
x=1cm,y=1cm,
axis lines=middle,
grid style=dashed,
ymajorgrids=true,
xmajorgrids=true,
xmin=-.5,
xmax=16.5,
ymin=-.5,
ymax=6.5,
xtick={0,2,...,16},
ytick={0,1,...,12},]
\draw[color=ttzzqq,fill=ttzzqq,fill opacity=0.1] (0,0) -- (16,0)--(12,3)--(12,5)-- cycle; 
\draw[line width=1pt,color=qqqqff] (0,0) -- (16,0)--(12,3)--(12,5)-- cycle; 
\draw[line width=1pt,color=qqqqff]   (12,0)-- (12,3);
\draw[line width=1pt,color=qqqqff]   (12,3)-- (72/7,30/7);
\draw[line width=1pt] (10,2) circle (2);
\draw[color=black] (-.2,-.3) node {$O$};
\draw[color=black] (16.2,.3) node {$A$};
\draw[color=black] (12.3,.3) node {$A'$};
\draw[color=black] (12.3,3.2) node {$B$};
\draw[color=black] (12.3,5) node {$C$};
\draw[color=black] (10.2,4.6) node {$C'$};
\end{axis}
\end{tikzpicture}
\caption{A concave tangential LEQ with side lengths 16,5,2,13}\label{F:con}
\end{figure}

For the rest of this section,  $OABC$ denotes a tangential (convex or concave) quadrilateral, with vertices in counterclockwise cyclic order, and $a,b,c,d$ denote the lengths of the sides $OA,AB,BC,CO$ respectively.

\begin{proposition}\label{P:kitess}
If $OABC$ is tangential, then
$OABC$ is a kite if and only if one of the diagonals divides $OABC$  into two triangles of equal area.
\end{proposition}

\begin{proof}
Obviously,  if $OABC$ is a kite, then its axis of symmetry diagonals divides $OABC$  into two triangles of equal area.
Conversely, applying Heron's formula to triangle $OAB$ gives
\begin{align*}
16K_A^2&= (a+b+p)(a+b-p)(a-b+p)(-a+b+p)\\
&=-(a^2- b^2)^2 + 2 (a^2 + b^2) p^2 - p^4.
\end{align*}
Similarly, from triangle $OBC$, we have
$16K_C^2= -(c^2- d^2)^2 + 2 (c^2 + d^2) p^2 - p^4$.
Hence, subtracting,
\begin{equation}\label{E:forp}
2(a^2 - d^2 + b^2 - c^2 )p^2=16(K_A^2-K_C^2)+(a^2- b^2)^2 -(c^2- d^2)^2.
\end{equation}
Notice that 
\begin{align*}
a^2 - d^2 + b^2 - c^2&=(a-d)(a+d)+(b-c)(b+c)\\
&=(a-d)(a+d+b+c)=2(a+c)(a-d),
\end{align*}
and
\begin{align*}
(a^2- b^2)^2 &-(c^2- d^2)^2=(a-b)^2(a+b)^2-(d-c)^2(c+d)^2\\
&=(a-b)^2(a+b+c+d)(a+b-c-d)=4(a-d)(a+c)(a-b)^2.
\end{align*}
So \eqref{E:forp} gives
\begin{equation}\label{E:areadiff}
(a+c)(a-d)p^2=4(K_A^2-K_C^2)+(a-d)(a+c)(a-b)^2.
\end{equation}
Now assume that $K_A=K_C$. Then \eqref{E:areadiff}
gives $(a-d)p^2=(a-d)(a-b)^2$.
Notice that $p=\pm(a-b)$ is impossible, as otherwise the triangle $OAB$ would be degenerate. Hence $a=d$.
Moreover, as  $K_A=K_C$, the points $A,C$ are equidistant from the line through $O,B$. So the triangles $OAB$ and $OBC$ are congruent, and hence $OABC$ is a kite.
Clearly, by considering  triangles $OAC$ and $BCA$, the same argument would hold if $K_O=K_B$.
\end{proof}

It is well known  that the incenter $I$ of a convex tangential quadrilateral lies on the  \emph{Newton line} $\mathcal{N_L}$, which is the line passing through the   midpoints of the two diagonals; see \cite[Chap.~7.5]{AN}, \cite[Chap.~2.7]{AN2} and \cite{DC}. This is also true for concave tangential quadrilaterals, because the midpoints of the three diagonals of a complete quadrilateral are colinear  (see \cite{Tan} for 23 proofs of this fact). Let $M_{A},M_{O}$ denote the midpoint of the diagonals $AC$, $OB$ respectively; see Figures~\ref{F:conv2} and~\ref{F:con2}. Notice that $M_{A},M_{O}$ are distinct, and the Newton line unambiguously defined, if and only if $OABC$ is not a parallelogram. 

\begin{figure}[h]
\begin{tikzpicture}[scale=.8][line cap=round,line join=round,>=triangle 45,x=1cm,y=1cm]
\draw[color=ttzzqq,fill=ttzzqq,fill opacity=0.1] (0,0) -- (12,0)--(10,2)-- cycle; 
\draw[color=ttzzqq,fill=ttzzqq,fill opacity=0.1] (12,3) -- (72/7,30/7)--(10,2)-- cycle; 
\draw[color=qqqqff,fill=qqqqff,fill opacity=0.1] (0,0) --  (72/7,30/7) -- (10,2)   -- cycle; 
\draw[color=qqqqff,fill=qqqqff,fill opacity=0.1] (12,3) --  (12,0) -- (10,2)   -- cycle; 
\draw[line width=1pt,color=qqqqff] (0,0) -- (12,0)--(12,3)--(72/7,30/7)-- cycle; 
\draw[line width=.5pt] (10,2) circle (2);
\draw[color=black] (-.2,.3) node {$O$};
\draw[color=black] (12.3,.3) node {$A$};
\draw[color=black] (12.3,3.2) node {$B$};
\draw[color=black] (10.2,4.6) node {$C$};
\draw[line width=1pt,dashed]   (12,0)-- (72/7,30/7);
\draw[line width=1pt,dashed]   (0,0)-- (12,3);
\draw[dotted,line width=1.5pt,color=qqqqff]   (6,1.5)-- (6+72/14,30/14);
\draw [fill=black] (10,2) circle (2pt);
\draw [fill=black] (6+72/14,30/14) circle (2pt);
\draw [fill=black] (6,1.5) circle (2pt);
\draw[color=black] (9.95,1.65) node {$I$};
\draw[color=black] (6+72/14+.4,30/14+.2) node {$M_{A}$};
\draw[color=black] (6,1.9) node {$M_{O}$};
\end{tikzpicture}
\caption{The Newton line of a convex tangential quadrilateral}\label{F:conv2}
\end{figure}
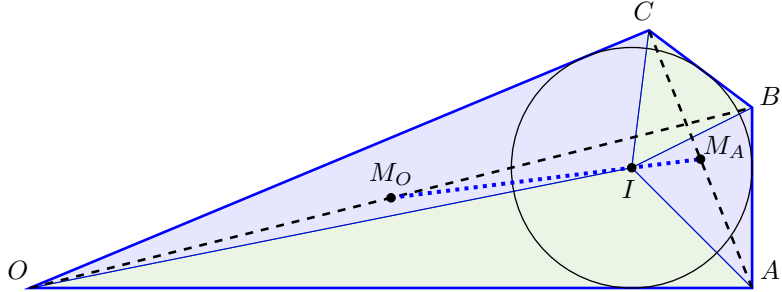

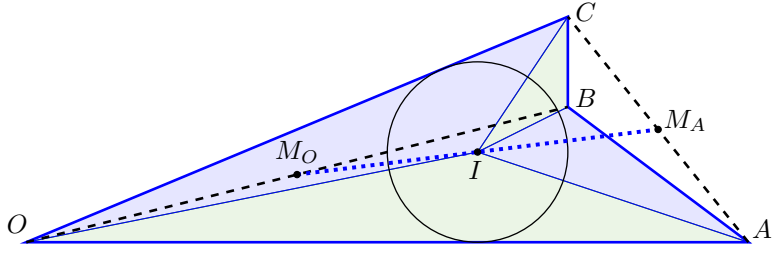
\begin{figure}[h]
\begin{tikzpicture}[scale=.6][line cap=round,line join=round,>=triangle 45,x=1cm,y=1cm]
\draw[color=ttzzqq,fill=ttzzqq,fill opacity=0.1] (0,0) -- (16,0)--(10,2)-- cycle; 
\draw[color=ttzzqq,fill=ttzzqq,fill opacity=0.1] (12,3) -- (12,5)--(10,2)-- cycle; 
\draw[color=qqqqff,fill=qqqqff,fill opacity=0.1] (0,0) --  (12,5) -- (10,2)   -- cycle; 
\draw[color=qqqqff,fill=qqqqff,fill opacity=0.1] (12,3) --  (16,0) -- (10,2)   -- cycle; 
\draw[line width=1pt,color=qqqqff] (0,0) -- (16,0)--(12,3)--(12,5)-- cycle; 
\draw[line width=.5pt] (10,2) circle (2);
\draw[color=black] (-.2,.4) node {$O$};
\draw[color=black] (16.3,.3) node {$A$};
\draw[color=black] (12.4,3.2) node {$B$};
\draw[color=black] (12.4,5.1) node {$C$};
\draw[line width=1pt,dashed]   (16,0)-- (12,5);
\draw[line width=1pt,dashed]   (0,0)-- (12,3);
\draw[dotted,line width=1.5pt,color=qqqqff]   (6,1.5)-- (14,2.5);
\draw [fill=black] (10,2) circle (2pt);
\draw [fill=black] (14,2.5) circle (2pt);
\draw [fill=black] (6,1.5) circle (2pt);
\draw[color=black] (9.95,1.6) node {$I$};
\draw[color=black] (14.6,2.7) node {$M_{A}$};
\draw[color=black] (6,2.0) node {$M_{O}$};
\end{tikzpicture}
\caption{The Newton line of a concave tangential quadrilateral}\label{F:con2}
\end{figure}

\begin{proposition}\label{P:kitess2}
If $OABC$ is tangential and is not a parallelogram, then $OABC$ is a kite if and only if
 the Newton line $\mathcal{N_L}$ contains one of the diagonals.
\end{proposition}

\begin{proof}
It is obvious that if $OABC$ is a kite, then $\mathcal{N_L}$ is the axis of symmetry of $OABC$ and hence contains a diagonal.
Conversely, suppose  $\mathcal{N_L}$ coincides with one of the diagonals, say $OB$. As $M_{A}\in \mathcal{N_L}$, we have $K(OAM_{A})=K(COM_{A})$ and  $K(ABM_{A})=K(BCM_{A})$, and hence 
\[
K_A=K(OAM_{A})+K(ABM_{A})=K(COM_{A})+K(BCM_{A})=K_C.
\]
Thus the diagonal $OB$ divides $OABC$  into two triangles of equal area. Then $OABC$ is a kite by Proposition~\ref{P:kitess}.
\end{proof}

\begin{remark}
For further equivalent conditions for a tangential quadrilateral to be a kite, see \cite{Jo0}.
\end{remark}

The radius $r$ of the incircle, called the \emph{inradius}, is given by the following obvious formula:
\[
r=\frac{K}{a+c}.
\]
We will be mainly interested in the equable case, where $r=2$, but in this section we consider the general case as it provides a useful comparison for results on the exradius of extangential quadrilaterals, which we will consider below in Part 2.

\begin{proposition}\label{P:incen}
If $OABC$ is tangential, we have the following two expressions for the incenter $I$:
\begin{enumerate}
\item $I=\frac{r}2\,\frac{aC+dA}{ K_O}$,
\qquad(b) \ $I=A+\frac{r}2\,\frac{a(B-A)-bA}{ K_A}$.
\end{enumerate}
\end{proposition}

\begin{proof}
Suppose $A,C$ have coordinates $(a_1,a_2),(c_1,c_2)$ respectively, let $I=(i_1,i_2)$ be the incenter.
Considering the area of triangle $AIO$, we have
$ra=a_1i_2-a_2i_1$.
Similarly, from  the area of triangle $COI$, we have
$rc=-c_1i_2+c_2i_1$.
Hence 
\[
r\begin{pmatrix}
a\\
c
\end{pmatrix}=\begin{pmatrix}
-a_2&a_1\\
c_2&-c_1
\end{pmatrix}\begin{pmatrix}
i_1\\
i_2
\end{pmatrix},
\quad\text{and so}\quad
\begin{pmatrix}
i_1\\
i_2
\end{pmatrix}=\frac{r}{a_1c_2-a_2c_1}\begin{pmatrix}
c_1&a_1\\
c_2&a_2
\end{pmatrix}\begin{pmatrix}
a\\
c
\end{pmatrix}.
\]
That is, $I=\frac{r}2\,\frac{aC+dA}{ K_O}$,
which is expression (a) in the statement of the proposition. Similarly, by considering triangles $BIA$ and $OAI$ we obtain
(b).
\end{proof}


Note that the above proposition holds in both convex and concave cases, but in the latter case, with a reflex angle at $B$ for example, the signed area  
$K_B$ is negative.
For more on the incenter of tangential quadrilaterals, see \cite{AHK}.

\begin{proposition}\label{P:famcor}
For a tangential equable quadrilateral $OABC$, one has
\[(K_A -(a+b))(K_O-(a+d))=
bd-ac.\]
\end{proposition}

\begin{proof} Equating the two expressions for $I$ from the above proposition, with $r=2$, and taking the vector cross product by $C$ on the right, gives
\[
d=K_O+\frac{a(K_C-K_O)-bK_O}{ K_A},
\]
so $K_OK_A-dK_A+aK_C-(a+b)K_O=0$. Thus, as $K_C=2(a+c)-K_A$,
we have $K_OK_A-(a+d)K_A-(a+b)K_O+2a(a+c)=0$. The required identity is then obtained by factorizing, using the fact that
$2a(a+c)=(a+b)(a+d)-(bd-ac)$ since $a+c=b+d$.
\end{proof}

Since the incenter $I$ lies on Newton line, $I$ is of the form $\lambda M_{A}+(1-\lambda) M_{O}$, for some $\lambda \in [0,1]$. 
The following result will use the fact  that for a (arbitrary) quadrilateral $OABC$, one has the following elementary vector equation:
\begin{equation}\label{E:college}
K_O \, B=K_C\,  A+K_A\, C.
\end{equation}
This equation is proved in  \cite{college}, as an application of the vector triple product. Alternately,  one can simply notice that the vector products $A\times(K_C\,  A+K_A\, C-K_O \, B)$ and $B\times(K_C\,  A+K_A\, C-K_O \, B)$ are both zero, so \eqref{E:college} follows as $A,B$ are linearly independent in our case.

\begin{proposition}\label{P:lambda}
If $OABC$ is tangential but is neither a parallelogram nor a kite, we have the following two expressions for the coordinate $\lambda$:
\begin{enumerate}
\item $\lambda= \frac{r(a-b)}{ 2K_O- r(a+c)}$,\qquad (b)\
$\lambda= 1- \frac{r(b-c)}{2K_A- r(a+c)}$.
\end{enumerate}
Furthermore, if $OABC$ is a kite, then the first of the above expressions for $\lambda$ holds if $OABC$ is  not a rhombus and  we relabel the vertices if necessary so that $OB$ is the axis of symmetry.\end{proposition}

\begin{proof}
By definition, $I=\lambda M_{A}+(1-\lambda) M_{O}=\lambda \frac{A+C}2+(1-\lambda) \frac{B}2$,
so using \eqref{E:college} to eliminate $B$, we have
\begin{equation}\label{E:I}
I= \frac{\lambda K_O+(1-\lambda) K_C}{2K_O}A+\frac{\lambda K_O+(1-\lambda) K_A}{2K_O}C.
\end{equation}
Comparing with Proposition \ref{P:incen}(a) gives
$rd=\lambda K_O+(1-\lambda) K_C$ and $ra=\lambda K_O+(1-\lambda) K_A$, so
\begin{align}
\lambda( K_O- K_C)&= r d -K_C \label{E:lambda1}\\
\lambda( K_O- K_A)&= r a -K_A. \label{E:lambda2}
\end{align}
Adding \eqref{E:lambda2} to \eqref{E:lambda1} and using $K_A+ K_C=r(a+c)$ gives
 $ \lambda (2K_O-r(a+c))= r(d-c)$. If $OABC$ is not a kite, then  by Proposition~\ref{P:kitess}, $ K_O\not= K_B$, so $2K_O\not=r(a+c)$. If $OABC$ is a kite, but  not a rhombus and  we relabel the vertices if necessary so that $OB$ is the axis of symmetry, then once again  $2K_O\not=r(a+c)$. In either case, 
  \[
\lambda=\frac{r(d-c)}{2K_O-r(a+c)}  =\frac{r(a-b)}{2K_O-r(a+c)},
\]
 as required.

Subtracting \eqref{E:lambda2} from \eqref{E:lambda1} gives
 $ \lambda(K_A- K_C)= r(d-a) +K_A- K_C$. If $OABC$ is not a kite, then once again $2K_O\not=r(a+c)$ by Proposition~\ref{P:kitess}, so
 \[
\lambda=1+\frac{r(d-a)}{K_A- K_C}  =1- \frac{2(b-c)}{2K_A- r(a+c)}.
\]
 \end{proof}

\begin{remark}\label{P:seg}
As we mentioned above, it is well known that for a tangential quadrilateral $OABC$, its incenter $I$ lies on the Newton line. 
It is less commonly mentioned that $I$ lies \emph{between} $M_{A}$ and $M_{O}$; that is, it lies on the  closed line segment between $M_{A}$ and $M_{O}$. 
This can be proved by an easy geometric argument. We will not require this fact, though for equable tangential quadrilaterals, it follows from the above proposition and  Remark~\ref{R:posi} below.
\end{remark}


\section{Examples of tangential LEQs}\label{S:tanleqs}

Of course, the lattice equable kites are tangential. For each of the four families $K1 - K4$ of \cite[Theorem~1]{AC2} we use Propositions~\ref{P:incen} and \ref{P:lambda} to compute the incenter $I_{n,i}$ and the parameter 
$\lambda_{n,i}$ for which $I_{n,i}=\lambda_{n,i} M+(1-\lambda_{n,i} )\frac{B}2$,  where $M=M_A$. We omit the details, which are completely routine. The results are given in Table~\ref{T:kites}.
Notice that in family K1, $n+i$ is even. Hence $I_{n,i}$ is a lattice point for all the families.

\begin{table}
\begin{tabular}{c|c|c|c|c|c}
  \hline
   Family & Equation & $M$  & $B$   & $I_{n,i}$& $\lambda_{n,i}$\\\hline
  \emph{K1}& $n^2-5i^2=4$ & $\frac12(n+5i)(2,1)$  & $n(2,1)$   &$\frac{n+i}2(2,1)$&1/5\\
  \emph{K2} & $n^2-5i^2=1$& $ (2n+5i)(2,1)$  & $4n(2,1)$   &$2(n+2i) (2,1)$ &4/5\\
\emph{K3} & $n^2-2i^2=1$& $(n+2i)(2,2) $ & $ 4n(1,1)$ &$2(n+i) (1,1)$&1/2\\
 \emph{K4} & $2n^2-i^2=1$& $ (4n+3i)(\frac32,\frac32)$ & $12n(1,1)$  &$2(3n+2i) (1,1) $&8/9\\
\hline
\end{tabular}
\bigskip
\caption{The four  families of kites}\label{T:kites}
\end{table}

 \bigskip
We will now exhibit an infinite nested family of non-dart concave tangential LEQs. 
Let $(u_i, v_i)$ be the $i$-th solution to the Pell equation $u^2-3 v^2 = 1$, with initial solution $(u_1,v_1)=(2,1)$.
From the standard theory of Pell equations, one has the recurrences:
\begin{equation}\label{E:rec}
u_{i+1} = 2 u_i + 3 v_i,\qquad
     v_{i+1} = u_i + 2 v_i.
\end{equation}
Let $A_i$ denote the point with coordinates $(x_i, y_i) = (2 u_i + 4, 6 v_i)$, and let $B$ be the point $(8,0)$.
We will consider the lattice quadrilateral $OA_iBA_{i+1}$.
To verify that $OA_iBA_{i+1}$ has no self-intersection, it suffices to calculate the vector cross products $\overrightarrow{ OA_i}\times \overrightarrow{ OA_{i+1}}$ and $ \overrightarrow{ BA_{i+1}} \times\overrightarrow{ BA_i}$, using \eqref{E:rec}, and see that they are both positive. We leave the details to the reader.

The distance $OA_i$ is  given by
\[
OA_i^2= x_i^2 + y_i^2 = 
 4 u_i^2 + 16 u_i + 16 + 36 v_i^2 = 16 u_i^2 + 16 u_i + 4 = (4 u_i + 2)^2.
 \]
So $OA_i=4 u_i + 2$. 
Similarly, the distance $A_iB$ is  given by
\[
A_iB^2=(2 u_i - 4)^2 + (6 v_i)^2 = 
 4 u_i^2 - 16 u_i + 16 + 36 v_i^2 = 16 u_i^2 - 16 u_i + 4 = (4 u_i - 2)^2.
 \]
 So $A_iB=4 u_i - 2$.
Thus $OA_iBA_{i+1}$ is tangential because 
\[
OA_i-A_iB+BA_{i+1}-A_{i+1}O= (4 u_i + 2)- (4 u_i - 2)+(4 u_{i+1} - 2)- (4 u_{i+1} + 2)=0.
\]
The perimeter $P(OA_iBA_{i+1})$ of $OA_iBA_{i+1}$ is
$OA_i+A_iB+OA_{i+1}+A_{i+1}B= 8 (u_i +  u_{i + 1})$,
while the area $K(OA_iBA_{i+1}) $ of $OA_iBA_{i+1}$ is
$4(y_{i+1}-y_i)= 24 (v_{i + 1}- v_i)$.
Hence, using \eqref{E:rec},
\begin{align*}
 K(OA_iBA_{i+1}) -   P(OA_iBA_{i+1})   &=  
24 (v_{i+1} - v_i) -8 (u_i + u_{i+1}) \\
&= 24 (u_i + v_i) - 8(3 u_i + 3v_i) = 0.
\end{align*}
So $OA_iBA_{i+1}$ is LEQ.
The vertices and side lengths of the first four members of this family are given in Table~\ref{F:tanfam}.
The first two members of the family are shown in Figure~\ref{F:tri}.

\begin{table}
\begin{tabular}{c|c||c|c|c||c|c|c|c}
  \hline
   $u_i$&$v_i$&$A_i$ & $B$ & $A_{i+1}$ & $OA_{i}$  & $A_{i}B$& $BA_{i+1}$ & $A_{i+1}O$  \\\hline
 2&1& (8,6)&(8,0)&(18,24) &10 &6 &26 &30 \\
7&4&  (18,24)&(8,0)&(56,90)& 30 &26 &102& 106 \\
26&15&  (56,90)&(8,0)&(198,336) &106 &102 &386& 390 \\
97&56&  (198,336)&(8,0)&(728,1254) &390& 386 &1446 &1450\\\hline
 \end{tabular}
\bigskip
\caption{The first four members of the tangential family}\label{F:tanfam}
\end{table}

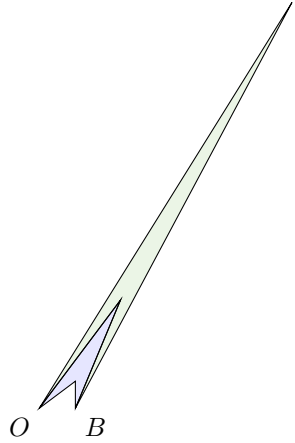
\begin{figure}[h!]
\definecolor{qqqqff}{rgb}{0,0,1}
\definecolor{ttzzqq}{rgb}{0.2,0.6,0}
\begin{tikzpicture}[scale=.06][line cap=round,line join=round,>=triangle 45,x=1.0cm,y=1.0cm]
\fill[color=blue,fill=qqqqff,fill opacity=0.1] (0,0) -- (8,6) -- (8,0)--(18,24)  -- cycle;
\fill[color=white,fill=ttzzqq,fill opacity=0.1] (0,0) -- (18,24)-- (8,0)--(56,90)  -- cycle;
\draw  (0,0) -- (8,6) -- (8,0)--(18,24)  -- cycle;
\draw  (0,0) -- (18,24) -- (8,0)--(56,90)  -- cycle;
\draw (0,0.0) node[anchor=north east] {$O$};
\draw (8,0.0) node[anchor=north west] {$B$};
\end{tikzpicture}\caption{The first two members of the family}\label{F:tri}
\end{figure}

By Proposition \ref{P:incen}(b), the incenter $I_i$ of $OA_iBA_{i+1}$ is calculated to be:
\begin{align*}
I_i&=A_i+\frac{(4 u_i + 2)(B-A_i)-(4 u_i - 2)A}{ K(A_iBO)}\\
&=(4+2u_i+2v_i,2u_i+6v_i)=A_i+(2v_i,2u_i),
\end{align*}
using $u_i^2=1+3v_i^2$. In particular,  the incenters $I_i$ are all lattice points.
From Proposition~\ref{P:lambda}, for $I_i=\lambda_i M_{O}+(1-\lambda_i) M_{A_iA_{i+1}}$, one has
\begin{align*}
\lambda_i&=\frac{4}{K(OA_iA_{i+1})-((4 u_i + 2)+(4 u_{i+1} - 2))}\\
&=\frac{4}{6(1+ 2u_i + 2v_i)-((4 u_i + 2)+4 (2u_{i} +3v_i-2))}=\frac{1}3.
\end{align*}
In particular, the family members all have the same value of the parameter $\lambda_i$. The Newton line for the first  member of the family is shown (dotted) in Figure ~\ref{F:frs}.

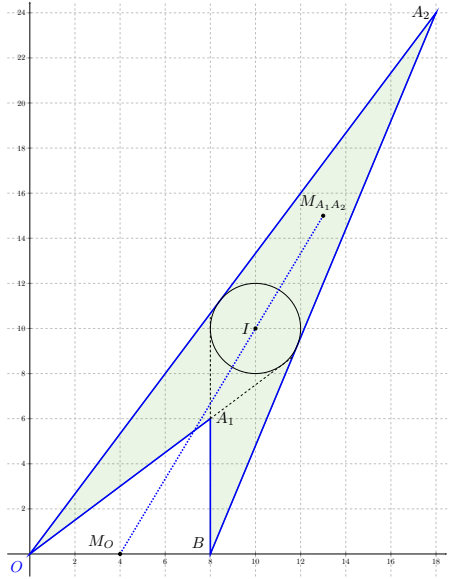
\begin{figure}[H]
\definecolor{qqqqff}{rgb}{0,0,1}
\definecolor{ttzzqq}{rgb}{0.2,0.6,0}
\begin{tikzpicture}[scale=.3][line cap=round,line join=round,>=triangle 45,x=1.0cm,y=1.0cm]
\begin{axis}[
x=1cm,y=1cm,
axis lines=middle,
grid style=dashed,
ymajorgrids=true,
xmajorgrids=true,
xmin=-1,
xmax=18.5,
ymin=-1,
ymax=24.5,
xtick={0,2,...,18},
ytick={0,2,...,24},]
\draw[color=ttzzqq,fill=ttzzqq,fill opacity=0.1]  (0,0) -- (8,6) -- (8,0)--(18,24)  -- cycle;
\draw[line width=2pt,color=qqqqff]   (0,0) -- (8,6) -- (8,0)--(18,24)  -- cycle;
\draw (8,6) node[scale=2,anchor= west] {$A_1$};
\draw (8,0) node[scale=2,anchor=south east] {$B$};
\draw (18,24) node[scale=2,anchor=east] {$A_2$};
\draw (10,10) node[scale=2,anchor=east] {$I$};
\draw (4,0) node[scale=2,anchor=south east] {$M_{O}$};
\draw (13,15) node[scale=2,anchor=south ] {$M_{A_1A_2}$};
\draw[line width=2pt,color=qqqqff]  (0,0) node[scale=2,anchor=north east] {$O$};
\draw[line width=1pt,dashed]   (8,6)-- (128/11,96/11);
\draw[line width=1pt,dashed]   (8,6)-- (8,10.5);
\draw[line width=1pt] (10,10) circle (2);
\draw[dotted,line width=2pt,color=qqqqff]   (4,0)-- (13,15);
\draw [fill=black] (4,0) circle (2pt);
\draw [fill=black] (13,15) circle (2pt);
\draw [fill=black] (10,10) circle (2pt);
\end{axis}
\end{tikzpicture}\caption{First  member of the family}\label{F:frs}
\end{figure}

\section{Lemmata for tangential LEQs}\label{S:lem}

For this section, $OABC$ denotes an equable tangential quadrilateral. In particular, it has inradius $r=2$.
Let $\theta$ denote the interior angle of  $OABC$ at $A$; see Figure~\ref{F:nota}. By the cosine rule, $p^2=a^2+b^2-2ab\cos\theta$. As $|ab\cos\theta|=\sqrt{a^2b^2-a^2b^2\sin^2\theta}=\sqrt{a^2b^2-4K_A^2}$, so 
\begin{equation}\label{E:p1}
 p^2=a^2+b^2\pm2  \sqrt{a^2b^2-(2K_A)^2},
\end{equation}
where the sign of the square root depends on whether $\theta$ is acute or obtuse. Similarly,
\begin{equation}\label{E:q1}
 q^2=a^2+d^2\pm2  \sqrt{a^2d^2-(2K_O)^2}.
\end{equation}
The distances $p,q$ may fail to be integers (see \cite[Theorem~4]{AC2}), but as $O,A,B,C$ are lattice points, $p^2,q^2$ are integers. So the following lemma is immediate from \eqref{E:p1} and \eqref{E:q1}, and doesn't require the equability or tangential hypothesis.

\begin{lemma}\label{L:pq}  
The integers $a^2b^2-(2K_A)^2$ and $a^2d^2-(2K_O)^2$ are squares. 
\end{lemma}

\begin{lemma}\label{L:pandq} If $OABC$ is not a kite, one has 
\[
p^2=
\frac{8 (K_A - K_C)}{a-d} + (a - b)^2\quad\text{and}\quad
q^2=
\frac{8 (K_O - K_B)}{a - b} + (a-d)^2.
\]
Furthermore, if $OABC$ is  a kite with $OB$ as its axis of symmetry, and if $OABC$ is not a rhombus, then the above formula for $q^2$ still applies and one has the following formula for $p^2$:
\[
p^2=\frac{2(a+c)^2(a - b)}{K_O - K_B}.\]
\end{lemma}

\begin{proof} Arguing exactly as in Proposition \ref{P:kitess} we reobtain \eqref{E:areadiff}:
\[
(a+c)(a-d)p^2=4(K_A^2-K_C^2)+(a-d)(a+c)(a-b)^2.
\]
If $a=d$, then by the tangential hypothesis, $b=c$, so $OABC$ is a kite. Thus, if $OABC$ is not a kite, $a\not=d$ and we have
\[
p^2=4(K_A^2-K_C^2)+(a-b)^2.
\]
Then as $K_A^2-K_C^2=(K_A+K_C)(K_A-K_C)=2(a+c)(K_A-K_C)$,
from which the required formula for $p^2$ follows. Similarly, the formula for $q^2$ is obtained by applying Heron's formula to triangles $OAC$ and $BCA$.

If $OABC$ is  a kite with $OB$ as its axis of symmetry, and  is not a rhombus, then $a=d,b=c$ and $a\not=b$, and the argument giving the formula for $q^2$ remains valid. For $p^2$, we use the standard formula for the area of a kite: $pq=2K$. So $p^2=16(a+c)^2/q^2=\frac{2(a+c)^2(a - b)}{K_O - K_B}$,
as required.
\end{proof}

\begin{remark}\label{R:KAint} If $OABC$ is not  a kite, then from the above lemma, using  \eqref{E:p1},
\[
\frac{8 (K_A - (a+c))}{a-d} =\frac{p^2- (a - b)^2}2= ab\pm  \sqrt{a^2b^2-(2K_A)^2}, 
\]
which is an integer by Lemma \ref{L:pq}. Similarly, $\frac{8 (K_O - (a+c))}{a-b}$  is an integer. If $OABC$ is  a kite with $OB$ as its axis of symmetry, and if $OABC$ is not a rhombus, then by the same reasoning,  $\frac{8 (K_O - (a+c))}{a-b}$  is again an integer.
\end{remark}

At this point we pause to explain the investigation we are about to perform. The integers $\frac{8 (K_A - (a+c))}{a-d}$ and $\frac{8 (K_O - (a+c))}{a-b}$, defined above for non-kites, will play a key role in what follows. Using Proposition~\ref{P:lambda} (or Proposition~\ref{P:famcor}) one could easily directly show that these integers obey an important relation: their product is $8$ times their sum (see Lemma~\ref{L:sum} below). This enables us to show that these integers are restricted to a small set of possibilities (see Lemma~\ref{L:seven} below). However, we will follow a somewhat more circuitous route to this result. We proceed by developing results that will lead to Definition~\ref{D:sigmatau} which holds for all tangential LEQs (kites as well as non-kites). This enables us to then progress in a more natural manner, without having to appeal to the classification of kites in \cite{AC1}. Although it involves some unpleasant computations, this pathway forward also has the advantage that it reveals certain important relations that will be useful in what follows.

\begin{lemma}\label{T:sq1} The integer 
$abcd-4( a+c)^2 $
 is a square, and
 \begin{align*}
K_A&=(a+c)+(a-d)\frac{ab+cd\pm 2\sqrt{a b c d - 4 (a + c)^2}}{16 + (a-d)^2}\\
K_O&=(a+c)+(a-b)\frac{ad+bc\mp  2 \sqrt{abcd-4( a+c)^2}}{ 16 + (a-b)^2 },
\end{align*}
where  the signs of the square roots in the formulas for $K_O$ and $K_A$ are opposite.
\end{lemma}


\begin{remark}\label{R:posit} In the statement of the above lemma, the terms 
\[
ab+cd\pm 2\sqrt{a b c d - 4 (a + c)^2}\quad \text{and}\quad ad+bc\mp  2 \sqrt{abcd-4( a+c)^2}
\]
 are strictly positive. Indeed, using $d=a-b+c$, by the arithmetic mean-geometric mean inequality,
$ab+cd \ge 2\sqrt{a b c d}>\sqrt{a b c d - 4 (a + c)^2}$.
In particular, $K_A<a+c$ if and only if $a<d$.
\end{remark}

\begin{proof}[Proof of Lemma~\ref{T:sq1}]
The formulas for $K_A$ and $K_O$ obviously hold when  $OABC$ is a rhombus. 
 So, without loss of generality, we may assume that either $OABC$ is not a kite, or is a kite that is not a rhombus and has axis of symmetry $OB$.
 Then, from Lemma \ref{L:pandq} and Equation \eqref{E:q1},
\[
\frac{4 (K_O - K_B)}{a-b} -ad  =\frac{q^2-a^2-d^2}2 =\pm  \sqrt{a^2d^2-(2K_O)^2}
\]
so squaring, using $K_O+K_B=2(a+c)$ and rearranging gives
\[
16(K_O - (a+c))^2 -4ad(a-b)(K_O - (a+c))=-(a-b)^2K_O^2.
\]
Let $s:=\frac{K_O - (a+c)}{a-b}$. Thus 
\begin{equation}\label{E:seqn}
16s^2 +((a-b)s+(a+c))^2 =4ad s.
\end{equation}
Hence $\alpha s^2+2\beta s+\gamma=0$,
where, using $a+c=b+d$, 
\[
   \alpha =16 + (a-b)^2,\quad
 \beta=
  -(a d + c b),\quad
\gamma=(a + c)^2.
\]
Thus, as $\beta^2-\alpha\gamma=4 (a b c d - 4 (a + c)^2)$ (using $a+c=b+d$ again), we have
\[
s=\frac{a d + c b\pm 2\sqrt{a b c d - 4 (a + c)^2}}{16 + (a-b)^2},
\]
which gives the required formula for $K_O$. In particular, as $s$ is rational, $abcd-4( a+c)^2$ is a square, as claimed.
The formula for $K_A$ is similarly obtained by equating $p^2$ from Lemma \ref{L:pandq} and Equation \eqref{E:p1}.

It remains to see that the signs of the square roots in the formulas for $K_O$ and $K_A$ are opposite.
Let $R=2\sqrt{a b c d - 4 (a + c)^2}$. Obviously, we may assume that $R\not=0$. Let us write
 \begin{align*}
K_A&=(a+c)+(a-d)\frac{ab+cd+\delta_A R}{16 + (a-d)^2}\\
K_O&=(a+c)+(a-b)\frac{ad+bc+\delta_O   R}{ 16 + (a-b)^2 },
\end{align*}
where $\delta_A,\delta_O$ are each $\pm1$. 
Using $a+c=b+d$,
\begin{align*}
K_A-(a+b)&=(d-a)+(a-d)\frac{ab+cd+\delta_A R}{16 + (a-d)^2}=(d-a)\frac{16 -(a c + b d)+\delta_A R}{16 + (a-d)^2} \\
K_O-(a+d)&=(b-a)+(a-b)\frac{ad+bc+\delta_O   R}{ 16 + (a-b)^2 }=(b-a)\frac{16 -(a c + b d) +\delta_O R}{16 + (a-b)^2} .
\end{align*}
Notice also that $(d-a)(b-a)=bd-ac$. Hence, by Proposition~\ref{P:famcor},
\begin{equation}\label{E:cur}
\frac{16 -(a c + b d)+\delta_AR}{16 + (a-d)^2}\cdot \frac{16 -(a c + b d) +\delta_O R}{16 + (a-b)^2}=1.
\end{equation}
Now,
\begin{align*}
&\frac{16 -(a c + b d)+\delta_AR}{16 + (a-d)^2}\cdot \frac{16 -(a c + b d) -\delta_A R}{16 + (a-b)^2}\\
&\quad=\frac{(16 -(a c + b d))^2-R^2}{(16 + (a-d)^2)(16 + (a-b)^2}=\frac{(16 -(a c + b d))^2-4(a b c d - 4 (a + c)^2)}{(16 + (a-d)^2)(16 + (a-b)^2)},
\end{align*}
and substituting $d=a+c-b$ one finds that this expression reduces to 1. Hence, if $\delta_O=\delta_A$, \eqref{E:cur} gives
\[
\frac{16 -(a c + b d)+\delta_AR}{16 + (a-d)^2}\cdot \frac{2\delta_A R}{16 + (a-b)^2}=0,
\]
which can only happen if $16 -(a c + b d)+\delta_A R=0$, as $R\not=0$. But in that case, one would have
$(16 -(a c + b d))^2=R^2$, which is impossible, since we have already seen that substituting $d=a+c-b$ one has
\[
(16 -(a c + b d))^2-R^2= (16 + (a-d)^2) (16 + (a-b)^2)>0.
\]
So $\delta_O=-\delta_A$, as claimed.
\end{proof}

A tangential quadrilateral is cyclic if and only if its area is given by $K=\sqrt{abcd}$ \cite[Theorem 4]{Jo12}. Hence the integer 
$abcd-4( a+c)^2 $ in the above proposition
is zero if and only if  $OABC$ is cyclic. This motivates the following result.

\begin{lemma}\label{P:sign} 
The sign of the square root in the formulas for $K_O$ is positive if and only if $B$ lies within the circumcircle of the triangle $OAC$; in particular, the sign for $K_O$  is positive if $OABC$ is concave.
\end{lemma}

\begin{proof} In the notation of the above proof, let $x=\delta_O 2\sqrt{a b c d - 4 (a + c)^2}$, so
 \begin{align*}
K_O&=(a+c)+(a-b)\frac{ad+bc+x}{ 16 + (a-b)^2 }\\
K_A&=(a+c)+(a-d)\frac{ab+cd-x}{16 + (a-d)^2}.
\end{align*}
From a standard criteria for a point to be within the circumcircle of a triangle (see \cite{Fo}),
$B$ is inside the circumcircle of the triangle $OAC$ if and only if
\begin{equation}\label{E:inside}
p^2K_O<d^2K_A+a^2K_C.
\end{equation}
First suppose that $OABC$ is not a kite. Now $d^2K_A+a^2K_C=K_A(d^2-a^2)+2a^2(a+c)$, and by Lemma \ref{L:pandq},
\[
K_Op^2
=K_O\left(\frac{16 (K_A - (a+c))}{a-d} + (a - b)^2\right).
\]
Also, by Proposition \ref{P:famcor}, $K_OK_A=(a+d)K_A+(a+b)K_O-2a(a+c)$.
So condition \eqref{E:inside} can be written as $E>0$ where 
\begin{align*}
E&=K_A(d^2-a^2)+2a^2(a+c)\\
&\quad-\left(\frac{16 ((a+d)K_A+(a+b)K_O-2a(a+c) - (a+c)K_O)}{a-d} + (a - b)^2K_O\right).
\end{align*}
Substituting the formulas for $K_O$ and $K_A$, one has
 \begin{align*}
E=& ( d^2-a^2 ) \left(a + c + \frac{(a - d) (a b + c d - x)}{
    16 + (a - d)^2}\right)+2 a^2 (a + c) \\&
     - \frac{16}{a - d}
  \left( (a + d) \left(a + c + \frac{(a - d) (a b + c d - x)}{
       16 + (a - d)^2} \right)-2 a (a + c)\right.\\&
       \left. + ( b- c) \left(a + c + \frac{(a - b) (b c + a d + x)}{
       16 + (a - b)^2}\right)\right)\\
       &
     - (a - b)^2 \left(a + c + \frac{(a - b) (b c + a d + x)}{
    16 + (a - b)^2}\right).
\end{align*}
Substituting $d=a+c-b$ one finds that the above expression reduces (rather miraculously) to $E=(a+c)x$. Hence, as claimed, $x>0$ if and only if 
$B$ is inside the circumcircle of the triangle $OAC$.

Now consider the  case where $OABC$ is a kite with axis of symmetry $OB$. Then  $a=d,b=c,K_A=K_C=a+c$ and
condition \eqref{E:inside} is:
$p^2K_O<2a^2(a+c)$.
By Lemma \ref{L:pandq}, $
p^2=\frac{2(a+c)^2(a - c)}{ K_O - K_B}=\frac{(a+c)^2(a - c)}{K_O -(a+c)}$. So the required condition is $E>0$, where
\[
E=\frac{2a^2 (K_O -(a+c)) }{(a+c)(a - c)}-K_O= \frac{(a^2+c^2 )K_O -2a^2(a+c) }{a^2 - c^2}.
\]
Substituting for $K_O$, and using $d=a,b=c$, one finds that the above expression reduces to 
\[
E=\frac{-16 (a+c)^2 + 4 a^2 c^2 + x(a^2  + c^2) }{(a + c) (16 + (a-c)^2)}.
\]
Notice that the denominator of $E$ is positive, and in the numerator, $-16 (a+c)^2 + 4 a^2 c^2 =x^2$, so the numerator is $x(a^2  + c^2+x)$.
Now $a^2  + c^2+x>0$ since $(a^2  + c^2)^2-x^2=(a^2  - c^2)^2+16 (a+c)^2 >0$.
Hence $x>0$ if and only if 
$B$ is inside the circumcircle of the triangle $OAC$.
\end{proof}

We now restate the definition given in the introduction.

 \begin{definition}\label{D:sigmatau} Let 
 \[
 \sigma=\frac{ad+bc+2\delta\sqrt{a b c d - 4 (a + c)^2}}{ 16 + (a-b)^2 },\quad\tau= \frac{ab+cd-2\delta\sqrt{a b c d - 4 (a + c)^2}}{ 16 + (a-d)^2 },
\]
 where $\delta=1$ if $B$ lies within the circumcircle of the triangle $OAC$, and $\delta=-1$ otherwise.
 \end{definition}

\begin{remark}\label{R:KOKA} From the above definition and Lemma~\ref{T:sq1} and \ref{P:sign},
\begin{align}
K_O&=a+c+(a-b)\sigma,\label{E:KO}\\
K_A&= a+c+(b-c)\tau.\label{E:KA}
\end{align}
\end{remark}

\begin{remark}\label{R:stlam} By Lemma \ref{T:sq1} and \ref{P:sign} and Proposition~\ref{P:lambda}, if $OABC$ is a tangential LEQ that is not a rhombus, then  $\lambda= \frac1{\sigma}$. 
\end{remark}

\begin{remark}\label{R:posi} By Remark \ref{R:posit}, $ \sigma$ and $\tau$  are both strictly positive. 
\end{remark}

\begin{remark}\label{R:steqns}  We saw in the proof of Lemma \ref{T:sq1}, in Equation \eqref{E:seqn}, that for $a\not= b$, one has, using $a+c=b+d$,
\begin{equation}\label{E:seqn2}
16  \sigma^2+((d-c)  \sigma+(a+c))^2=4ad \sigma.
\end{equation}
It is easy to verify directly that this equation also holds when $a=b$. Similarly, the  following equation holds in all cases:
\begin{equation}\label{E:teqn}
16\tau^2+((b-c)\tau+(a+c))^2=4ab\tau.
\end{equation}
\end{remark}

\begin{remark}\label{R:choice}  Suppose $\sigma=\tau=2$. Then \eqref{E:seqn2} gives
\[
16  \cdot 2^2+(2(d-c) +(a+c))^2=8ad.
\]
In particular, $(2(d-c) +(a+c))^2$ is divisible by 8, and hence, being a square, it is divisible by 16. In particular, $a+c$ is even. Furthermore $ad$ must be even.
Hence, by a reflection  in the line $y=x$ if necessary, we may assume that $a$ is even. Then as $a+c$ is even, $c$ is also even.

Suppose $\tau=3$. Then \eqref{E:teqn} gives
\[
16  \cdot 3^2+(a+3b-2c)^2=12ab.
\]
In particular,  $a+3b-2c$ is even so $a$ and $3b$ have the same parity, and we can pose $(3b-a)/2=u$ and $(3b+a)/2=v$. This gives
$36+(u+a-c)^2=3ab$, so $36+(u+c)^2=3ab-a^2-2ua+4uc+2ac=6bc$. Hence $u+c$ is divisible by 3, say $u+c=3k$, so $bc$ is divisible by 3. But $4+k^2$ is not divisible by 3, so $bc$ is not divisible by $9$. Thus precisely one of the numbers $b,c$ is divisible by 3. 
Hence, by a reflection  in the line $y=x$ is necessary, we may assume that $c$ is not divisible by 3, and that $b$ is divisible by 3.
By the same reasoning, for  $\sigma=3$, we may assume that $c$ is not divisible by 3, and that $d$ is divisible by 3.
\end{remark}

\begin{lemma}\label{L:sum} One has
${\sigma}+{\tau}=\sigma\tau$.
\end{lemma}

\begin{proof} As in the proof of Lemma~\ref{P:sign}, let  $x=2\delta\sqrt{a b c d - 4 (a + c)^2}$. Then, cross-multiplying,  the required identity is
$E=0$, where 
\[
E=(ad+bc+x)(16 + (a-d)^2)-(ab+cd-x)(16 + (a-b)^2 )-(ad+bc+x)(ab+cd-x).
\]
Expanding and using $d=a+c-b$, one has 
\[
E=x^2+4 (4 a^2 + 8 a c - a^2 b c + a b^2 c + 4 c^2-a b c^2).
\]
Then replacing $x^2$ by $4(a b c d - 4 (a + c)^2)$ and using $d=a+c-b$ again gives
 $E=0$, as required.
\end{proof}

\begin{remark}\label{R:integers} Observe that $8 \sigma$ and $8\tau$ are integers. Indeed, if $OABC$ is not a kite, then 
from Lemma~\ref{T:sq1}, 
\[
8 \sigma=\frac{8 (K_O - (a+c))}{a-b},\qquad8\tau= \frac{8 (K_A - (a+c))}{a-d},
\]
which are integers by Remark \ref{R:KAint}.
If $OABC$ is a kite but not a rhombus, with for example, axis of symmetry $OB$ so $a=d,b=c$, then $\sigma$ is still given by the above formula and is an integer by Remark~\ref{R:KAint}, while
$8\tau= ac-\delta\sqrt{a^2 c^2  - 4 (a + c)^2}$, which is an integer by Lemma~\ref{T:sq1}. In fact, if $OABC$ is a kite that is not a rhombus, then by \cite[Theorem~1]{AC2}, $OABC$ appears in Table~\ref{F:tanfam}, at the beginning of 
Section~\ref{S:tanleqs}. Its $\lambda$ value is thus either   $1/5, 4/5,1/2$ or $8/9$, and so here $(\sigma,\tau)$ is either $(5,5/4),(5/4,5),(2,2)$ or $(9/8,9)$ respectively, by Remark~\ref{R:stlam}. 
If  $OABC$ is a rhombus, then by \cite[Corollary~1]{AC2}, $OABC$ is either the $4\times4$ square or the equable rhombus of side length 5.
Furthermore, $8\sigma$ and $8\tau$ are $a^2\pm\sqrt{a^4  - 16 a^2}$, which are also   integers
by Lemma~\ref{T:sq1}.  For the $4\times4$ square, this gives $ (\sigma,\tau)=(2,2)$. For the rhombus of side length 5, if one chooses $OB$ to be the longest diagonal, then $ (\sigma,\tau)=(\frac54,5)$, while  if $OB$ is the shortest diagonal, then $ (\sigma,\tau)=(5,\frac54)$. 
\end{remark}

\begin{lemma}\label{L:seven}  The only possibilities for the unordered pairs  $\{\sigma,\tau\}$ are 
$\{9,\frac98\}$, $\{5,\frac54\}$, $\{3,\frac32\}$ and $\{2,2\}$.
\end{lemma}

\begin{proof} By Remarks \ref{R:posi}  and \ref{R:integers}, $\sigma'=8\sigma,\tau'=8\tau$ are positive integers and by Lemma~\ref{L:sum}, $\sigma'\tau' =8(\sigma'+\tau')$ which can be written as
\begin{equation}\label{E:st}
 (\sigma' -8)(\tau' -8)=2^6.
\end{equation}
The only positive integer solutions of the above equation are then  
\[
\{\sigma', \tau'\} \in \{\{9,72\},\{10,40\},\{12,24\},\{16,16\}\},
\]
giving the result announced.
\end{proof}

As mentioned in Remark \ref{R:stlam}, if $OABC$ is a tangential LEQ that is not a rhombus, then   $\lambda=\frac1{\sigma}$. So Lemma \ref{L:seven} has the following corollary.

\begin{corollary}\label{C:lambdas} There are only seven possibilities for the barycentric coordinate parameter $\lambda$, namely $\frac12,\frac13, \frac23,\frac15,\frac45,\frac19,\frac89$, corresponding  to $\sigma=2,3,\frac32,5,\frac54,9,\frac98$ respectively.
\end{corollary}

\begin{remark}\label{R:inter} Consider a reflection in the line $y=x$, followed by a relabelling of the vertices so they are positively oriented; that is, the vertices $O,A,B,C$ are permuted
to $O,C,B,A$ respectively. It is easy to see that under this operation, $\sigma$ and  $\tau$ are left unchanged, and the side lengths $a,b,c,d$ are permuted
to $d,c,b,a$ respectively. 

Notice that for convex tangential LEQs (where we are not  concerned about having the reflex angle at $B$), under the rotation for which the vertices $O,A,B,C$ are permuted
to $A,B,C,O$ respectively, $\sigma$ and  $\tau$ are interchanged, and the side lengths $a,b,c,d$ are permuted
to $d,a,b,c$ respectively. 
So, for the study of convex tangential LEQs, up to Euclidean motions, we may assume that $\sigma\le \tau$; that is, 
$\tau\in\{2,3,5,9\}$. 

Notice also for convex  tangential LEQs, under the rotation for which the vertices $O,A,B,C$ are permuted
to $B,C,O,A$ respectively, $\sigma$ and  $\tau$ are also left unchanged, and the side lengths $a,b,c,d$ are permuted
to $c,d,a,b$ respectively. Note that the two permutations $\sigma_1:(a,b,c,d)\mapsto (d,c,b,a)$ and $\sigma_2:(a,b,c,d)\mapsto (c,d,a,b)$ are involutions and their compositions give a group of order 4, under which each letter can be moved to any of the four positions.  So, for example, without changing $\sigma$ and  $\tau$, we may assume in the convex case that 
$b$ is the smallest of the side lengths. Note however that when $\sigma=\tau=2$, this potentially conflicts with the requirement in Theorem~\ref{T:sigmatau} (and in the proof of Corollary~\ref{C:convex} which uses Theorem~\ref{T:sigmatau}) that we also require $a$ and $c$ to be even.  As we saw in Remark~\ref{R:choice}, $a+c$ and $ad$ are even when $\sigma=\tau=2$. So by reflection we may suppose that $a,c$ to be even. Then by applying $\sigma_2$ if necessary, we may assume that $b\le d$. 

In summary, for convex  tangential LEQs we may assume:
\begin{enumerate}
\item $\tau\in\{2,3,5,9\}$,
\item $a,c$ are even and $b\le d$ when $\tau=2$, 
\item $b$ is the smallest of the side lengths when $\tau\in\{3,5,9\}$.
\end{enumerate}
\end{remark}

\section{Proof of Theorem \ref{T:sigmatau} and Corollary~\ref{C:convex}}\label{S:proof}

\begin{proof}[Proof of Theorem \ref{T:sigmatau}] Lemma \ref{L:seven} gives 7 possibilities for the ordered pair $(\sigma,\tau)$. 
Using $\frac1{\tau}+\frac1{\sigma}=1$ and $a+c=b+d$, let us restate \eqref{E:KA} and \eqref{E:KO}:
\begin{align}
K_A&= a+c+(b-c)\tau,\label{E:KA2}\\
K_O&=a+c+(a-b)\frac{\tau}{\tau-1}.\label{E:KO2}
\end{align}
We now consider the area restrictions:
\begin{enumerate}
\item As $K_A>0$, so \eqref{E:KA2} gives $a+c+(b-c)\tau>0$, which gives part of (i).
\item As $K_C>0$ and $K_C=2(a+c)-K_A$, so \eqref{E:KA2} gives $a+c-(b-c)\tau>0$, which gives the other part of (i).
\item As $K_O>0$, so \eqref{E:KO2} gives  $(a+c)(\tau-1)+(a-b)\tau>0$, which gives (ii). 
\item We also have $K_B\not=0$ as otherwise $ABC$ would be colinear. Thus $K_O\not=2(a+c)$ and  \eqref{E:KO2}  gives 
$(a-b)\frac{\tau}{\tau-1}\not=a+c$, which gives (iii).
\end{enumerate}
Further, $OABC$ is convex if and only if $K_O<2(a+c)$. As we have just seen in part (d), this occurs when $(a-b)\frac{\tau}{\tau-1}<a+c$; that is, when $(b+c)\tau >a+c $.

\bigskip
Recall that from Remark~\ref{R:steqns},
\begin{align}
16\tau^2+((b-c)\tau+(a+c))^2&=4ab\tau,\label{E:pell}\\
16  \sigma^2+((d-c)  \sigma+(a+c))^2&=4ad \sigma.\label{E:pell2}
\end{align}

(I). If $\tau \in\{2,3,5,9\}$,  then \eqref{E:pell} gives 
$16  \tau ^2 + (\tau b +a-(\tau-1) c)^2=4\tau ab$. 
So $\tau b +a-(\tau-1) c$ is even. If $\tau=2$, then $a,\tau  b$ have the same parity since $a$ is even by assumption.  If $\tau \in\{3,5,9\}$, then  $(\tau-1) c$ is even, so $a,\tau  b$ again have the same parity. Thus $v=\frac{\tau b+a}2$ is an integer, and we have
$(2 \tau) ^2 + (v-\frac{\tau -1}2c)^2=\tau ab$. Then since $\tau ab=v^2-u^2$, we have \eqref{E:gentau} as required.

\bigskip
(II). This case is completely analogous to case (I). Let $\sigma \in\{3,5,9\}$,  then \eqref{E:pell2} gives 
$16  \sigma ^2 + (\sigma d +a-(\sigma-1) c)^2=4\sigma ab$. 
So $\sigma d +a$ is even and $v=\frac{\sigma  d+a}2$ is an integer. We have
$(2 \sigma) ^2 + (v-\frac{\sigma -1}2c)^2=\sigma ad$. Then since $\sigma ad=v^2-u^2$, we have \eqref{E:gensigma} as required.
\end{proof}

\begin{proof}[Proof of Corollary  \ref{C:convex}] We use the notation of Theorem \ref{T:sigmatau}.
By Remark~\ref{R:inter}, we may assume that $\tau\in\{2,3,5,9\}$,
that $a,c$ are even and $b\le d$ when $\tau=2$, 
and that $b$ is the smallest of the side lengths when $\tau\in\{3,5,9\}$.

As in the statement of Theorem~\ref{T:sigmatau}, let $u=\frac{\tau b-a}2,v=\frac{\tau b+a}2$. Rewriting the conditions (a), (b), (c)
of Theorem~\ref{T:sigmatau}, we have
\begin{align}
(\tau-1)c &<2v,\label{E:conda}\\
2u&<(\tau+1)c,\label{E:condb}
\end{align}
and the convexity condition is
\begin{equation}\label{E:conv}
2u > -c(\tau-1). 
\end{equation}
So by \eqref{E:condb} and \eqref{E:conv}, we have $-\frac{\tau-1}2c<u<\frac{\tau+1}2c$.
When $\tau=2$, as $b\leq d$, we have $2b\le b+d=a+c$, so $u\le \frac12c$. When $\tau \in\{3,5,9\}$, 
as $b$ is the smallest of the side lengths, we have
$\tau b\le (\tau-1)c+a$, so $u=(\tau b-a)/2\le \frac{\tau-1}2c$. Thus, in all cases, we have
\begin{equation}\label{E:utau}
-\frac{\tau-1}2c<u\leq \frac{\tau-1}2c. 
\end{equation}

Assume $\tau=2$. By \eqref{E:utau}, we have $u^2\le \frac14 c^2$. Thus by Theorem~\ref{T:sigmatau}, 
$16 + u^2=v^2-(v-\frac12c)^2$ gives
\[
16 +\frac14c^2\ge 16 + u^2=vc-\frac14c^2,
\]
from which it follows that 
\begin{equation}\label{E:abd}
32 \ge c(2v-c).
\end{equation}
From \eqref{E:conda}, we have $2v>c$. So \eqref{E:abd} has only a finite number of solutions. Indeed, one finds readily there are just 20 such pairs  $c, v$ with $c$ even and $2v>c$ for which \eqref{E:abd} holds.  For only three of these pairs does the equation $16 + u^2=v^2-(v-\frac12c)^2$ have an integer solution for $u$ with $u+v$ even; these are  $(c,v,u)=(4,2,6),(2,1,9),(8,4,6)$, corresponding to the sides $(a,b,c,d)=(4,4,4,4),(8,5,2,5),(2,5,8,5)$ respectively. The last two cases correspond to the same LEQ, up to Euclidean motion.

Assume $\tau=3$. By \eqref{E:utau}, we have   $-c<u\le c$. So 
$36 + u^2=v^2-(v-c)^2$ gives
$36 +c^2\ge  36 + u^2=2vc-c^2$,
from which it follows that 
\begin{equation}\label{E:bbd}
18 \ge c(v-c).
\end{equation}
By \eqref{E:conda}, we have $v> c$, so $c,v-c\in\{1,\dots,18\}$.  
One finds there are just 58 pairs  $c, v$ with $v> c$ for which \eqref{E:abd} holds.  Of these, there is only one where the equation $16 + u^2=v^2-(v-c)^2$ has an integer solution $u$ for which $v+u\equiv 0 \pmod 3$, and such that for the resulting side lengths $(a,b,c,d)$, one has $b=\min\{a,b,c,d\}$; this is the case  $(c,v,u)=(4,7,2)$, corresponding to the sides $(a,b,c,d)=(5,3,4,6)$.

Assume $\tau=5$. By \eqref{E:utau}, we have  $-2c<u\le 2c$. So 
$100 + u^2=v^2-(v-2c)^2$ gives
$100 +4c^2\ge  100 + u^2=4vc-4c^2$,
from which it follows that 
\begin{equation}\label{E:dbd}
25 \ge c(v-2c).
\end{equation}
By \eqref{E:conda}, we have $v> 2c$, so $c,v-2c\in\{1,\dots,25\}$.  
One finds there are just 86 pairs  $c, v$ with $v> 2c$ for which \eqref{E:dbd} holds.  Of these,  one finds there is only one where the equation $100 + u^2=v^2-(v-2c)^2$ has an integer solution $u$ for which $(v+u)/5$ is an integer, and such that for the resulting side lengths $(a,b,c,d)$, one has $b=\min\{a,b,c,d\}$; this is the case  $(c,v,u)=(5,15,10)$, corresponding to the sides $(a,b,c,d)=(5,5,5,5)$.

Assume $\tau=9$. By \eqref{E:utau}, we have    $-4c<u\le 4c$. So 
$324 + u^2=v^2-(v-4c)^2$ gives
$324 +16c^2\ge  324 + u^2=8vc-16c^2$,
from which it follows that 
\begin{equation}\label{E:fbd}
41 \ge c(v-4c).
\end{equation}
By \eqref{E:conda}, we have $v> 4c$, so $c,v-4c\in\{1,\dots,41\}$.  
One finds there are 979 pairs  $c, v$ with $v> 4c$ for which \eqref{E:fbd} holds.  Of these,  one finds there is only two where the equation $324 + u^2=v^2-(v-4c)^2$ has an integer solution $u$ for which $(v+u)/9$ is an integer, and such that for the resulting side lengths $(a,b,c,d)$, one has $b=\min\{a,b,c,d\}$; these are the cases  $(c,v,u)=(3,21,6)$ and $(5,23,1)$, corresponding respectively  to the sides $(a,b,c,d)=(15,3,3,15)$ and $(37,1,5,41)$.

This completes the proof of the corollary.
\end{proof}

\section{Proof of Theorem \ref{T:conversegen}}\label{S:converse}

We follow the general strategy used in \cite{Yiu}, but in our case we employ a slightly different solution form for the diophantine equations that appear in the statement of Theorem~\ref{T:conversegen}.

\begin{lemma}\label{L:mordell}
Suppose $z^2 + w^2+u^2=v^2$ for integers $u,v,w,z$ and that the prime decomposition of  $\gcd(u,v,w,z)$   contains no term $\rho^k$ where $\rho$ is congruent to $3$ modulo 4 and $k$ is odd.
Then there are  integers $p,q,m,n$  such that 
\[
v-u=p^2+q^2,\quad v+u=m^2+n^2,\quad w=pm+qn,\quad z=pn-qm.
\]
\end{lemma}

Numbers $u,v,w,z$ for which $z^2 + w^2+u^2=v^2$ are said to form a \emph{Pythagorean quadruple}, and of course their study has a long history; see \cite{RGM}. 
The above lemma is essentially equivalent to a classical result which says that if 
 $z^2 + w^2+u^2=v^2$ for integers $u,v,w,z$ with $\gcd(u,v,w,z)=1$, then supposing $z,w$ are even,  there are  integers $p,q,m,n$  such that 
\[
v-u=2(p^2+q^2),\quad v+u=2(m^2+n^2),\quad w=2(pm+qn),\quad z=2(pn-qm).
\]
This result, sometimes attributed to V. A. Lebesgue, is proved  in many places;  see \cite[pp. 28-37]{Ca}, \cite{Du}, \cite[p.14]{Mordell}  and \cite{Sp}.
We require the slightly stronger formulation of Lemma~\ref{L:mordell}, which is readily deduced from the treatment given in \cite[Section~II]{Di}.

We will also make use of a certain elementary fact which we give in the following lemma. For convenience, let us make a definition.

\begin{definition}
We say that a positive integer $k$ has the \emph{lattice preservation property}, or is a  \emph{lattice preserver}, if for every lattice point $X$ for which $\frac1{k}X$ has integer length, the point $\frac1{k}X$ is also a lattice point.
\end{definition}

For example, it is easy to see that $2$ and $3$ are lattice preservers. Notice that the set of lattice preservers is closed under multiplication. 
Hence, for example, $4$ and $6$ are lattice preservers.
Recall that a \emph{hypotenuse number} is a positive integer that occurs as the length of the hypotenuse of some Pythagorean triangle. It is well known that hypotenuse numbers are those numbers that have a prime factor congruent to 1 modulo 4 \cite{NZM}.

\begin{lemma}\label{L:trick}  A positive integer $k$ is a lattice preserver if and only if $k$ is not a hypotenuse number. So $k$ is a lattice preserver if and only if $k$ has no prime factor congruent to 1 modulo 4.
\end{lemma}

\begin{proof} If $k$ is a hypotenuse number, say $k^2=x^2+y^2$, then $\frac1k(x,y)$ has length 1 but it not a lattice point. So hypotenuse numbers are not lattice preservers. Conversely, if $k$ is not a lattice preserver (so $k>2$), then there exists a lattice point $(x,y)$ such that $(x,y)/k$ has integer length, $a$ say, but is not a lattice point.
We may assume without loss of generality  that $\gcd(x,y,k)=1$.  We have $x^2+y^2=k^2a^2$. Write $x'^2+y'^2=k^2a'^2$ where $x'=x/\gcd(x,y,a)$, etc. So $\gcd(x',y',ka')=1$ and hence $x',y',ka'$ is a primitive Pythagorean triple. So by \cite[Theorem~3.20]{NZM} for example,  all the odd prime factors of $k$ are congruent to 1 mod 4 and $k$ is not divisible by 4. So as $k>2$, we conclude that $k$ has at least one prime factor congruent to 1 mod 4 and so $k$ is a hypotenuse number.
\end{proof}

Recall that by Remark~\ref{R:choice}, when working with tangential LEQs we may suppose without loss of generality that $c$ is even when $\sigma=\tau=2$ and that
$c$ is not divisible by 3 when $\sigma$ or $\tau$ equals 3.

\begin{proof}[Proof of Theorem \ref{T:conversegen}](I). Suppose $t\in\{2,3,5,9\}$.  Notice that $a+c=b+d$ and so from hypothesis (i), $(d-a) t>-(a+c) $. Adding hypothesis (ii) gives $2dt>0$, so $d>0$. Furthermore,
\eqref{E:gen} gives  $(v+u)(v-u)=v^2-u^2>0$, so as $v=(a+tb)/2>c(t-1)/2>0$ by condition (i) of our hypotheses, $v+u$ and $v-u$ are both necessarily positive. That is, $a,b>0$.
So, in all cases, $a,b,c,d$ are all positive.

The basic idea of the proof is to apply Lemma \ref{L:mordell}
to obtain integers $p,q,m,n$  such that 
\begin{equation}\label{E:genident}
a=p^2+q^2,\ tb=m^2+n^2, \ pm+qn=v-\frac{t-1}2c, \ pn-qm=-2t.
\end{equation}
Then we consider the Gaussian integers
$z:=p+qi,\ w:=m+ni$,
and let
\begin{equation}\label{E:gendefs}
A=z^2,\qquad
B=z^2-\frac1tw^2,\qquad
C=\frac1{t(t-1)} (tz-w)^2.
\end{equation}
We call this the \emph{general case}. Unfortunately, as we will see below, this procedure is not always possible, and we will require two variations on this approach. 

(Let us explain, in parenthesis, how the proposal of vertices of \eqref{E:gendefs} can be understood. Obviously, $A,B$ are suggested by \eqref{E:genident}. For a tangential LEQ, the areas $K_O,K_A$ are determined by $\sigma,\tau$ and the side lengths, by Remark~\ref{R:integers}. Then 
\eqref{E:college} enables one to express $C$ in terms of $A$ and $B$. This gives a formula for $C$ that must hold if this construction is to produce a tangential LEQ. We suppress this derivation of the formula for $C$, and focus on showing that it has the required properties).

First suppose that $t=2$. Then \eqref{E:gen} is
$16 + u^2=v^2-\left(v-\frac{1}2c\right)^2$.
Clearly $\gcd(4,u,v,c)$ is either 1, 2 or 4, so we may apply Lemma \ref{L:mordell},
and obtain the general case of \eqref{E:genident} and \eqref{E:gendefs}.

Now suppose that $t=3$. Then \eqref{E:gen} is
$6^2 + u^2=v^2-\left(v-c\right)^2$.
As $c$ is not divisible by $3$ by assumption, $\gcd(6,u,v,c)$ is $1$ or $2$, and we may again apply Lemma \ref{L:mordell} and obtain the general case of \eqref{E:genident} and \eqref{E:gendefs}.


Now suppose that $t=5$. Then \eqref{E:gen} is
$10^2 + u^2=v^2-\left(v-2c\right)^2$,
and as  $\gcd(10,u,v,c)$ is $1,2,5$ or $10$,  we could apply Lemma \ref{L:mordell} in all cases. In fact, for reasons that will become apparent later in the proof, we will directly apply Lemma \ref{L:mordell}, and obtain the
general case, only in the cases where  $u,v,c$ are not all divisible by 5, so $\gcd(10,u,v,c)=1$ or $2$. Note that  $u,v,c$ are all divisible by 5 precisely when $a$ and $c$ are divisible by 5. In this case, let $\frac{u}5 =u',\frac{v}5 =v',\frac{c}5 =c'$. 
 Thus \eqref{E:gen} can be written as 
$4 + u'^2=v'^2-\left(v'-2c'\right)^2$,
 and applying Lemma \ref{L:mordell},
we have integers $p,q,m,n$  such that 
\begin{equation}
a=5(p^2+q^2),\ b=m^2+n^2, \ pm+qn=\frac15(v-2c), \ pn-qm=-2.
\end{equation}
Then let
$z:=p+qi,\ w:=m+ni$,
and set
\begin{equation}\label{E:gendefs2}
A=5z^2,\qquad
B=5z^2-w^2,\qquad
C=\frac1{4} (5z-w)^2.
\end{equation}
We call this the \emph{first exceptional case}.


Now suppose that $t=9$. Then \eqref{E:gen} is
$18^2 + u^2=v^2-\left(v-4c\right)^2$.
If  $\gcd(18,u,v,c)$ is not $3$ or $6$, we may apply Lemma \ref{L:mordell} and obtain the general case.
If  instead $\gcd(18,u,v,c)$ is $3$ or $6$, which occurs when $\gcd(a,c)$ is divisible by 3 but not 9, let $\frac{u}3 =u',\frac{v}3 =v',\frac{c}3 =c'$. 
 Thus \eqref{E:gen} can be written as 
$36 + u'^2=v'^2-\left(v'-4c'\right)^2$,
 and applying Lemma \ref{L:mordell},
we have integers $p,q,m,n$  such that 
\begin{equation}
a=3(p^2+q^2),\ 3b=m^2+n^2, \ pm+qn=\frac13(v-4c), \ pn-qm=-6.
\end{equation}
Then let
$z:=p+qi,\ w:=m+ni$,
and set
\begin{equation}\label{E:gendefs3}
A=3z^2,\qquad
B=3z^2-\frac13w^2,\qquad
C=\frac1{24} (9z-w)^2.
\end{equation}
We call this the \emph{second exceptional case}.

We now proceed to show that the points $O,A,B,C$ define a tangential LEQ $OABC$ with successive side lengths $a,b,c,d$ for which $(\sigma,\tau)=(\frac{t}{t-1},t)$. We first treat the general case of \eqref{E:genident} and \eqref{E:gendefs}, and deal with the two exceptional cases later. So we are assuming that for $t=5$, the integers $u,v,c$ are not all divisible by 5, and for $t=9$, we have that $\gcd(18,u,v,c)$ is not $3$ or $6$.

Note that from \eqref{E:genident}, $OA$ has length $p^2+q^2=a$ and $A-B=\frac1tw^2$, which has length $\frac1t(m^2+n^2)=b$. It remains to see that: 
\begin{enumerate}
\item $C-B$ has length $c$, and $OC$ has length $d$, 
\item the quadrilateral $OABC$ has no self-intersections,
\item $OABC$ is equable,
\item the points $A,B,C$ are not colinear,
\item  $B$ is the only point at which the angle may be reflex,
\item for $OABC$, one has $\tau=t$,
\item $B$ and $C$ are lattice points.
\end{enumerate}

\smallskip
Let us make some preliminary calculations. Substituting $z=p+qi,w=m+ni$ and using \eqref{E:genident}, one has 
\begin{align}
z\bar w -\bar z w&=2  (q m - p n)i=4ti,\label{E:crossgen}\\
z\bar w +\bar z w&=2 (p m + q n)=2v-(t-1)c=a+c+t(b-c).\label{E:dotgen}
\end{align}
Consequently,
\begin{equation}\label{E:cross2gen}
z^2\bar w^2-\bar z^2w^2= (z\bar w -\bar z w)(z\bar w +\bar z w)=4t(a+c+t(b-c))i.
\end{equation}
Now consider  the  signed areas $K_O,K_A,K_C,K_B$.
Recall that if $Z,W$ are points in the complex plane, the triangle $OZW$ has signed area $i(Z\bar W-\bar ZW)/4$.
Using  \eqref{E:cross2gen}, we have 
$4tK_A= i(-z^2\bar w^2 + \bar z^2w^2)=4t(a+c+t(b-c))$,
so
\begin{equation}\label{E:KAgen}
K_A=a+tb-(t-1)c.
\end{equation}
Using $z\bar z=a,w\bar w=tb$ and \eqref{E:crossgen}, \eqref{E:dotgen}, one has
\begin{align*}
4(t-1)t^2K_C&=i((tz^2-w^2)  (\overline{tz-w})^2-(tz-w)^2\overline{(tz^2-w^2)} )\\
&=i t (z\bar w - \bar z w) (-2 t(a+b) + (t+1)(z\bar w + \bar z w) )\\
&=-4 t^2 (-2 t(a+b) + (t+1)(a+c+t(b-c) )\\
&=4 t^2(t-1)(a-tb+(t+1)c),
\end{align*}
so
\begin{equation}\label{E:KCgen}
K_C =a-tb+(t+1)c.
\end{equation}
Using $z\bar z=a$ and \eqref{E:crossgen}, \eqref{E:dotgen}, one has
\begin{align*}
4t(t-1)K_O&=i(z^2  (\overline{tz-w})^2-(tz-w)^2\bar z^2)=i (z\bar w - \bar z w) (z\bar w + \bar z w - 2 t a)\\
&=-4 t (a+c+t(b-c) -2 ta))=4 t((2t-1)a-tb+(t-1)c),
\end{align*}
so
\begin{equation}\label{E:KOgen}
K_O=\frac1{t-1} ((2t-1)a-tb+(t-1)c).
\end{equation}
Before calculating $K_B$, note that
$t(t-1)(C-B)=
(tz-w)^2-(t-1)(tz^2-w^2) =t(z-w)^2$,
so
\begin{equation}\label{E:C-Bgen}
C-B=\frac1{t-1}(z-w)^2.
\end{equation}
Thus, using $w\bar w=tb$ and \eqref{E:crossgen}, \eqref{E:dotgen}, one has
\begin{align*}
4t(t-1)K_B&=i((z-w)^2 \bar w^2-\overline{(z-w)^2 }w^2)=i  (- z\bar w + \bar z w ) (2t b - z\bar w -\bar z w )\\
&=4 t (2 tb -( a+c+t(b-c)))=4 t(-a+tb+(t-1)c),
\end{align*}
so
\begin{equation}\label{E:KBgen}
K_B=\frac1{t-1} (-a+tb+(t-1)c).
\end{equation}

\smallskip
We now prove the  requirements (a) -- (g).

(a). From \eqref{E:C-Bgen}, we have, using  \eqref{E:dotgen},
\begin{align*}
(t-1)\Vert C-B\Vert&=\Vert z-w\Vert^2=z\bar z +w\bar w-z\bar w-\bar z w\\
&=a+tb-(a+c+t(b-c))=(t-1)c,
\end{align*}
 as required.
Using \eqref{E:dotgen} again, we also have 
\begin{align*}
t(t-1)&\Vert C\Vert=\Vert tz-w\Vert^2=t^2z\bar z +w\bar w-tz\bar w-t\bar z w=t^2a+tb-t(z\bar w+\bar z w)\\
&=t^2a+tb-t(a+tb-(t-1)c)\\
&=(t^2-t)a-(t^2-t)b+t(t-1)c=t(t-1)d,
\end{align*}
 as required.

(b).  To verify that the quadrilateral $OABC$ has no self-intersections, it suffices to show that the respective signed areas $K_A,K_C$ of triangles $OAB,OBC$ are both positive.
The hypothesis (i) gives $tb+a-(t-1)c>0$, so $K_A>0$ by \eqref{E:KAgen}.
Hypothesis (i) also gives $a-tb+(t+1)c>0$, so $K_C>0$ by \eqref{E:KCgen}.  

(c). From \eqref{E:KAgen},\eqref{E:KCgen},
we have $K_A+K_C= 2(a+c)$, as required.

(d). To verify that the points $A,B,C$ are not colinear, it suffices to show that $K_B\not=0$. But by \eqref{E:KBgen},
$(t-1)K_B=-a+tb+(t-1)c)\not=0$, by  hypothesis (iii).

(e). To see that $B$ is the only point at which the angle may be reflex, it remains to show that $K_O>0$.
The hypothesis (ii) gives $(2a+c-b)t >a+c$, from which we have $(2t-1)a-tb+(t-1)c>0$, so $K_O>0$ by \eqref{E:KOgen}.

(f).  If $a\not=d$, then \eqref{E:KA},\eqref{E:KAgen} give
$\tau= \frac{ K_A - (a+c)}{a-d}
= \frac{ a+tb-(t-1)c-a-c}{a-d}
= \frac{ t(b-c)}{a-d}=t$.
Thus by Lemma \ref{L:seven}, $\sigma=\frac{t}{t-1}$.

 If $a\not=b$, then \eqref{E:KO},\eqref{E:KOgen} give
 $\sigma= \frac{K_O- (a+c)}{a-b}
= \frac{ (2t-1)a-tb+(t-1)c-(t-1)(a+c)}{(t-1)(a-b)}
= \frac{ ta-tb}{(t-1)(a-b)}
=\frac{t}{t-1}$.
Thus by Lemma \ref{L:seven}, $\tau=t$.

Finally, if $a=d$ and $a=b$, then $OABC$ is a rhombus and $a=b=c=d$. But if $a=b$, then $u=\frac{t-1}2a,v=\frac{t+1}2a$, and so by \eqref{E:gen}, $a=c$ would give 
$(2t)^2 + \frac{(t-1)^2}4a^2=\frac{(t+1)^2}4a^2-a^2$, so $a^2=\frac{4t^2}{t-1}$. For $t=\frac98,\frac54,\frac32,2,3,5,9$, this would give respectively
$a^2= \frac{81}2,25,18,16,18,25$, which is impossible for $t=\frac98,\frac32,3,9$. For $t=2$ we have $a=4$, so $OABC$ is the $4\times4$ square, which has $(\sigma,\tau)=(2,2)$, by Remark~\ref{R:integers}. For $t=5$ and $\frac54$, we have $a=5$, so $OABC$ is the rhombus of side length 5, which has $(\sigma,\tau)=(\frac54,5)$ and $(5,\frac54)$ respectively, again by Remark~\ref{R:integers}.

(g). We now come to the most delicate part of the proof. Note that parts (a)-(e) were simply equations or inequalities, and did not use the values of $t$, or the fact that certain variables are integers. Part (f) did use these facts, but only in a very simple manner.  

First suppose $t=2$. So $B=z^2-\frac12w^2$ 
and $C=\frac1{2} (2z-w)^2$. Now $z^2,w^2$ are lattice points. And from above, $\frac12w^2$ has integer length $b$. So by Lemma~\ref{L:trick}, $\frac12w^2$ is a  lattice point. Thus $B$ is a lattice point. Similarly, $(2z-w)^2$ is a lattice point and  $C=\frac1{2} (2z-w)^2$ has integer length $c$, so by Lemma~\ref{L:trick},  $C$ is a lattice point.

Now suppose $t=3$. So $B=z^2-\frac13w^2$ 
and $C=\frac1{6} (3z-w)^2$. Now $z^2,w^2$ are lattice points. And from above, $\frac13w^2$ has integer length $b$. So by Lemma~\ref{L:trick}, $\frac13w^2$ is a  lattice point. Thus $B$ is a lattice point. Similarly, $(3z-w)^2$ is a lattice point and  $C=\frac1{6} (3z-w)^2$ has integer length $c$, so by Lemma~\ref{L:trick},  $C$ is a lattice point.

Now suppose $t=5$. So $B=z^2-\frac15w^2$ 
and $C=\frac1{20} (5z-w)^2$, where $w=m+ni$. We claim that, in the general case, $m,n$ are multiples of 5. 
First note that \eqref{E:gen} can be written as 
$10^2+(v-2c)^2=(v+u)(v-u)$.
So as $5$ divides $v+u=5b$, it follows that  $v-2c\equiv 0\pmod5$. 
Hence from \eqref{E:genident}, 
\[
ma=m(p^2+q^2)=p(pm+qn)-q(pn-qm)=p(v-2c)+10q\equiv 0\pmod5.
\]
Similarly, $na\equiv 0\pmod5$. So if $a\not\equiv 0\pmod5$, we have $m,n\equiv 0$ as required. 
If $a\equiv 0\pmod5$, then as
$v-2c=\frac{5b+a-4c}2=\frac{5(b-c)+a+c}2$,
 and $v-2c\equiv 0\pmod5$, so $5$ divides $a+c$, and thus $c\equiv 0\pmod5$ and hence $v\equiv 0\pmod5$. So, as $v+u=5b$, we have that $5$ divides $u,v,c$. But this is the first exceptional case, contrary to our current assumption.
 
As $m,n$ are multiples of 5, let  $m=5m',n=5n', w'=m'+n'i$. So $B=z^2-5w'^2$, which is obviously a lattice point,  
and $C=\frac5{4} (z-w')^2$, which is a lattice point by Lemma~\ref{L:trick}.

Now suppose $t=9$. So $B=z^2-\frac19w^2$ 
and $C=\frac1{72} (9z-w)^2$. Now $z^2,w^2$ are lattice points. And from above, $\frac19w^2$ has integer length $b$. So by Lemma~\ref{L:trick}, $\frac19w^2$ is a  lattice point. Thus $B$ is a lattice point. Similarly, $(9z-w)^2$ is a lattice point and  $C=\frac1{72} (3z-w)^2$ has integer length $d$. Hence, as $72=2^33^2$ is a lattice preserver,  $C$ is a lattice point by Lemma~\ref{L:trick}. This completes part (g).

We now treat the first exceptional case. So
$t=5$ and $u,v,c$ are all divisible by 5. In the preliminary calculations part of the argument, analogous to \eqref{E:crossgen}, \eqref{E:dotgen}  and \eqref{E:cross2gen}, one has:
\begin{align}
z\bar w -\bar z w&=4i,\label{E:crossgen1}\\
z\bar w +\bar z w&=\frac15(a+c)+(b-c),\label{E:dotgen1}\\
z^2\bar w^2-\bar z^2w^2&=(\frac45(a+c)+4(b-c))i.\label{E:cross2gen1}
\end{align}
For the areas, using the above three expressions and $z\bar z=a/5, w\bar w=b$,  we find exactly the same formulas for $K_O,K_A,K_B,K_C$ as before; that is, we obtain  \eqref{E:KOgen},\eqref{E:KAgen},\eqref{E:KBgen},\eqref{E:KCgen} respectively with $t=5$.

Analogous to \eqref{E:C-Bgen}, one has
 \begin{equation}\label{E:C-Bgen1}
C-B=\frac54(z-w)^2.
\end{equation}
For the proof of part (a), we have from \eqref{E:C-Bgen1}, using  \eqref{E:dotgen1},
\begin{align*}
4\Vert C-B\Vert&=5\Vert z-w\Vert^2=5(z\bar z +w\bar w-z\bar w-\bar z w)\\
&=a+5b-(a+c+5(b-c))
=4c,
\end{align*}
 as required.
Using \eqref{E:dotgen1} again, we also have 
\begin{align*}
4&\Vert C\Vert=\Vert 5z-w\Vert^2=5^2z\bar z +w\bar w-5z\bar w-5\bar z w=5a+b-5(z\bar w+\bar z w)\\
&=5a+b-((a+c)+5(b-c))
=4a-4b+4c=4d,
\end{align*}
 as required.

As parts (b)-(f) only rely on the expressions for $K_O,K_A,K_B,K_C$, and as these are unchanged, the proofs of these parts need no amendment. It remains to verify part (g). But $B=5z^2-w^2$, which is obviously a lattice point,  
and $C=\frac1{4} (5z-w)^2$, which is a lattice point by Lemma~\ref{L:trick}.

Finally, we treat the second exceptional case. So
 $t=9$ and $\gcd(18,u,v,c)$ is  $3$ or $6$.
In the preliminary calculations part of the argument, analogous to \eqref{E:crossgen}, \eqref{E:dotgen}  and \eqref{E:cross2gen}, one has:
\begin{align}
z\bar w -\bar z w&=12i,\label{E:crossgen2}\\
z\bar w +\bar z w&
=\frac13(a+9b-8c),\label{E:dotgen2}\\
z^2\bar w^2-\bar z^2w^2&=4(a+9b-8c)i.\label{E:cross2gen2}
\end{align}
For the areas,  using the above three expressions and $z\bar z=a/3, w\bar w=3b$,  we find the exactly same formulas for $K_O,K_A,K_B,K_C$ as before; that is, we obtain  \eqref{E:KOgen},\eqref{E:KAgen},\eqref{E:KBgen},\eqref{E:KCgen} respectively with $t=9$.

Analogous to \eqref{E:C-Bgen}, one has
 \begin{equation}\label{E:C-Bgen2}
C-B=\frac3{8} (z-w)^2.
\end{equation}
For the proof of part (a), we have from \eqref{E:C-Bgen2}, using  \eqref{E:dotgen2},
\begin{align*}
8\Vert C-B\Vert&=3\Vert z-w\Vert^2=3(z\bar z +w\bar w-z\bar w-\bar z w)\\
&=3(a/3+3b-(a+9b-8c)/3)=a+9b-(a+9b-8c)=8c,
\end{align*}
 as required.
Using \eqref{E:dotgen2} again, we also have 
\begin{align*}
24&\Vert C\Vert=\Vert 9z-w\Vert^2=9^2z\bar z +w\bar w-9z\bar w-9\bar z w=27a+3b-9(z\bar w+\bar z w)\\
&=27a+3b-3(a+9b-8c)
=24a-24b+24c=24d,
\end{align*}
 as required.

As parts (b)-(f) only rely on the expressions for $K_O,K_A,K_B,K_C$, and as these are unchanged, the proofs of these parts need no amendment. It remains to verify part (g). Now $A,z,w$ are lattice points. Thus, as $A-B=\frac13w^2$ is a lattice point by Lemma~\ref{L:trick}, so $B$ is a lattice point.  
Finally, $C=\frac1{24} (9z-w)^2$ is a lattice point by Lemma~\ref{L:trick} since $24=2^3 3$ is a lattice preserver.

This completes the proof of part (I).

\smallskip
(II). Suppose $s\in\{3,5,9\}$. 
By hypothesis, $b$ and $c$ are positive. Further, Equation \eqref{E:gen} gives  $(v+u)(v-u)=v^2-u^2>0$, so as $v$ is positive by condition (i) of our hypotheses, $v+u$ and $v-u$ are both necessarily positive. That is, $a,d>0$.
So, in all cases, $a,b,c,d$ are all positive.

We now proceed exactly as we did in case (I) by considering a general case and two exceptional cases. 
The first exceptional case is where $s=5$ and the integers $u,v,c$ are all divisible by 5. The second exceptional case is where $s=9$ and $\gcd(18,u,v,c)$ is either $3$ or $6$. In the general case, we apply Lemma \ref{L:mordell}
to obtain integers $p,q,m,n$  such that 
\begin{equation}\label{E:genidentII}
a=p^2+q^2,\ sd=m^2+n^2, \ pm+qn=v-\frac{s-1}2c, \ pn-qm=2s.
\end{equation}
Then we consider the Gaussian integers
$z:=p+qi,\ y:=m+ni$,
and let
\begin{equation}\label{E:gendefsII}
A=z^2,\qquad
B=z^2-\frac{1}{s(s-1)}(sz-y)^2,\qquad
C=\frac1{s} y^2.
\end{equation}
Note that $A$ is a lattice point, and $OA$ has length $a$, while $OC$ has length $d$ (we will see later that $C$ is a lattice point).
For $t=\frac{s}{s-1}$, let 
\[
w:=\frac{sz-y}{s-1}. 
\]
Notice that $w$ has length (i.e., norm) given by
\begin{align*}
(s-1)^2\Vert w\vert|^2&=s^2a+sd-s(z\bar y+\bar z y)=s^2a+sd-s(2v-(s-1)c)\\
&=s^2a+sd-s(sd+a-(s-1)c)=
s(s-1)(a-d+c)=
s(s-1)b .
\end{align*}
So $\frac{s-1}{s}w^2$ has length $b$. Observe that for $t=\frac{s}{s-1}$, we have 
$B=z^2-\frac1tw^2$ and $A-B=\frac1tw^2$ has length $b$, exactly as in case (I). Moreover, we can write
\[
C=\frac1{s} (sz-(s-1)w)^2=\frac{(s-1)^2}{s} \left(\frac{s}{s-1}z-w\right)^2=\frac{1}{t(t-1)} (tz-w)^2,
\]
which is the same formula as appeared in \eqref{E:gendefs}, in case (I). Now compute:
\begin{align*}
z\bar w -\bar z w&=\frac{-1}{s-1}  (z\bar y -\bar z y)= \frac{-1}{s-1} 2 (q m - p n)i=\frac{s}{s-1} 4i=4ti,\\
z\bar w +\bar z w&=\frac{1}{s-1}(z\overline{(sz-y)} +\bar z (sz-y))=\frac{1}{s-1}(2sa -z\bar y-\bar z y)\\
&=\frac{1}{s-1}(2sa -(2v-(s-1)c))=\frac{1}{s-1}(2sa -sd-a+(s-1)c)\\
&=a+c+\frac{s}{s-1}(a-d)=a+c+t(a-d)=a+c+t(b-c),\\
z^2\bar w^2-\bar z^2w^2&= (z\bar w -\bar z w)(z\bar w +\bar z w)=4t(a+c+t(b-c))i,
\end{align*}
which are exactly as in Equations \eqref{E:crossgen}, \eqref{E:dotgen}, \eqref{E:dotgen} of  case (I). 
It follows that the areas $K_O,K_A,K_B,K_C$ are given by the case (I)  formulas  \eqref{E:KOgen},\eqref{E:KAgen},\eqref{E:KBgen},\eqref{E:KCgen}.
Thus the proof of parts (a)-(f) hold by the arguments used in case (I) and it remains to verify part (g). 

Let us first suppose $s=5$. So $C=\frac1{5} y^2$, where $w=m+ni$. We claim that $m,n$ are multiples of 5. 
First note that \eqref{E:gen} can be written as 
$10^2+(v-2c)^2=(v+u)(v-u)$.
So as $5$ divides $v+u=5d$, it follows that  $v-2c\equiv 0\pmod5$. 
Hence from \eqref{E:genident}, 
\[
ma=m(p^2+q^2)=p(pm+qn)-q(pn-qm)=p(v-2c)-10q\equiv 0\pmod5.
\]
Similarly, $na\equiv 0\pmod5$. So if $a\not\equiv 0\pmod5$, we have $m,n\equiv 0$ as required. 
If $a\equiv 0\pmod5$, then as
$v-2c=\frac{5d+a-4c}2=\frac{5(d-c)+a+c}2$,
 and $v-2c\equiv 0\pmod5$, so $5$ divides $a+c$, and thus $c\equiv 0\pmod5$ and hence $v\equiv 0\pmod5$. So, as $v+u=5d$, we have that $5$ divides $u,v,c$. But this is the first exceptional case, contrary to our current assumption.
 
As $m,n$ are multiples of 5, let  $m=5m',n=5n', y'=m'+n'i$. So $C=5y'^2$, which is obviously a lattice point.  
Furthermore, for $s=5$, one has from \eqref{E:gendefsII} that
\[
A-B=\frac1{20}5(z-y')^2=\frac1{4}(z-y')^2.
\]
So as $A-B$ has integer length $b$, and $(z-y')^2$ is a lattice, so $A-B$ is a lattice point by Lemma~\ref{L:trick}, and hence $B$ is a lattice point.

Now suppose $s=3$ or $9$.
As $y$ is a  lattice point and  $C=\frac1s y^2$ has length $d$, so $C$  is a lattice point when $s=3$ or $9$, by Lemma~\ref{L:trick}.
For $s=3$, we have $A-B=\frac1{6}(3z-y)^2$ and for $s=9$, we have $A-B=\frac1{72}(9z-y)^2$. In both cases, $A-B$ is a lattice point by Lemma~\ref{L:trick}, and hence $B$ is a lattice point.

We now treat the first exceptional case. So
$t=5$ and $u,v,c$ are all divisible by 5.  Let $\frac{u}5 =u',\frac{v}5 =v',\frac{c}5 =c'$. 
 Thus \eqref{E:gen} can be written as 
$4 + u'^2=v'^2-\left(v'-2c'\right)^2$,
 and applying Lemma \ref{L:mordell},
we have integers $p,q,m,n$  such that 
\begin{equation}\label{E:genident2}
a=5(p^2+q^2),\ d=m^2+n^2, \ pm+qn=\frac15(v-2c), \ pn-qm=2.
\end{equation}
Then let
$z:=p+qi,\ y:=m+ni$,
and set
\begin{equation}\label{E:gendefsII2}
A=5z^2,\qquad
B=5z^2 - \frac14(5z-y)^2,\qquad
C= y^2.
\end{equation}
We remark that $B$ has other useful expressions:
\[
B=y^2 - \frac54(z-y)^2=\frac14(-y^2 + 10 y z - 5 z^2).
\]
Note that $A$ has length $a$ and $C$ has length $d$.
Moreover, $A-B =\frac14(5z-y)^2$ and has length
\begin{align*}
\frac14(5z-y)\overline{(5z-y)}&=\frac14(5a+d-5(z\bar y +\bar z y))=\frac14(5a+d- (5d+a-4c))\\
&=\frac14(4a-4d+4c)=b,
\end{align*}
and $C-B=\frac54(z-y)^2$ has length
\begin{align*}
\frac54(z-y)\overline{(z-y)}&=\frac54(\frac15a+d-(z\bar y +\bar z y))=\frac14(a+5d- (5d+a-4c))=c.
\end{align*}
So part (a) of the requirements (a)-(g) is satisfied. Moreover, $A=5z^2$ and $C=y^2$ are obviously lattice points, and $-y^2 + 10 y z - 5 z^2$ is a lattice point
and  $B=\frac14(-y^2 + 10 y z - 5 z^2)$ has integer length $b$, so $B$ is a  lattice point by Lemma~\ref{L:trick}. So part (g) of the requirements (a)-(g) is satisfied. 

For our preliminary calculations we will use $z$ and $y$, rather than $z$ and $w$ as we did before. We have
\begin{align*}
z\bar y -\bar z y&=-2(pn-qm)i=-4i,\\
z\bar y +\bar z y&=\frac15(5d+a-4c)=\frac15(6a-5b+c),\\
z^2\bar y^2-\bar z^2y^2&=\frac{-4}5(6a-5b+c)i.
\end{align*}
These relations are also different to the Equations \eqref{E:crossgen1}, \eqref{E:dotgen1}, \eqref{E:cross2gen1} we obtained in the first exceptional case of case (I).
Nevertheless,  using the above three relations together with \eqref{E:gendefsII2}, we find exactly the same formulas for $K_O,K_A,K_B,K_C$ as before; that is, we obtain  \eqref{E:KOgen},\eqref{E:KAgen},\eqref{E:KBgen},\eqref{E:KCgen} respectively with $t=\frac54$. As parts (b)-(f) only rely on the expressions for $K_O,K_A,K_B,K_C$, and as these are unchanged, the proofs of these parts need no amendment.

Finally, we treat the second exceptional case. So
 $s=9$ and $\gcd(18,u,v,c)$ is  $3$ or $6$.
Let $\frac{u}3 =u',\frac{v}3 =v',\frac{c}3 =c'$. 
 Thus \eqref{E:gen} can be written as 
\[
36 + u'^2=v'^2-\left(v'-4c'\right)^2,
\]
 and applying Lemma \ref{L:mordell},
we have integers $p,q,m,n$  such that 
\begin{equation}\label{E:genident3}
a=3(p^2+q^2),\ 3d=m^2+n^2, \ pm+qn=\frac13(v-4c), \ pn-qm=6.
\end{equation}
Then let
$z:=p+qi,\ y:=m+ni$,
and set
\begin{equation}\label{E:gendefsII3}
A=3z^2,\qquad
B=3z^2 - \frac1{24}(9z-y)^2,\qquad
C= \frac13y^2.
\end{equation}
Observe that
\[
B=\frac13y^2 - \frac38(z-y)^2=\frac1{24}(-y^2 + 18 y z - 9 z^2).
\]
Note that $A$ has length $a$ and $C$ has length $d$. 
Moreover, $A-B=\frac1{24}(9z-y)^2$ and has length
\begin{align*}
\frac1{24}(9z-y)\overline{(9z-y)}&=\frac1{24}(27a+3d-9(z\bar y +\bar z y))=\frac1{8}(9a+d- (9d+a-8c))\\
&=\frac18(8a-8d+8c)=b,
\end{align*}
and $C-B=\frac38(z-y)^2$ has length
\begin{align*}
\frac38(z-y)\overline{(z-y)}&=\frac38(\frac13a+3d-(z\bar y +\bar z y))=\frac18(a+9d- (9d+a-8c))=c.
\end{align*}
So part (a) of the requirements (a)-(g) is satisfied. Obviously $A$ is a lattice point, and $B,C$ are  lattice points by Lemma~\ref{L:trick}. So part (g) of the requirements (a)-(g) is satisfied. 

We have
\begin{align*}
z\bar y -\bar z y&=-2(pn-qm)i=-12i,\\
z\bar y +\bar z y&=\frac13(9d+a-8c)=\frac13(10a-9b+c),\\
z^2\bar y^2-\bar z^2y^2&=-4(10a-9b+c)i.
\end{align*}
Using these relations  together with \eqref{E:gendefsII2}, we find exactly the same formulas for $K_O,K_A,K_B,K_C$ as before; that is, we obtain  \eqref{E:KOgen},\eqref{E:KAgen},\eqref{E:KBgen},\eqref{E:KCgen} respectively with $t=\frac98$. As parts (b)-(f) only rely on the expressions for $K_O,K_A,K_B,K_C$, and as these are unchanged, the proofs of these parts need no amendment.

This completes the proof of part (II), and the theorem.
\end{proof}


\begin{proof}[Proof of Corollary \ref{C:incenter}] We use the terminology and results from the proof of Theorem~\ref{T:conversegen}. 
By Remark \ref{R:stlam}, $\lambda=\frac1{\sigma}$, so
\[
I=\frac{(A+C)+(\sigma-1)B}{2\sigma}=\frac{(\tau-1)(A+C)+B}{2\tau}.
\]
First suppose $\tau=2$. Then $I=\frac{A+C+B}{4}$, so by \eqref{E:gendefs},
\[
I=\frac{z^2+z^2-\frac1{2}w^2+ \frac1{2} (2z-w)^2}{4}=z^2-\frac{zw}{2}.
\]
Now $z^2=A$ is a lattice point and
$\frac{zw}2=\frac12((pm-qn)+(pn+qm)i)$.
We have
$zw=(pm-qn)+(pn+qm)i$ and by \eqref{E:genident}, we have $2 b=m^2+n^2$. So if $b$ is even, then $m^2+n^2\equiv 0\pmod 4$ and hence $m,n$ are both even. In this case $\frac{zw}2$ is a lattice point, and hence $I$ is a lattice point. So we may assume that $b$ is odd and that $m,n$ are both odd. Then by \eqref{E:genident}, modulo 2  one has $pn+qm \equiv pm-qn\equiv p n-q m=-4\equiv 0$. So once again, $I$ is a lattice point.

Now suppose $\tau=3$. Then $I=\frac{2(A+C)+B}{6}$, so by \eqref{E:gendefs},
\[
I=\frac{2z^2+  \frac13(3z-w)^2+z^2-\frac1{3}w^2}{6}=z^2 - \frac{zw}{3}.
\]
We have
$zw=(pm-qn)+(pn+qm)i$ and by \eqref{E:genident}, we have $3 b=m^2+n^2$. So $m^2+n^2\equiv 0\pmod 3$ and hence $m,n$ are divisible by $3$. So $zw$ is divisible by $3$ and thus $I$ is a lattice point.

Similarly, if $\sigma=3$. Then $I=\frac{(A+C)+2B}{6}$, so by \eqref{E:gendefsII},
\[
I=\frac{z^2+  \frac1{3} y^2+2z^2-\frac{1}{3}(3z-y)^2}{6}= \frac{zy}{3}.
\]
We have
$zy=(pm-qn)+(pn+qm)i$ and by \eqref{E:genidentII}, we have $3 d=m^2+n^2$. So $m^2+n^2\equiv 0\pmod 3$ and hence $m,n$ are divisible by $3$. So $zy$ is divisible by $3$ and thus $I$ is a lattice point.

Now suppose $\tau=5$. Then $I=\frac{4(A+C)+B}{10}$. First assume that $\frac{5d+2}2,\frac{5d-2}2,c$ are not all divisible by 5, so we are in the general case. Then by \eqref{E:gendefsII},
\[
I=\frac{4z^2+  \frac15(5z-w)^2+z^2-\frac1{5}w^2}{10}=z^2 - \frac{zw}{5}.
\]
We have
$zw=(pm-qn)+(pn+qm)i$.  It was proved in the proof of Theorem~\ref{T:conversegen} that, in the general case, $m,n$ are divisible by $5$. So $zw$ is divisible by $5$ and thus $I$ is a lattice point.
Now consider  the exceptional case. By \eqref{E:gendefs2},
\[
I=\frac{4(A+C)+B}{10}=\frac{20z^2+  (5z-w)^2+5z^2-w^2}{10}=5z^2 - zw,
\]
which is clearly a lattice point.

Now suppose $\sigma=5$. Then $I=\frac{(A+C)+4B}{10}$. First assume that $\frac{5b+2}2,\frac{5b-2}2,c$ are not all divisible by 5, so we are in the general case. Then by \eqref{E:gendefs},
\[
I=\frac{z^2+ \frac15 y^2 +4z^2-\frac1{5}(5z-y)^2}{10}=\frac{zw}{5}.
\]
We have
$zw=(pm-qn)+(pn+qm)i$.  It was proved in the proof of Theorem~\ref{T:conversegen} that, in the general case, $m,n$ are divisible by $5$. So $zw$ is divisible by $5$ and thus $I$ is a lattice point.
Now consider  the exceptional case. By \eqref{E:gendefsII2},
\[
I=\frac{(A+C)+4B}{10}=\frac{5z^2+ y^2+20z^2 - (5z-y)^2}{10}=zy,
\]
which is clearly a lattice point.
\end{proof}

\begin{example}\label{E:noninttang}
The convex  tangential LEQ with vertices $(0,0),(40,9),(36,12),(35,12)$ and side lengths $ 41,  5, 1,  37$, has incenter $(\frac{106}3, 10)$,
by Proposition~\ref{P:incen}.
Similarly, the concave  tangential LEQ with vertices $(0,0),(16,63),(12,60),(11,60)$ and side lengths $ 65, 5,  1,  61$, has a non-lattice point incenter $(\frac{38}3, 58)$.
Notice that by Proposition~\ref{P:lambda}, one finds that $\lambda
=\frac89$
in both of these examples; that is, $\tau=9$. 
\end{example}

\begin{example}\label{Eg:except} The proof of Theorem~\ref{T:conversegen} was complicated by the two exceptional cases. Let us show that such cases really do occur. Further, one might wonder whether it might be possible to remove this inconvenience by taking a reflection in the line $y=x$ and thus interchanging $a$ with $d$ and $b$ with $c$. Our examples show that this is not always possible.

Consider the tangential LEQ with vertices $(0,0),(35,120)(32,116)(32,126)$. It has $\tau=5$ and the side lengths $a,b,c,d$ are $125, 5, 10, 130$ respectively. Here $u=(5b-a)/2=-50, v=(5b-a)/2=75,c=10$.
 So
 $u,v,c$ are all divisible by 5, which is the first exceptional case. Interchanging $a$ with $d$ and $b$ with $c$ would give new values
 $(u,v,c)=(-40,90,5)$ but again $u,v,c$ are all divisible by 5.
Similarly, consider the tangential LEQ with vertices $(0,0),(231,108),(228,108),(240,117)$. It has $\tau=9$ and the side lengths $a,b,c,d$ are $255, 3, 15, 267$ respectively. Here $u=(9b-a)/2=-114, v=(9b-a)/2=141$.
 So
 $\gcd(18,u,v,c)=3$, which is the second exceptional case. Interchanging $a$ with $d$ and $b$ with $c$ would give new values
 $(u,v,c)=(-66,201,15)$ but again $\gcd(18,u,v,c)=3$.
\end{example}



\section{Infinite families of tangential LEQs}\label{S:disc}

As we saw in Section~\ref{S:tanleqs}, the four infinite families $K1-K4$ of kites from \cite[Theorem~1]{AC2} gave examples of tangential LEQs with $(\sigma,\tau)$ equal to $(5,5/4),(5/4,5),(2,2)$ and $(9/8,9)$ respectively. Also in  Section~\ref{S:tanleqs}, we exhibited an infinite nested family of  tangential LEQs with $(\sigma,\tau)=(3,3/2)$. In this final section we exhibit 
infinite families with $(\sigma,\tau)=(3/2,3)$ and $(9,9/8)$, thus showing that there are infinitely many tangential LEQs in each of the seven cases of Theorem~\ref{T:sigmatau}. The method employed can also be used to give further examples in the cases where $(\sigma,\tau)$ is $(5,5/4),(5/4,5),(2,2),(9/8,9)$ and $(3,3/2)$. While by no means comprehensive, we hope the examples in this section will convey the impression that tangential LEQs are quite abundant.

First observe that for $\tau=3$, Equation \ref{E:gentau} of Theorem~\ref{T:sigmatau} is:
\begin{equation}\label{E:3fam}
6^2 + u^2=v^2-\left(v-c\right)^2.
\end{equation}
So this equation can be solved by fixing $u$,  and then expressing $6^2+u^2$ as a difference of two squares.
Recall that an integer can be written as a difference of two squares if and only if it is  odd or a multiple of 4; see sequence A100073 in  OEIS \cite{OEIS}. 
Clearly $6^2+u^2\not\equiv2\pmod 4$, for all $u$. Thus for every integer value of $u$, there are integers $v,c$ for which $6^2 + u^2=v^2-(v-c)^2$. (Note that we are interested in solutions $v\in\N$ and $u\in\Z$). However,  we must also impose the restrictions of Theorem~\ref{T:conversegen}. We require $u+v\equiv 0\pmod 3$
  and $c>0$, and $c$ is not divisible by $3$, as well as the three conditions (i) -- (iii):
  \begin{enumerate}
\item[\rm(i)]  $3|c-b| <a+c$,
\qquad(ii) \ $3(a+d) >a+c$, 
\qquad(iii) \    $3(b+c) \not=a+c $.
\end{enumerate}
For convenience we separate (i) into two conditions:
 \begin{enumerate}
\item[\rm(ia)]  $3(c-b) <a+c$,
\qquad(ib) \  $3(b-c) <a+c $.
\end{enumerate}
Notice that the conditions can be rewritten in terms of $u=\frac{3b-a}2,v=\frac{3b+a}2$ as:
  \begin{enumerate}
\item[\rm(ia)] $c<v$,
\qquad(ib) \   $u<2c$,
\qquad(ii) \   $3u<(2v+c)$,
\qquad(iii) \    $u\not=-c$.
\end{enumerate}

It is not true that for every integer value of $u$, there are integers $v,c$ for which $6^2 + u^2=v^2-(v-c)^2$ and the above restrictions hold. For example, for $u=4$, the only solutions are $v=14,c=2$ and $v=14,c=26$, but (ib) fails for the first solution and (ia) fails for the second.

One infinite family of solutions is as follows: for negative integer $x$, let $u=6x-1$. Then it is easy to check that 
$v=6 (3x^2 - x + 3) + 1, c=1$ is a solution to \eqref{E:3fam}. Note that $u+v \equiv 0\pmod 3$. Condition (ia) is true for all $x$. Conditions (ib) and  (iii) hold as $x<0$. Condition (ii) can be written as 
$6x^2 -5 x + 7>0 $, which is true for all real $x$. Notice that this infinite family has the values:
\[
a=18x^2 -12x +19,\quad b=6x^2 +6,\quad c=1,\quad d= 12x^2 -12x +14.
\]
Another infinite family of solutions is obtained by taking $c=2$ and  for integer $x\le -3$, letting $u=6x-2$ and $v=9x^2 -6x +11$. This family has the values:
\[
a=9x^2 -12x +13,\quad b=3x^2 +3,\quad c=2,\quad d= 6x^2 -12x +12.
\]

We now turn to $(\sigma,\tau)=(9,9/8)$.
First observe that for $\sigma=9$, Equation \eqref{E:gensigma} of Theorem~\ref{T:sigmatau} is:
\begin{equation}\label{E:9fam}
18^2 + u^2=v^2-\left(v-4c\right)^2.
\end{equation}
We require $u+v\equiv 0\pmod 9$
  and $c>0$, as well as the four conditions:
 \begin{enumerate}
\item[\rm(ia)]  $9(d-a) <8(a+c)$,
\qquad(ib) \    $9(a-d) <8(a+c) $,
\item[\rm(ii)]  $9(a+d) >8(a+c)$,  
 \qquad(iii) \   $9(a+2c-d) \not=8(a+c) $,
\end{enumerate}
which can be rewritten in terms of $u=\frac{9d-a}2,v\frac{9d+a}2$ as:
  \begin{enumerate}
\item[\rm(ia)] $8v+4c>9u$,
\qquad(ib) \   $u> -4c$,
\qquad(ii) \  $v>4c$,
\qquad(iii) \   $u \not= 5c$.
\end{enumerate}
One infinite family of solutions is as follows: for integer $x>1$, let $u=6x-1$, then it is easy to check that 
$v=6 (3 x^2 - x + 27) +1, c=1$ is a solution to \eqref{E:9fam}. Note that $u+v \equiv 0\pmod 9$. Condition  (ia) is true for all $x$. Conditions (ib) and  (iii) hold as $x>1$. Condition (ii) can be written as $6x^2 -2x +53>0$, which is true for all real $x$. Notice that this infinite family has the values:
\[
a=2 (9 x^2 -6 x + 82),\quad b=16x^2 -12x+147,\quad c=1,\quad d= 2  x^2 + 18.
\]

\newpage
\part{Extangential quadrilaterals}


\section{Basic notions for extangential LEQs}\label{S:extan}

An \emph{extangential} quadrilateral is a quadrilateral with an \emph{excircle},
that is,  a circle exterior to the quadrilateral that is tangent to the extensions of all four sides \cite{Joe1,Joe2,Joe3}.
Analogous to Pitot's Theorem, one has the following result \cite{Sau}: a quadrilateral with consecutive side lengths $a,b,c,d$ is extangential if and only if 
it has no pair of parallel sides and 
\begin{equation}\label{E:extan}
|a-c|=|b-d|.
\end{equation}
As for Pitot's Theorem, the above criteria is usually only stated for convex quadrilaterals, but also holds in the concave case. Indeed, if $OABC$ is a concave quadrilateral  with reflex angle at $B$,  let $A'$ (resp.~$C'$) denote the point of intersection of the side $OA$ (resp.~$OC$) and the extension of side $BC$ (resp.~$AB$). Let $a,b,c,d$ denote the lengths of $OA,AB,BC,CO$ respectively, and similarly, let $a',b',c',d'$ denote the lengths of $OA',A'B,BC',C'O$.  
Then $a+b=c+d$ if and only if $a'+b'=c'+d'$. This follows from a result sometimes referred to Urquhart's quadrilateral theorem, which has a long and interesting history; see \cite{Sau,Ped,Ha}.

A quadrilateral  $OABC$ with consecutive side lengths $a,b,c,d$ can have at most one excircle and when one exists, its radius $r_e$, called the \emph{exradius}, is given by the formula 
$r_e=\frac{K}{|a-c|}$ \cite[Theorem~8]{Joe1}.
By relabelling the vertices if necessary, we will suppose throughout this paper that the excircle lies outside the vertex $B$. In particular, the extangential hypothesis is now $a+b=c+d$, and from the proof  of \cite[Theorem~8]{Joe1}, one has $a>c$ and
\begin{equation}\label{E:exradK}
r_e=\frac{K}{a-c}.
\end{equation}
For an extangential quadrilateral, one of the diagonals $L$ separates the sides into two pairs of equal sum, and the excircle is located outside one of the vertices joined by $L$ (for kites $L$ is the axis of symmetry). In fact, of these two vertices, the excircle is located outside the vertex at which the quadrilateral makes the largest angle \cite{Joe1}. For concave extangential quadrilaterals, the excircle is located outside the vertex with the reflex angle. Obviously, extangential  quadrilaterals cannot have a pair of parallel sides; in particular no trapezoid is extangential  and no parallelogram is extangential even though parallelograms satisfy the $a+b=c+d$ condition.

We remark that there is a strong relationship between tangential and extangential quadrilaterals. Recall from \eqref{E:tan} that  a quadrilateral $OABC$ is tangential if and only if $a+c=b+d$. If $OABC$ is a convex quadrilateral, and if $B'$ denotes the reflection of $B$ in the perpendicular bisector of $AC$, then the quadrilaterial 
$OAB'C$ is extangential if and only if $OABC$ is tangential. Moreover, equability is preserved by this construction. However, notice that if $OABC$ is a LEQ, $OAB'C$ may fail to have its vertices on lattice points, as in Figure~\ref{F:tanextanconv}. When $OABC$ is concave, the same construction can be made, but it can happen that 
$OAB'C$  has self-intersections, as in Figure~\ref{F:tanextanconc}.

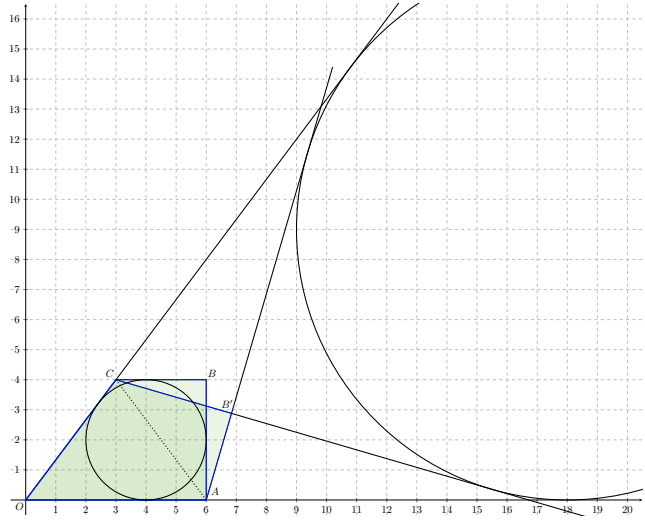
\begin{figure}
\begin{tikzpicture}[scale=.4][line cap=round,line join=round,>=triangle 45,x=1cm,y=1cm]
\begin{axis}[
x=1cm,y=1cm,
axis lines=middle,
grid style=dashed,
ymajorgrids=true,
xmajorgrids=true,
xmin=-.5,
xmax=20.5,
ymin=-.5,
ymax=16.5,
xtick={0,1,...,30},
ytick={0,1,...,20},]
\draw[color=ttzzqq,fill=ttzzqq,fill opacity=0.1] (0,0) -- (6,0) -- (171/25,72/25) -- (3,4) -- cycle; 
\draw[line width=1pt,color=qqqqff] (0,0) -- (6,0) -- (171/25,72/25) -- (3,4) -- cycle; 
\draw[color=ttzzqq,fill=ttzzqq,fill opacity=0.1] (0,0) -- (6,0) -- (6,4) -- (3,4) -- cycle; 
\draw[line width=1pt,color=qqqqff] (0,0) -- (6,0) -- (6,4) -- (3,4) -- cycle; 
\draw [line width=1pt,dotted,color=black] (3,4)-- (6,0);
\draw [line width=1pt,color=black] (3,4) -- (15,20);
\draw [line width=1pt,color=black] (-24+5*171/25,5*72/25) -- (171/25,72/25);
\draw [line width=1pt,color=black] (171/25-15+5*171/25,72/25-20+5*72/25) -- (171/25,72/25);
\draw [color=black] (-0.2,-0.2) node {$O$};
\draw [color=black] (6.3,.3) node {$A$};
\draw [color=black] (6.2,4.2) node {$B$};
\draw [color=black] (6.7,3.2) node {$B'$};
\draw [color=black] (2.8,4.2) node {$C$};
\draw[line width=1pt] (4,2) circle (2);
\draw[line width=1pt] (18,9) circle (9);
\end{axis}
\quad 
\end{tikzpicture}
\caption{Tangential to extangential convex quadrilaterals}\label{F:tanextanconv}
\end{figure}

\begin{figure}[h!]
\definecolor{qqqqff}{rgb}{0,0,1}
\definecolor{ttzzqq}{rgb}{0.2,0.6,0}
\begin{tikzpicture}[scale=.22][line cap=round,line join=round,>=triangle 45,x=1.0cm,y=1.0cm]
\begin{axis}[
x=1cm,y=1cm,
axis lines=middle,
grid style=dashed,
ymajorgrids=true,
xmajorgrids=true,
xmin=-1,
xmax=19,
ymin=-1,
ymax=24.5,
xtick={0,2,...,18},
ytick={0,2,...,24},]
\draw[color=ttzzqq,fill=ttzzqq,fill opacity=0.1]  (0,0) -- (10,24)-- (10,18) --(18,24)  -- cycle;
\draw[line width=2pt,color=qqqqff]   (0,0) -- (10,24)--  (10,18) --(18,24)  -- cycle;
\draw (10,18) node[scale=2,anchor= north] {$B$};
\draw (10,24) node[scale=2,anchor=east] {$C$};
\draw (18,24) node[scale=2,anchor=west] {$A$};
\draw[line width=2pt,color=qqqqff]  (0,0) node[scale=2,anchor=north east] {$O$};
\draw[line width=1pt,dashed]  (10,24)--(18,24) ;
\end{axis}
\draw [->,line width=0.8pt] (21,12.5) -- (24,12.5);
\end{tikzpicture}
\hskip.1cm 
\begin{tikzpicture}[scale=.22][line cap=round,line join=round,>=triangle 45,x=1.0cm,y=1.0cm]
\begin{axis}[
x=1cm,y=1cm,
axis lines=middle,
grid style=dashed,
ymajorgrids=true,
xmajorgrids=true,
xmin=-1,
xmax=19,
ymin=-1,
ymax=24.5,
xtick={0,2,...,18},
ytick={0,2,...,24},]
\draw[color=ttzzqq,fill=ttzzqq,fill opacity=0.1]  (0,0) -- (10,24) -- (18,18)--(18,24)  -- cycle;
\draw[line width=2pt,color=qqqqff]   (0,0)-- (10,24) -- (18,18)--(18,24)  -- cycle;
\draw[line width=1pt,dashed]  (10,24)--(18,24) ;
\draw (18,18) node[scale=2,anchor= west] {$B'$};
\draw (10,24) node[scale=2,anchor=east] {$C$};
\draw (18,24) node[scale=2,anchor=west] {$A$};
\draw[line width=2pt,color=qqqqff]  (0,0) node[scale=2,anchor=north east] {$O$};
\end{axis}
\end{tikzpicture}
\caption{Self-intersections can occur in the concave case}\label{F:tanextanconc}
\end{figure}
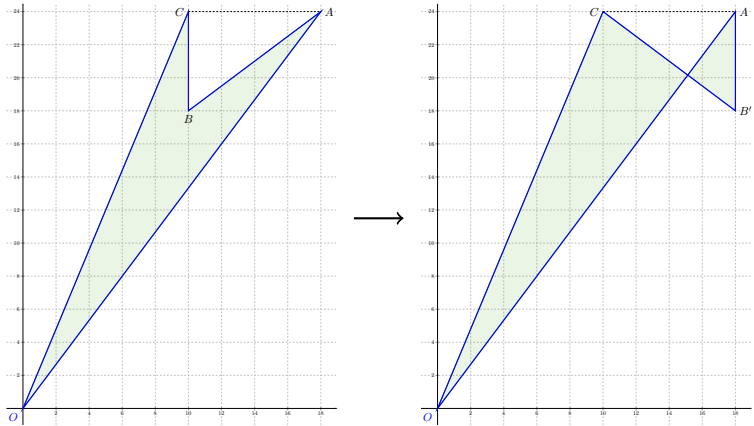

Analogous to  Proposition~\ref{P:kitess}, we have the following result; its proof is very similar to that of Proposition~\ref{P:kitess}.

\begin{proposition}\label{P:exkitess}
If $OABC$ is extangential with excircle outside $B$, then $OABC$ is a kite if and only if
 the diagonal $OB$ divides $OABC$  into two triangles of equal area.
\end{proposition}

\begin{proof}
Obviously,  if $OABC$ is a kite, then its axis of symmetry diagonal divides $OABC$  into two triangles of equal area.
Conversely, applying Heron's formula to triangle $OAB$ gives
\begin{align*}
16K_A^2&= (a+b+p)(a+b-p)(a-b+p)(-a+b+p)\\
&=-(a^2- b^2)^2 + 2 (a^2 + b^2) p^2 - p^4.
\end{align*}
Similarly, from triangle $OBC$
\[
16K_C^2= -(c^2- d^2)^2 + 2 (c^2 + d^2) p^2 - p^4.
\]
Hence, subtracting,
\begin{equation}\label{E:exforp}
2(a^2 +b^2 -c^2- d^2  )p^2=16(K_A^2-K_C^2)+(a^2- b^2)^2 -(c^2- d^2)^2.
\end{equation}
Notice that, using $a+b=c+d$,
\[
a^2 +b^2 -c^2- d^2=(a-d)(a+d)+(b-c)(b+c)=2(a-c)(a-d),
\]
and
\begin{align*}
(a^2- b^2)^2 -(c^2- d^2)^2&=(a-b)^2(a+b)^2-(d-c)^2(c+d)^2\\
&=4(a-c)(a-d)(a+b)^2.
\end{align*}
So \eqref{E:exforp} gives
\begin{equation}\label{E:exareadiff}
(a-c)(a-d)p^2=4(K_A^2-K_C^2)+(a-c)(a-d)(a+b)^2.
\end{equation}
Now assume that $K_A=K_C$. Then  \eqref{E:exareadiff}
gives $(a-c)(a-d)p^2=(a-c)(a-d)(a-b)^2$.
Notice that $p=\pm(a-b)$ is impossible, as otherwise the triangle $OAB$ would be degenerate. Hence either $a=d$ or $a=c$.
If $a=c$, then by the extangential hypothesis, $b=d$ and so $OABC$ is a parallelogram, which is impossible. So $a=d$. Note that as  $K_A=K_C$, the points $A,C$ are equidistant from the line through $O,B$. So the triangles $OAB$ and $OBC$ are congruent, and hence $OABC$ is a kite.
\end{proof}

Like the incenter of a tangential quadrilateral, the excenter of an extangential quadrilateral lies on the Newton line $\mathcal{N_L}$ joining the midpoints of the two diagonals. We have not seen this stated explicitly in the literature, but the proof in the tangential case is readily adapted. For example, 
the vector proof of Anne's Theorem given in \cite[Lemma 1]{DC} is valid as is, for signed areas, as the authors indicate, and then \cite[Theorem 3]{DC} can be easily modified, with two positive areas and two negative areas. Since the excenter $I_e$ lies on Newton line, $I_e$ is of the form $\lambda_e M_{A}+(1-\lambda_e) M_{O}$ for some $\lambda_e \in [0,1]$, where we recall $M_{A}, M_{O}$ refer to the midpoints of the diagonals $AC,OB$ respectively.

For the rest of this section,  $OABC$ denotes an extangential (convex or concave) quadrilateral, with vertices in counterclockwise cyclic order, and $a,b,c,d$ denote the lengths of the sides $OA,AB,BC,CO$ respectively. We suppose furthermore that the excircle lies outside the vertex $B$. In particular, the extangential hypothesis is now $a+b=c+d$, and from the proof  of \cite[Theorem~8]{Joe1}, one has $a>c$.

\begin{remark}\label{R:a>b} By reflecting in the line $y=x$ if necessary, we may always assume that $a\ge d$. In this case, we have $a-b\ge d-b=a-c>0$; thus $a>b$. Hence $a=\max\{a,b,c,d\}$. Similarly, $b=\min\{a,b,c,d\}$.
\end{remark}

For tangential quadrilaterals, equability is equivalent to the condition that the inradius is 2. For extangential quadrilaterals, the equability condition \eqref{E:exradK}
is equivalent to the condition 
\begin{equation}\label{E:exrad}
r_e=\frac{2(a+b)}{a-c}.
\end{equation}
Note that if $OABC$ is not a kite, then by Proposition~\ref{P:exkitess}, $K_A\not=\frac{K}2$ and $K_O\not=\frac{K}2$. So, for equable extangential quadrilaterals that are not kites,
$K_A\not=a+b$ and $K_O\not=a+b$. In fact, one has $K_O\not=a+b$ even when $OABC$ is a kite (with its excircle outside the vertex $B$). Indeed, otherwise $OABC$ would be a rhombus, and no rhombus is extangential. This will be important in Proposition~\ref{P:lambdae}  below.

Analogous to many results for tangential quadrilaterals, there are very similar results for extangential quadrilaterals.
Indeed, analogous to  Propositions~\ref{P:kitess2}, \ref{P:incen}, \ref{P:famcor} and \ref{P:lambda} of Part 1, we have the following four analogous propositions. We omit the proofs which are essentially the same as those of the propositions of Part 1.


\begin{proposition}\label{P:exkitess2}
If $OABC$ is extangential, then  $OABC$ is a kite if and only if the Newton line $\mathcal{N_L}$ contains one of the diagonals.
\end{proposition}

\begin{proposition}\label{P:excen}
If $OABC$ is extangential, we have the following two expressions for the excenter $I_e$:
\begin{enumerate}
\item $I_e=\frac{r_e}2\,\frac{aC+dA}{ K_O}$,
\qquad (b) $I_e=A+\frac{r_e}2\,\frac{a(B-A)+bA}{ K_A}$.
\end{enumerate}
\end{proposition}

\begin{proposition}\label{P:identity} If $OABC$ is extangential,  we have:
\[
(K_A -\frac{r_e}2(a-b))(K_O-\frac{r_e}2(a+d))=\frac{r_e^2}4(ac-bd).
\]
\end{proposition}


\begin{proposition}\label{P:lambdae}
If $OABC$ is extangential but is not a kite, we have the following two expressions for the coordinate $\lambda_e$:
\begin{enumerate}
\item $\lambda_e=\frac{r_e}2\cdot \frac{a+b}{  K_O-(a+b)}$,
\qquad(b) $\lambda_e=1-  \frac{r_e}2\cdot\frac{c-b}{ K_A-(a+b)}$. 
\end{enumerate}
Furthermore, the first of the above expressions for $\lambda_e$ is valid if $OABC$ is a kite.
\end{proposition}

\begin{example}
Apart from the rhombus of side length 5 (K1, $n=2$) and the $4\times 4$ square (K3, $n=1$), the lattice equable kites of \cite[Theorem~1]{AC2} are extangential. 
In each case $n,j$, to determine the exradius $r_{e,n,j}$, the excenter $I_{e,n,j}$ and the parameter 
$\lambda_{e,n,j}$, one can employ  \eqref{E:exrad} and  Propositions~\ref{P:excen}  and \ref{P:lambdae}. Here  $I_{e}=\lambda_{e} M+(1-\lambda_{e} )\frac{B}2$,  where $M=M_A$. We omit the details, which are completely routine. The results are given in Table~\ref{T:kitefam}. Notice that unlike the incenters, the excenters are not necessarily lattice points.

 \bigskip
\begin{table}
\begin{tabular}{c|c|c|c|c|c|c}
  \hline
   Case & Equation & $M$  & $B$   & $r_e$ & $I_{e,n,j}$& $\lambda_{e,n,j}$\\\hline
  \emph{K1}& $n^2-5j^2=4$ & $\frac12(n+5j)(2,1)$  & $n(2,1)$&$\frac{2n}j$   &$\frac{n(n+j)}{2j}(2,1)$&$\frac{n^2}{5 j^2}$\\
  \emph{K2} & $n^2-5j^2=1$& $ (2n+5j)(2,1)$  & $4n(2,1)$&   $\frac{n}{j}$ &$\frac{n(n+2j)}{j}(2,1)$ &$\frac{n^2}{5 j^2}$\\
\emph{K3} & $n^2-2j^2=1$& $(n+2j)(2,2) $ & $ 4n(1,1)$&  $\frac{2n}j$ &$\frac{2n(n+j)}j (1,1)$&$\frac{n^2}{2 j^2}$\\
 \emph{K4} & $2n^2-j^2=1$& $ (4n+3j)(\frac32,\frac32)$&$12n(1,1)$ & $\frac{3n}{j}$  &$\frac{3n(3n+2j)}{j} (1,1) $&$\frac{2n^2}{ j^2}$\\
\hline
\end{tabular}
\bigskip
\caption{The four kite families}\label{T:kitefam}
\end{table}

\begin{figure}[h]
\begin{tikzpicture}[scale=.3][line cap=round,line join=round,>=triangle 45,x=1cm,y=1cm]
\begin{axis}[
x=1cm,y=1cm,
axis lines=middle,
grid style=dashed,
ymajorgrids=true,
xmajorgrids=true,
xmin=-2,
xmax=18,
ymin=-2,
ymax=18,
xtick={0,2,...,18},
ytick={0,2,...,18},]
\draw[color=ttzzqq,fill=ttzzqq,fill opacity=0.1] (0,0) -- (12,9) -- (12,12) -- (9,12) -- cycle; 
\draw [line width=2pt,color=qqqqff] (12,12)-- (12,9);
\draw [line width=2pt,color=qqqqff] (12,9)-- (0,0);
\draw [line width=2pt,color=qqqqff] (0,0)-- (9,12);
\draw [line width=2pt,color=qqqqff] (9,12)-- (12,12);
\draw [dashed,line width=.5pt] (0,0)-- (12,12);
\draw[line width=1pt,color=qqqqff]   (12,12)-- (12,18);
\draw[line width=1pt,color=qqqqff]   (12,12)-- (18,12);
\draw[line width=1pt,color=qqqqff]   (9,12)-- (15,20);
\draw[line width=1pt,color=qqqqff]   (12,9)-- (20,15);
\draw  (15,15) circle (3);
\draw [fill=black] (15,15) circle (3pt);
\end{axis}
\end{tikzpicture}
\caption{Kite with side lengths 3 and 15 (K4, $n=j=1$)}\label{F:315}
\end{figure}
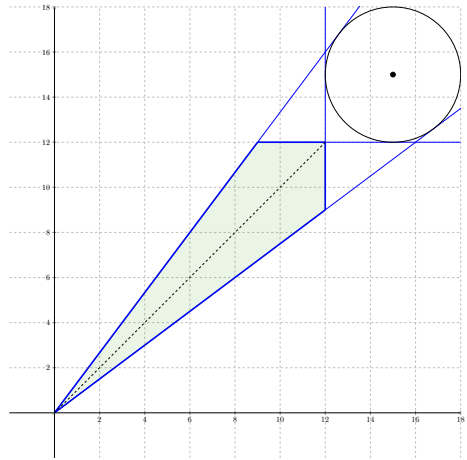
\end{example}

\begin{example}
Consider the convex extangential LEQ shown on the left of Figure~\ref{F:extans}; its vertices are (0,0),(21,20),(20,20),(0,5)
and the side lengths are 29, 1, 25, 5.
The exradius is $r_e=\frac{K}{a-c}=15$.
By Proposition~\ref{P:excen}, the 
excenter $I_e$
is $
=(15,35)$, and by Proposition~\ref{P:lambdae}, the coordinate $\lambda_e$ of the excenter is
$
10$.

Similarly, a concave extangential LEQ is shown on the right of Figure~\ref{F:extans}; 
its vertices are (0,0),(12,5),(10,5),(6,8),
and the side lengths are:
13, 2, 5, 10.
The exradius is $r_e=\frac{K}{a-c}=\frac{15}4$.
By Proposition~\ref{P:excen}, the 
excenter $I_e$
is $
\frac{5}{4}(9,7)$, and by Proposition~\ref{P:lambdae}, the coordinate $\lambda_e$ of the excenter is
$
\frac{25}{16}$.

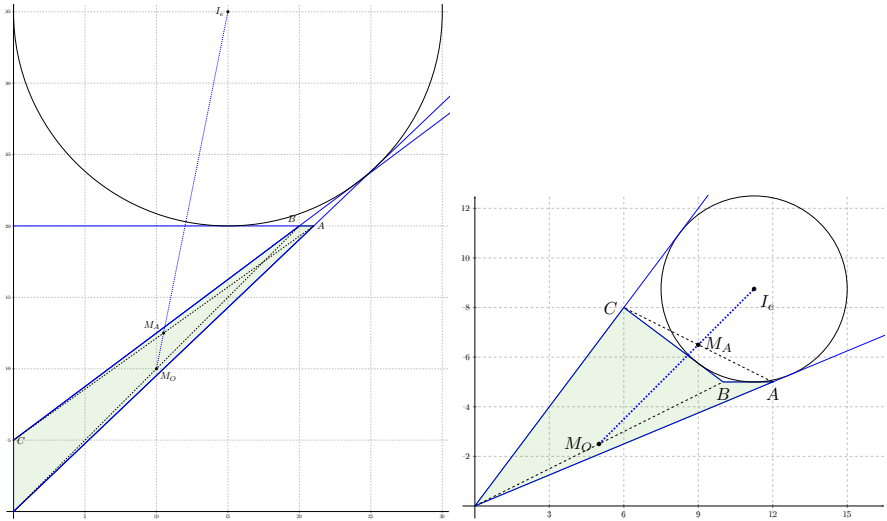
\begin{figure}
\begin{tikzpicture}[scale=.19]
\begin{axis}[
x=1cm,y=1cm,
axis lines=middle,
grid style=dashed,
ymajorgrids=true,
xmajorgrids=true,
xmin=-0.5,
xmax=30.5,
ymin=-0.5,
ymax=35.5,
xtick={5,10,...,30},
ytick={5,10,...,35},]
\draw[color=ttzzqq,fill=ttzzqq,fill opacity=0.1] (21,20) -- (20,20) -- (0,5) -- (0,0) -- cycle; 
\draw[line width=2pt,color=qqqqff] (21,20) -- (20,20) -- (0,5) -- (0,0) -- cycle; 
\draw[dotted,line width=2pt,color=qqqqff]   (10,10)-- (15,35);
\draw[line width=1pt,color=qqqqff]   (0,0)-- (42,40);
\draw[line width=1pt,color=qqqqff]   (0,5)-- (40,35);
\draw[line width=1pt,color=qqqqff]   (0,20)-- (20,20);
\draw[dashed,line width=1pt]   (21,20)-- (0,5);
\draw[dashed,line width=1pt]   (20,20)-- (0,0);
\draw[line width=1pt] (15,35) circle (15);
\draw [fill=black] (21/2,25/2) circle (2pt);
\draw [fill=black] (10,10) circle (2pt);
\draw [fill=black] (15,35) circle (2pt);

\draw (21,20) node[scale=2,anchor=west] {$A$};
\draw (20,20) node[scale=2,anchor= south east] {$B$};
\draw (0,5) node[scale=2,anchor=west] {$C$};
\draw (15,35) node[scale=2,anchor=east] {$I_e$};
\draw (21/2,25/2) node[scale=2,anchor=south east] {$M_{A}$};
\draw (10,10) node[scale=2,anchor=north west ] {$M_{O}$};

\end{axis}
\end{tikzpicture}
\begin{tikzpicture}[scale=.33][line cap=round,line join=round,>=triangle 45,x=1cm,y=1cm]
\begin{axis}[
x=1cm,y=1cm,
axis lines=middle,
grid style=dashed,
ymajorgrids=true,
xmajorgrids=true,
xmin=-0.5,
xmax=16.5,
ymin=-0.5,
ymax=12.5,
xtick={0,3,...,21},
ytick={0,2,...,12},]
\draw[color=ttzzqq,fill=ttzzqq,fill opacity=0.1] (0,0) -- (12,5)--(10,5)--(6,8) -- cycle; 
\draw[line width=1pt,color=qqqqff] (0,0) --  (12,5)--(10,5)--(6,8)-- cycle; 
\draw[line width=1pt] (5*9/4,5*7/4) circle (15/4);
\draw[line width=1pt,color=qqqqff]   (6,8)-- (12,16);
\draw[line width=1pt,color=qqqqff]   (12,5)-- (18,7.5);

\draw[dotted,line width=2pt,color=qqqqff]   (5,5/2)-- (5*9/4,5*7/4);
\draw[dashed,line width=1pt]   (0,0)-- (10,5);
\draw[dashed,line width=1pt]   (12,5)--(6,8);

\draw [fill=black] (5*9/4,5*7/4) circle (2pt);
\draw [fill=black] (9,13/2) circle (2pt);
\draw [fill=black] (5,5/2) circle (2pt);

\draw (12,5) node[scale=2,anchor=north] {$A$};
\draw (10,5) node[scale=2,anchor= north ] {$B$};
\draw (6,8) node[scale=2,anchor=east] {$C$};
\draw (5*9/4,5*7/4) node[scale=2,anchor=north west] {$I_e$};
\draw (9,13/2) node[scale=2,anchor=west] {$M_{A}$};
\draw (5,5/2) node[scale=2,anchor=east ] {$M_{O}$};

\end{axis}
\end{tikzpicture}
\caption{Extangential LEQs; one convex and one concave}\label{F:extans}
\end{figure}
\end{example}


\section{Lemmata for extangential LEQs}\label{S:exlem}

For this section, $OABC$ denotes a non-kite extangential quadrilateral with consecutive sides $a,b,c,d$ and with its excircle outside the vertex $B$, so $a>c$. In particular, it has exradius $r_e=\frac{2(a+b)}{a-c}$. As explained in Remark~\ref{R:a>b}, we may assume  $a=\max\{a,b,c,d\}$ and $b=\min\{a,b,c,d\}$.

The approach adopted in this section is the same as that of Section~\ref{S:lem}, and the results obtained are analogous, but the calculations are often more complicated.
For the convenience of the reader, we repeat Equations~\eqref{E:p1} and  \eqref{E:q1}:
\begin{align}
 p^2=a^2+b^2\pm2  \sqrt{a^2b^2-(2K_A)^2},\label{E:exp1}\\
 q^2=a^2+d^2\pm2  \sqrt{a^2d^2-(2K_O)^2},\label{E:exq1}
\end{align}
where $p,q$ are the lengths of the diagonals $OB,AC$ respectively. As $O,A,B,C$ are lattice points, $p^2,q^2$ are integers, so by Lemma~\ref{L:pq}, the integers $a^2b^2-(2K_A)^2$ and $a^2d^2-(2K_O)^2$ are squares.

\begin{lemma}\label{L:expandq}  One has 
\[
p^2=
\frac{8 (a+b) (K_A - K_C)}{(a-c)(a-d)} + (a + b)^2\quad\text{and}\quad
q^2=
\frac{8(K_O - K_B)}{a-c} + (a -d)^2.
\]
\end{lemma}

\begin{proof} As $OABC$ is not a kite, by hypothesis, so $a\not=d$. Arguing exactly as in Proposition \ref{P:exkitess} we reobtain \eqref{E:exareadiff}:
\[
(a-c)(a-d)p^2=4(K_A^2-K_C^2)+(a-d)(a-c)(a+b)^2,
\]
and since $OABC$ is extangential and hence not a parallelogram, $a\not=c$. Thus, as $K_A^2-K_C^2=(K_A+K_C)(K_A-K_C)=2(a+b)(K_A-K_C)$,  
we obtain the required formula for $p^2$.

Similarly,  by applying Heron's formula to triangles $OAC$ and $BCA$, we obtain 
\begin{equation}\label{E:forq}
2(d^2 - c^2 + a^2 - b^2 )q^2=16(K_O^2-K_B^2)+(d^2- a^2)^2 -(b^2- c^2)^2.
\end{equation}
Simplifying as in Proposition \ref{P:exkitess}  gives
\[
(a-c)(a+b)q^2=8(a+b)(K_O-K_B)+(a+b)(a-c)(a-d)^2,
\]
from which the required formula for $q^2$ follows. 
\end{proof}

\begin{remark}\label{R:exKAint} From the above lemma, using  \eqref{E:exp1},
\[
\frac{8(a+b) (K_A - (a+b))}{(a-c)(a-d)} =\frac{p^2- (a + b)^2}2= -ab\pm  \sqrt{a^2b^2-(2K_A)^2}, 
\]
which is an integer by Lemma \ref{L:pq}. Similarly, $\frac{8 (K_O - (a+b))}{a-c}$  is an integer. 
\end{remark}

\begin{lemma}\label{T:exsq1} The integer 
$abcd-4( a+b)^2 $
 is a square, and
 \begin{align*}
K_A&=(a+b)+(a-c)(a-d)(a + b)\frac{- (a b + c d) \pm 2\sqrt{a b c d - 4 (a + b)^2}}{16(a+b)^2 +  (a - c)^2 (a - d)^2},
\\
K_O&=(a+b)+(a-c)\frac{ad+bc\mp  2 \sqrt{abcd-4( a+b)^2}}{ 16 + (a-c)^2 },
\end{align*}
where  the signs of the square roots in the formulas for $K_O$ and $K_A$ are opposite.
\end{lemma}


\begin{remark}\label{R:exposit} In the statement of the above lemma, the terms 
\[
ab+cd\pm 2\sqrt{a b c d - 4 (a + b)^2}\quad \text{and}\quad ad+bc\mp  2 \sqrt{abcd-4( a+b)^2}
\]
 are strictly positive. Indeed, using $d=a+b-c$, by the arithmetic mean-geometric mean inequality,
$ab+cd \ge 2\sqrt{a b c d}>\sqrt{a b c d - 4 (a + b)^2}$.
In particular, $K_A<\frac{K}2=a+b$ if and only if $a>d$, which is opposite to the situation for tangential LEQs; see Remark~\ref{R:posit}.
\end{remark}

\begin{proof}[Proof of Lemma~\ref{T:exsq1}]
From Lemma \ref{L:expandq} and Equation \eqref{E:exq1},
\[
\frac{4 (K_O - K_B)}{a-c} -ad  =\frac{q^2-a^2-d^2}2 =\pm  \sqrt{a^2d^2-(2K_O)^2},
\]
so setting $s:=\frac{K_O - (a+b)}{a-c}$, squaring, and using $K_O+K_B=2(a+b)$ gives
 \[
 \alpha s^2-2\beta s+\gamma=0,
 \]
where
\[
   \alpha =16 + (a-c)^2,\quad
 \beta=
  a d + bc ,\quad
\gamma=(a + b)^2.
\]
Thus, as $\beta^2-\alpha\gamma=4 (a b c d - 4 (a + b)^2)$ (using $a+b=c+d$ again), we have
\[
s=\frac{a d + bc \pm 2\sqrt{a b c d - 4 (a + b)^2}}{16 + (a-c)^2},
\]
which gives the required formula for $K_O$. In particular, as $s$ is rational, $abcd-4( a+b)^2$ is a square, as claimed.

The formula for $K_A$ is similarly obtained by equating $p^2$ from Lemma \ref{L:expandq} and Equation \eqref{E:exp1}. We have
\[
\frac{4 (a+b) (K_A - K_C)}{(a-c)(a-d)} + ab =\frac{p^2-a^2-b^2}2 =\pm  \sqrt{a^2b^2-(2K_A)^2}.
\]
Define $t:=\frac{ (a+b) (K_A - (a+b))}{(a-c)(a-d)}$. One obtains $\bar \alpha t^2+2\bar \beta t+\bar \gamma=0$,
where
\[
   \bar \alpha =16(a+b)^2 +  (a - c)^2 (a - d)^2,\quad
 \bar \beta=
   (a + b)^2 (a b + c d),\quad
\bar \gamma=(a + b)^4.
\]
One has
\[
\bar \beta^2-\bar \alpha\bar \gamma=4 (a + b)^4 (a b c d - 4 (a + b)^2),
\]
so
\[
t=\frac{-(a + b)^2 (a b + c d) \pm 2(a + b)^2\sqrt{a b c d - 4 (a + b)^2}}{16(a+b)^2 +  (a - c)^2 (a - d)^2},
\]
which gives the required formula for $K_O$. 

It remains to see that the signs of the square roots in the formulas for $K_O$ and $K_A$ are opposite.
Let $R=2\sqrt{a b c d - 4 (a + b)^2}$. Obviously, we may assume that $R\not=0$ and $a\not=c$. Let us write
 \begin{align*}
K_A&=(a+b)+(a-c)(a-d)(a + b)\frac{- (a b + c d) +\delta_A R}{16(a+b)^2 +  (a - c)^2 (a - d)^2},
\\
K_O&=(a+b)+(a-c)\frac{ad+bc+\delta_O   R}{ 16 + (a-c)^2 },
\end{align*}
where $\delta_A,\delta_O$ are each $\pm1$. Using $a+b=c+d$,
\begin{align*}
K_A-&\frac{(a+b)(a-b)}{a-c}\\
&=(a+b)\left(\frac{b-c}{a-c}+(a-c)(a-d)\frac{- (a b + c d) +\delta_A R}{16(a+b)^2 +  (a - c)^2 (a - d)^2}\right)\\
&=\frac{(a+b)(d-a)}{a-c}\cdot\frac{16(a+b)^2+(a - c)^2(a c + b d)-(a-c)^2\delta_A R}{16(a+b)^2 +  (a - c)^2 (a - d)^2}, \\
K_O-&\frac{(a+b)(a+d)}{a-c}=-\frac{(a+b)(c+d)}{a-c}+(a-c)\frac{ad+bc+\delta_O   R}{ 16 + (a-c)^2 }\\
&=\frac{-1}{a-c} \cdot\frac{16(a+b)^2 +(a-c)^2 (a c + b d)-(a - c)^2\delta_O R}{16 + (a-c)^2} .
\end{align*}
Notice also that $(a+b)(d-a)=bd-ac$. Hence, by Proposition~\ref{P:identity},
\begin{equation}\label{E:excur}
\frac{X-(a-c)^2\delta_A R}{16(a+b)^2 +  (a - c)^2 (a - d)^2} \cdot \frac{X-(a - c)^2\delta_O R}{16 + (a-c)^2} =(a+b)^2,
\end{equation}
where $X=16(a+b)^2+(a - c)^2(a c + b d)$. Now, substituting $d=a+c-b$ one finds that 
\[
\frac{X-(a-c)^2\delta_A R}{16(a+b)^2 +  (a - c)^2 (a - d)^2} \cdot \frac{X+(a - c)^2\delta_A R}{16 + (a-c)^2} =(a+b)^2.
\]
Subtracting from \eqref{E:excur} gives
\begin{equation}\label{E:ze}
(X-(a-c)^2\delta_A R)\cdot (a - c)^2(\delta_A+\delta_O) R=0.
\end{equation}
Note that $X-(a-c)^2\delta_A R\not=0$ as otherwise $X^2=(a-c)^4R^2$ which would give
\[
(16(a+b)^2+(a - c)^2(a c + b d))^2-(a-c)^4(abcd-4(a+b)^2)=0,\]
and hence
\[
(a + b)^2 (16 + (a-c)^2 ) (16(a+b)^2+(a - c)^2 (b - c)^2)=0,
\]
which is impossible. So from \eqref{E:ze}, we have $\delta_A=-\delta_O$, as claimed.
\end{proof}

\begin{lemma}\label{P:exsign} 
The sign of the square root in the formulas for $K_O$ is positive if and only if $B$ lies within the circumcircle of the triangle $OAC$; in particular, the sign for $K_O$  is positive if $OABC$ is concave.
\end{lemma}

\begin{proof} In the notation of the above proof, let $x=\delta_O 2\sqrt{a b c d - 4 (a + b)^2}$, so
 \begin{align*}
K_A&=(a+b)+(a-c)(a-d)(a + b)\frac{- (a b + c d) -x}{16(a+b)^2 +  (a - c)^2 (a - d)^2},
\\
K_O&=(a+b)+(a-c)\frac{ad+bc+x}{ 16 + (a-c)^2 }.
\end{align*}
From a standard criteria for a point to be within the circumcircle of a triangle (see \cite{Fo}),
$B$ is inside the circumcircle of the triangle $OAC$ if and only if
\begin{equation}\label{E:exinside}
p^2K_O<d^2K_A+a^2K_C.
\end{equation}
Now $d^2K_A+a^2K_C=K_A(d^2-a^2)+2a^2(a+b)$, and by Lemma \ref{L:expandq},
\[
K_Op^2
=K_O\left(\frac{16(a+b)  (K_A - (a+b))}{(a-c)(a-d)} + (a + b)^2\right).
\]
So condition \eqref{E:exinside} can be written as $E>0$ where 
\[
E=K_A(d^2-a^2)+2a^2(a+b)-K_O\left(\frac{16(a+b)  (K_A - (a+b))}{(a-c)(a-d)} + (a + b)^2\right).
\]
Substituting the formulas for $K_O,K_A$ and $x$, one finds upon simplification that
\[
E=2 \delta_O(a + b) \sqrt{abcd-4 (a + b)^2 }.
\]
Hence, as claimed, $\delta_O>0$ if and only if 
$B$ is inside the circumcircle of the triangle $OAC$.
\end{proof}

 \begin{definition}\label{D:SigmaT}
 Let 
 \begin{align*}
 \Sigma&=8\cdot\frac{ad+bc+2\delta\sqrt{a b c d - 4 (a + b)^2}}{ 16 + (a-c)^2 },\\
T&=8(a + b)^2\cdot\frac{ a b + c d +  2\delta\sqrt{a b c d - 4 (a + b)^2}}{16(a+b)^2 +  (a - c)^2 (a - d)^2},
\end{align*}
 where $\delta=1$ if $B$ lies within the circumcircle of the triangle $OAC$, and $\delta=-1$ otherwise.
 \end{definition}

 \begin{remark}\label{R:exintegers} Observe that $\Sigma$ and $T$ are positive integers. Indeed, 
from Lemma~\ref{T:exsq1}, 
\begin{equation}\label{E:SigmaT1}
 \Sigma=\frac{8 (K_O - (a+b))}{a-c},\qquad T= \frac{8(a+b) ((a+b)-K_A )}{(a-c)(a-d)},
\end{equation}
which are integers by Remark \ref{R:exKAint}, and they are positive by Remark~\ref{R:exposit}. Furthermore, from Lemma~\ref{L:pandq},
\begin{equation}\label{E:SigmaT2}
 \Sigma=\frac12(q^2-(a -d)^2),\qquad T= \frac12((a+b)^2-p^2).
 \end{equation}
In the notation of   the proof of Lemma \ref{T:exsq1},
\begin{equation}\label{E:sigmaquad}
\alpha \Sigma^2-16\beta  \Sigma+64\gamma=0,
\end{equation}
where
$\alpha =16 + (a-c)^2$, 
$ \beta=
  a d + bc $,
$\gamma=(a + b)^2=(c + d)^2$,
and 
\begin{equation}\label{E:tauquad}
\bar \alpha T^2-16\bar \beta T+64\bar \gamma=0,
\end{equation}
where
$   \bar \alpha =16(a+b)^2 +  (a - c)^2 (a - d)^2$,
$ \bar \beta= (a + b)^2 (a b + c d)$,
and $\bar \gamma=(a + b)^4$.
\end{remark}

\begin{lemma}\label{L:exsum}  The following relations hold:
\begin{enumerate}
\item $\Sigma T=8\frac{(a+b)^2}{(a-c)^2}(T-{\Sigma})$,
\item $2\Sigma T  = (a + b)^2 (\Sigma -8)- (b - c)^2   T$.
\end{enumerate}
\end{lemma}

\begin{proof}(a).  As in the proof of Lemma~\ref{P:exsign} , let  $x=2\delta\sqrt{a b c d - 4 (a + b)^2}$. Then, cross-multiplying,  the required identity is
$E=0$, where 
\begin{align*}
E=&(ad+bc+x)\big(16(a+b)^2 +  (a - c)^2 (a - d)^2\big)\\
&-(a + b)^2(ab+cd+x)\big( 16 + (a-c)^2 \big)+(a-c)^2(ad+bc+x)(ab+cd+x).
\end{align*}
Expanding and using $d=a+c-b$, one has 
\[
E=(a - c)^2\big(4( 4(a+ b)^2 - abcd) + x^2\big).
\]
Then replacing $x^2$ by $4(a b c d - 4 (a + b)^2)$  gives
 $E=0$, as required.

(b).  From the definitions, since $(a-c)^2=(a-c)(d-b)=ad+bc-a b - c d$,
\begin{align*}
(a + b)^2 (16 + (a - 
       c)^2) \Sigma &- (16 (a + b)^2 + (b - c)^2 (a - c)^2 ) T\\
        &=8(a + b)^2 (ad+bc-a b - c d)=8 (a + b)^2 (a - c)^2.
 \end{align*}
Part (b) follows by applying part (a). 
\end{proof}

\begin{remark}\label{R:exposi} As $ \Sigma$ and  $T$ are positive, Lemma~\ref{L:exsum}(a) gives $\Sigma<T$.
Furthermore, as $T-\Sigma<T$,  Lemma~\ref{L:exsum}(a) gives $\Sigma<8\frac{(a+b)^2}{(a-c)^2} $, and Lemma~\ref{L:exsum}(b) gives 
$2T< (a + b)^2 \frac{\Sigma -8}{\Sigma}<(a + b)^2$. In particular, $\Sigma >8$.
\end{remark}

\begin{remark} From \eqref{E:SigmaT1},
\begin{align}
K_O&=(a+b)+\frac18(a-c)\Sigma,\label{E:exKO}\\
K_A&=(a+b)-\frac{(a-c)(a-d)}{8(a + b)}T.\label{E:exKA}
\end{align}
Then by Proposition \ref{P:lambdae}, the parameter $\lambda_e$ and the exradius $r_e$ are related to $\Sigma$ by
\begin{equation}\label{E:lam}
\lambda_e\cdot  \Sigma=8\frac{(a+b)^2}{(a-c)^2}=2r_e^2.
\end{equation}
Using Lemma \ref{L:exsum}(a), we can also write
\begin{equation}\label{E:lam2}
\lambda_e=\frac{T}{T-\Sigma}.
\end{equation}
\end{remark}

\begin{remark}\label{R:conc} 
We  have the non-degeneracy condition $K_B\not=0$ as otherwise $ABC$ would be colinear. Thus $K_O\not=2(a+b)$ and  \eqref{E:exKO}  gives 
$\Sigma\not=8\frac{a+b}{a-c}$. Hence, by Lemma \ref{L:exsum}(a), we have
\begin{equation}\label{E:degen}
\Sigma(T-\Sigma)\not=8T.
\end{equation}
Notice also that $OABC$ is concave if and only if $K_O>2(a+b)$, that is, from \eqref{E:SigmaT1}, when
$\Sigma>8\frac{a+b}{a-c}$.
\end{remark}

\begin{lemma}\label{L:nonselfint}  
$(c-b)T<(a-b)\Sigma$.
\end{lemma}

\begin{proof} From the assumption that the vertices $O,A,B,C$  are positively oriented and if $OABC$ is concave, the reflex angle is at $B$, we have $K_A>0$. So   \eqref{E:exKA} gives
$ 8(a+b)^2>(a-c)(a-d)T$.
 Lemma~\ref{L:exsum}(a) gives $8(a+b)^2=(a-c)^2\frac{\Sigma T}{T-\Sigma}$. Hence, as $T-\Sigma>0$ by Remark~\ref{R:exposi}, and  using $a+b=c+d$, we obtain 
$(a-c)\Sigma  > (c-b)(T-\Sigma)$.
Rearranging this gives the required result.
\end{proof}

\begin{lemma}\label{L:sigmadivide}   $\Sigma$ and $T$ both divide $8(a+b)^2$.
\end{lemma}

\begin{proof} From \ref{E:sigmaquad}, 
$\frac1{16}\alpha \Sigma^2-\beta  \Sigma+4\gamma=0$,
where
$\alpha =16 + (a-c)^2$, 
$ \beta=
  a d + bc $,
$\gamma=(a + b)^2$. 
So $\frac1{16}\alpha \Sigma^2$ is an integer, and hence $\frac1{16}(a-c)^2 \Sigma^2$ is an integer. So $\frac1{4}(a-c) \Sigma$ is an integer, and hence $\frac1{4}\alpha \Sigma$ is an integer.  Hence, as
\[
16(a + b)^2=16\gamma=(-\frac1{4}\alpha \Sigma+4\beta ) \Sigma,
\]
$\Sigma$ divides $16(a+b)^2$. So if $\Sigma$ is odd, then $\Sigma$ divides $(a+b)^2$. If $\Sigma$ is even, then as $\frac1{32}\alpha \Sigma^2-\frac12\beta  \Sigma+2\gamma=0$, so 
$\frac1{32}\alpha \Sigma^2$ is an integer, and hence $\frac1{32}(a-c)^2 \Sigma^2$ is an integer. It follows that  $\frac1{64}(a-c)^2 \Sigma^2$ is an integer. So $\frac18(a-c) \Sigma$ is an integer, and hence $\frac18\alpha \Sigma$ is an integer.   Hence, as
\[
8(a + b)^2=(-\frac18\alpha \Sigma+2\beta ) \Sigma,
\]
$\Sigma$ divides $8(a+b)^2$. 

By Lemma \ref{L:exsum}(a), we have  $8(a+b)^2(\frac1{\Sigma}-\frac1T)=(a-c)^2$. Since $8(a+b)^2\frac1{\Sigma}$ is an integer, it follows that $8(a+b)^2\frac1{T}$ is also an integer.
\end{proof}


\section{Explicit examples of extangential LEQs}\label{S:extanleqs}

In this section we exhibit non-kite extangential LEQs in the three cases with $(\Sigma,T)$ equal to $(9,18),(18,50)$ and $(45,50)$ respectively. As before, let us define $h:=\frac{a+b}{a-c}$, so $h=\sqrt{\frac{\Sigma T} {8(T-\Sigma)}}$,  by Lemma~\ref{L:exsum}(a). Note that $h>1$, but as we will see, $h$ may fail to be an integer. 

We have
\begin{equation}\label{E:b}
b=(h-1)a-hc.
\end{equation}
In particular, $b>0$ gives $a>\frac{h}{h-1}c$ and since $b\le c$, we have $a\le \frac{h+1}{h-1}c$.  
From \eqref{E:sigmaquad}, $
\alpha \Sigma^2-16\beta  \Sigma+64\gamma=0$,
where
$\alpha =16 + (a-c)^2$, 
$ \beta=
  a d + bc $,
$\gamma=(a + b)^2=(c + d)^2$. Using \eqref{E:b}, substituting $d=a+b-c$ and solving for $a$ gives
\[
a= \frac{64 c h^2 - 16 c \Sigma + c \Sigma^2 \pm  
 4 \sqrt{2 c^2\Sigma (8 h^2 - \Sigma) (\Sigma-8)  -(\Sigma-8h)^2 \Sigma^2}}{(\Sigma-8h)^2}.
   \]
We claim that $\frac{64 c h^2 - 16 c \Sigma + c \Sigma^2 }{(\Sigma-8h)^2}>\frac{h+1}{h-1}c$.
Indeed, cross multiplying and simplifying, the claim is $(8 h^2 - \Sigma) (\Sigma-8)>0$, which is true since $\Sigma<8h^2$ by Remark~\ref{R:exposi} and $\Sigma>8$, also by Remark~\ref{R:exposi}.
Hence, since   $a\le \frac{h+1}{h-1}c$, it follows that
\begin{equation}\label{E:a}
a= \frac{64 c h^2 - 16 c \Sigma + c \Sigma^2 - 
 4 \sqrt{2 c^2\Sigma (8 h^2 - \Sigma) (\Sigma-8)  -(\Sigma-8h)^2 \Sigma^2}}{(\Sigma-8h)^2}.
\end{equation}

   \bigskip
We first classify the extangential LEQs with $(\Sigma,T)=(9,18)$. 
Note that this is one of two cases in Theorem~\ref{T:main}(a).
Suppose we have an extangential LEQ with $a>c\ge b$ and  $\Sigma=9,T=18$. So $h=\sqrt{\frac{\Sigma T} {8(T-\Sigma)}}=\frac32$,
and from \eqref{E:b},
 \begin{equation}\label{E:918}
 b=\frac12(a-3c).
 \end{equation}
Hence $a>3c$ and since $b\le c$, we have $a\le 5c$. From \eqref{E:a},
\begin{equation}\label{E:918a}
 a=9 c- 4 \sqrt{2 c^2-9}.
\end{equation}
Working modulo 3, the fact that $2 c^2-9$ is a square gives us that $c$ is divisible by 3, say $c=3v$. Then $v$ satisfies the negative Pell equation
\begin{equation}\label{P1}
u^2-2 v^2 = -1,
\end{equation}
for some positive integer $u$. Then  
\eqref{E:918} and \eqref{E:918a} give 
\begin{equation}\label{E:abcd}
(a,b,c,d)=3(9v-4u,3v-2u,v,11v-6u).
\end{equation}
It is well known that the solutions $(u_j, v_j)$ to \eqref{P1} are given recursively by
\begin{equation}\label{E:recur}
u_{j+1}=3 u_j+4v_j,\qquad
v_{j+1}=2 u_j+3v_j,
\end{equation}
with initial values $(u_1,v_1)=(1,1)$. 


We now define the vertices of our quadrilaterals. 
Let 
\begin{align*}
A_{j}&= \frac32\left( 9u_j -8v_j  -7 ,9u_j -8v_j  +7 \right),\\
B_j&=6\left( 3u_j -3v_j-2,3u_j-3v_j +2\right)\\
C_j&= \frac32\left( 11u_j-12v_j -7, 11u_j-12v_j+7\right).
\end{align*}
Note that $A_j,B_j,C_j$ are lattice points as $u_j,v_j$ are odd, as one can see from the recursive formula \eqref{E:recur}.
We will consider the quadrilateral $OA_jB_jC_{j}$. Let $a_j,b_j,c_j,d_j$ denote the lengths of the sides $OA_j,A_jB_j,B_jC_j,C_jO_j$ respectively. 
The distance $a_j$ is  given by
\begin{align*}
a_j^2&= \frac9{4}((9u_j -8v_j  -7)^2+(9u_j -8v_j  +7)^2)=\frac9{2}(81u_j ^2-144u_jv_j+64v_j ^2+49)\\
&=\frac9{2}(81u_j ^2-144u_jv_j+64v_j ^2+49(2v_j^2-u_j^2))\quad\text{(by \eqref{P1})}\\
&=\frac9{2}(32u_j ^2-144u_jv_j+162v_j ^2)\\
&=9(9v_j-4u_j)^2.
\end{align*}
So $a_j= 3(9v_j-4u_j)$. Similarly,  the other side lengths are as follows:
\[
b_j= 3(3v_j-2u_j),\qquad
c_j=3v_j,\qquad
d_j=3(11v_j-6u_j),
\]
as anticipated by \eqref{E:abcd}. So  the perimeter of $OA_jB_jC_{j}$ is 
$a_j+b_j+c_j+d_j=
36 (2v_j -u_j)$.

The area of $OA_jB_j$ is
$\frac12\Vert \overrightarrow{OA_j}\times\overrightarrow{OB_j}\Vert
=9(5v_j-3u_j )$. 
The area of $OB_jC_j$ is
$\frac12\Vert \overrightarrow{OB_j}\times\overrightarrow{OC_j}\Vert
=9(3v_j-u_j )$. 
Notice in passing that as the signed areas of $OA_jB_j$ and  $OB_jC_{j}$ are both positive, as one can see recursively using \eqref{E:recur}, so $OA_jB_jC_{j}$ has no self-intersection. The area of $OA_jB_jC_{j}$ is
$9(5v_j-3u_j )+9(3v_j-u_j  )=36(2v_j-u_j)$, so $OA_jB_jC_{j}$ has equal area and perimeter, i.e., it is equable.
Thus $OA_jB_jC_{j}$ is a LEQ.
Furthermore, $OA_jB_jC_{j}$ is extangential because $a_j+b_j=c_j+d_j$.

The first member of this family, corresponding to the initial condition $(u_1,v_1)=(1,1)$, is the kite with side lengths $15,3,3,15$ shown in Figure~\ref{F:315}. The vertices and side lengths of the next four members of this family are given in Table~\ref{F:extanfam}.
The case $(u_2,v_2)=(7,5)$ of the family is shown in Figure~\ref{F:triex}.

\begin{table}
\begin{tabular}{c|c||c|c|c|c|c|c|c}
  \hline
   $u_j$&$v_j$&$A_j$ & $B_j$ & $C_{j}$ & $a_{j}$  & $b_j$& $c_{j}$ & $d_{j}$  \\\hline
7 & 5 & (24,45) & (24,48) & (15,36) & 51 & 3 & 15 & 39\\
41 & 29 & (195,216) & (204,228) & (144,165) & 291 & 15 & 87 & 219 \\
239 & 169 & (1188,1209) & (1248,1272) & (891,912) & 1695 & 87 & 507 & 1275  \\
1393 & 985 & (6975,6996) & (7332,7356) & (5244,5265) & 9879 & 507 & 2955 & 7431 
\\\hline
 \end{tabular}
\bigskip
\caption{Four members of the $(\Sigma,T)=(9,18)$ family}\label{F:extanfam}
\end{table}


\begin{figure}
\definecolor{qqqqff}{rgb}{0,0,1}
\definecolor{ttzzqq}{rgb}{0.2,0.6,0}
\begin{tikzpicture}[scale=.17][line cap=round,line join=round,>=triangle 45,x=1.0cm,y=1.0cm]
\begin{axis}[
x=1cm,y=1cm,
axis lines=middle,
grid style=dashed,
ymajorgrids=true,
xmajorgrids=true,
ymin=-0.5,
ymax=60.5,
xmin=-0.5,
xmax=33.5,
ytick={0,4,...,60},
xtick={0,4,...,32},]
\draw[color=ttzzqq,fill=ttzzqq,fill opacity=0.1] (0,0) --(24,45) --  (24,48) --  (15,36) -- cycle;
\draw[line width=3pt,color=qqqqff]   (0,0) --(24,45) --  (24,48) --  (15,36)  -- cycle;
\draw[line width=1pt,color=qqqqff]   (0,0)-- (15*1.7,36*1.7);
\draw[line width=1pt,color=qqqqff]   (0,0)-- (32,60);
\draw[line width=1pt,color=qqqqff]   (24,45)-- (24,60);
\draw[line width=1pt,color=qqqqff]   (15,36)-- (15+9*2.7,36+12*2.7);
\draw (0,0) node[anchor=south east] {$O$};
\draw[dashed,line width=.5pt]   (0,0)-- (24,48);
\draw[dashed,line width=.5pt]   (24,45)-- (15,36);
\draw [fill=black] (27,57) circle (4pt);
\draw [fill=black] (39/2,81/2) circle (4pt);
\draw [fill=black] (12,24) circle (4pt);
\draw[line width=.5pt] (27,57) circle (3);
\draw[dotted,line width=3pt,color=qqqqff]   (27,57)-- (12,24);
\end{axis}
\end{tikzpicture}\caption{The case $j=2$ of the $(\Sigma,T)=(9,18)$ family}\label{F:triex}
\end{figure}
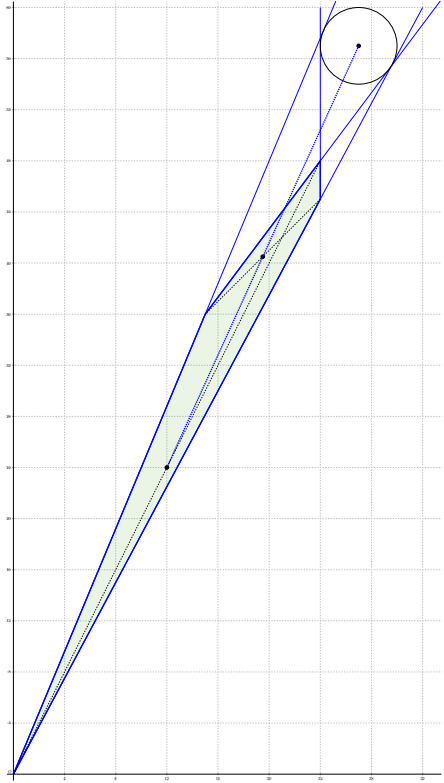

The exradius of $OA_jB_jC_{j}$ is 
$r_{e,j}=\frac{K(OA_jB_jC_j)}{a_j-c_j}
$. Substituting the values  gives  $r_{e,j}= 3$ for each $j$.
Then by Proposition~\ref{P:excen}, the 
excenter is
$I_{e,j}=\frac{3}2\,\frac{a_jC_j+d_jA_j}{ K(OA_jC_j)}$,
which simplifies to
$15(1,-1)+21(u_j+7v_j)(1,1)$. In particular, the excenter of each family member  is a lattice point.
By Proposition~\ref{P:lambdae}, the coordinate $\lambda_{e,j}$ of the excenter is
$\lambda_{e,j}=\frac{3(a_j+b_j)}{ 2 K(OA_jC_j)-3(a_j-c_j)}$. Substituting the values  gives
$\lambda_{e,j}=2$ for each $j$.

\begin{remark}\label{R:convex1}
By Remark~\ref{R:conc}, an extangential LEQ $OABC$ is concave if and only if 
$\Sigma>8\frac{a+b}{a-c}$; that is, $\Sigma>8h$. For the above  family, with $(\Sigma,T)=(9,18)$, we have 
$h=\frac32$. So all members of this family are convex.
\end{remark}

\bigskip
Now suppose we have an extangential LEQ with $a>c\ge b$ and  $\Sigma=18,T=50$. 
Note that this is the other case in Theorem~\ref{T:main}(a). 
So $h=\sqrt{\frac{\Sigma T} {8(T-\Sigma)}}=\frac{15}8$ and \eqref{E:b} gives
 \begin{equation}\label{E:1850}
 b=\frac18(7a-15c).
 \end{equation}
Hence $a>15/7c$ and since $b\le c$, we have $a\le 23c/7$. From \eqref{E:a}, 
\begin{equation}\label{E:1850a}
 a=29 c- 12 \sqrt{5 c^2-4}.
\end{equation}
So $c$ satisfies the Pell-like equation
\begin{equation}\label{P2}
u^2-5 c^2 = -4,
\end{equation}
for some positive integer $u$. Then  
\eqref{E:1850} and \eqref{E:1850a} give 
\begin{equation}\label{E:abcd2}
(a,b,c,d)=(29c-12u,\frac12(47c-21u),c,\frac12(103c-45u)).
\end{equation}

From \cite{Lind}, the solutions to \eqref{P2} are
$(u_j,c_j)=(L_{2j-1},F_{2j-1})$,
where $L_{j}$ is the \mbox{$j$-th} Lucas number and $F_{j}$ is the $j$-th Fibonacci number. (Recall that the Lucas and Fibonacci numbers satisfy the same recurrence relation but with different initial conditions: $F_1=F_2=1$ while $L_1=1,L_2=3$). Hence \eqref{E:abcd2} gives the potential solutions
\begin{equation}\label{E:abcdi}
(a_j,b_j,c_j,d_j)=F_{2j-1}\big(29,\frac{47}2,1,\frac{103}2\big) -L_{2j-1}\big(12,\frac{21}2,0,\frac{45}2\big).
\end{equation}

From Lemma \ref{L:nonselfint}, 
$T(c-b)< \Sigma(a-b)$.
So $50(c-b)<18(a-b)$ and thus from \eqref{E:1850} and \eqref{E:1850a},
$51 c > 23 u$,
which is 
\begin{equation}\label{E:FL}
51 F_{2j-1} > 23L_{2j-1}.
\end{equation}
Now $F_6=8,L_6=18$ and $F_7=13,L_7=29$, and thus $51 F_{6} < 23L_{6}$ and $51 F_{7} < 23L_{7}$. It follows from the Lucas and Fibonacci recurrence relation that \eqref{E:FL} only holds for $j=1,2,3$. For $i=1,2$ one finds using \eqref{E:abcdi} that $a_j<d_j$, contrary to our hypothesis. So the only possibility is $j=3$, which gives the solution $(a,b,c,d)=(13,2,5,10)$, which is the concave LEQ shown on the right of Figure~\ref{F:extans}.

\bigskip
Now suppose we have an extangential LEQ with $a>c\ge b$ and  $\Sigma=45,T=50$.  
Note that this is the case $m=3$ in Theorem~\ref{T:main}(b).
So $
h=\sqrt{\frac{\Sigma T} {8(T-\Sigma)}}=\frac{15}2$ and \eqref{E:b} gives
\begin{equation}\label{E:4550}
 b=\frac12(13a-15c).
 \end{equation}
Hence $a>15/13c$ and since $b\le c$, we have $a\le 17c/13$. From \eqref{E:a}, 
\begin{equation}\label{E:4550a}
 a=\frac15(109 c- 12 \sqrt{74 c^2-25}).
\end{equation}
So $c$ satisfies the Pell-like equation
\begin{equation}\label{P22}
W^2-74 c^2 = -25,
\end{equation}
for some positive integer $W$. The solutions to this equation are not readily enumerated, and moreover, some solutions do not result in LEQs. For example, for the solution $W=5927, c=689$, one obtains the non-integer value $a=\frac{3977}5$ from \eqref{E:4550a}. We will restrict ourselves to constructing a particular infinite family of LEQs for which $c$ is divisible by $5$. Set $w=5u,c=5v$, so that \eqref{E:4550a} gives the negative Pell equation $u^2-74 v^2 = -1$. Let us denote its solutions $(u_j,v_j)$, where $(u_1,v_1)=(43,5)$ and $(u_2,v_2)=(318157, 36985)$.
So
\begin{equation}\label{P3}
u_j^2-74 v_j^2 = -1.
\end{equation}
The solutions $(u_j,v_j)$ are well known; see entries A228546 and A228547 in \cite{OEIS}. In particular, they satisfy the second order recurrence relation
\begin{equation}\label{E:exrec}
X_{j+2}=7398X_{j+1}-X_{j}.
\end{equation}
 It follows from this recurrence relation and the initial conditions that $u_j$ is divisible by 43 and $v_j$ is divisible by 5 for all $j$. Let $x_j=u_j/43,y_j=v_j/5$, so 
\begin{equation}\label{P32}
43^2 x_j^2-74\cdot25\, y_j^2 = -1,
\end{equation}
and $(x_1,y_1)=(1,1)$ and $(x_2,y_2)=(7399, 7397)$. 
Note that $(x_j,y_j)$ also satisfies the recurrence relation \eqref{E:exrec}. We have $c_j=25y_j$ and from \eqref{E:4550a}, $a_j=\frac15(109 c_j- 60u_j)=-12\cdot 43 x_j+545y_j$. Then from \eqref{E:4550}, $b_j=\frac12(13a_j-15c_j)=-3354 x_j+3355 y_j$. Thus, as $d_j=a_j+b_j-c_j$, 
\begin{equation}\label{E:abcd3}
(a_j,b_j,c_j,d_j)=-x_j(516,3354 ,0,3870 )+y_j(545,3355 ,25, 3875 ).
\end{equation}
In particular, $(a_j,b_j,c_j,d_j)$ are determined by the recurrence relation \eqref{E:exrec} with $(a_1,b_1,c_1,d_1)=(29,1,25,5)$ and $(a_2,b_2,c_2,d_2)=(213481  , 689  , 184925  , 29245)$. The first three members of this family are shown in Table~\ref{F:extanfam2a}.

\begin{table}
\begin{tabular}{c|c|c|c}
  \hline
   $a_{j}$  & $b_j$& $c_{j}$ & $d_{j}$  \\\hline
     29  & 1 &25  & 5  \\
   213481  & 689  & 184925  & 29245    \\
1579332409 &5097221&1368075125& 216354505
\\\hline
 \end{tabular}
\bigskip
\caption{Side lengths of three members of the $(\Sigma,T)=(45,50)$ family}\label{F:extanfam2a}
\end{table}

We have shown above that if an extangential  LEQ with side lengths $a,b,c,d$ has $(\Sigma,T)=(45,50)$ and $c$ is divisible by 5, then $(a,b,c,d)=(a_j,b_j,c_j,d_j)$ for some $j$, where $(a_j,b_j,c_j,d_j)$ is given by \eqref{E:abcd3}. 
We now show that conversely, for each $j$, the 4-tuples  $(a_j,b_j,c_j,d_j)$ given by \eqref{E:abcd3} are realised by the  side lengths of an extangential  LEQ. 
We mention in passing that the difficulty in determining suitable vertices is two-fold: firstly, the quadrilaterals in this family grow so fast that we only had three examples to base our study on, and secondly, the pattern of the vertex coordinates is considerably more complicated than in the $(\Sigma,T)=(8,18)$ family exhibited above.  

To define the vertices, we will employ the following two first-order recurrence relations in two variables
\begin{align}
x_{j+1} &= 78 x_j + 25 y_j,\qquad
y_{j+1} = 25 x_j + 8 y_j,\label{E:rec1}\\
x_{j+1} &= 68 x_j + 35 y_j,\qquad
y_{j+1} = 35 x_j + 18 y_j,\label{E:rec2}
\end{align}
under various initial conditions. 
We identify the 2-tuple $(x,y)\in\Z^2$ with the Gaussian integer $x+yi\in \mathbb C$, and we use the notations interchangeably, according to convenience. Consider the following families,  for all $j\ge 1$:
\begin{enumerate}
\item $z_{a,j}=x_{a,j}+y_{a,j}i$ satisfies \eqref{E:rec1} with $z_{a,1}=5+2i$,
\item $z_{b,j}=x_{b,j}+y_{b,j}i$ satisfies \eqref{E:rec1}  with $z_{b,1}=i$,
\item $z_{c,j}=x_{c,j}+y_{c,j}i$ satisfies \eqref{E:rec2}  with $z_{c,1}=2+i$,
\item $z_{d,j}=x_{d,j}+y_{d,j}i$ satisfies \eqref{E:rec2}  with $z_{d,1}=1$.
\end{enumerate}
Let $\rho: \Z^2\to\Z^2$ denote the reflection in the line $y=x$, so $\rho(x,y)=(y,x)$ or equivalently $\rho(z)=i\bar z$, where $\bar z$ denotes the complex conjugate of $z$. Let  $\rho^j$ denote the $j$-th iterate of  $\rho$ under composition, so $\rho^{j}=\rho$ if $j$ is odd and $\rho^{j}=\id$ otherwise. 
We now define the vertices. 
Set 
\begin{align*}
A_j&=\rho^{j+1}z_{a,j}^2,\quad
B_j=A_j+\rho^{j+1}z_{b,j}^2,\quad
C_j=5\rho^{j}z_{d,j}^2,\quad
B'_j=C_j+5\rho^{j}z_{c,j}^2.
\end{align*}
The first three members of this family are shown in Table~\ref{F:extanfam2b}; the LEQ given in the first row is shown on the left of Figure~\ref{F:extans}.

\begin{table}
\begin{tabular}{c|c|c}
  \hline
   $A_j$ & $B_j$ & $C_{j}$   \\\hline
(21, 20)&(20, 20)&(0, 5)     \\
(124080,173719)&(124480,174280)&(16995,23800)        \\
(1285155641, 917968320) &  (1289303420, 920931020) & (176054900, 125753505)
 \\\hline
 \end{tabular}
\bigskip
\caption{Vertices of  three members of the $(\Sigma,T)=(45,50)$ family}\label{F:extanfam2b}
\end{table}

We will establish the following properties:
\begin{enumerate}
\item[(i)] $B_j=B'_j$ for all $j$,
\item[(ii)] The  areas $K_{A_j},K_{C_j}$ of triangles $OA_jB_j$ and $B_jC_jO$ are positive (and hence $OA_jB_jC_j$ is a non-self-intersecting, positively oriented quadrilateral),
\item[(iii)] $OA_jB_jC_j$ has the side lengths $a_j,b_j,c_j,d_j$ given by \eqref{E:abcd3},
\item[(iv)]   $a_j+b_j=c_j+d_j$; i.e., $OA_jB_jC_j$ is extangential,
\item[(v)]   $a_j+b_j+c_j+d_j=K_{A_j}+K_{C_j}$; i.e., $OA_jB_jC_j$ is equable.
\end{enumerate}

The proofs of these properties will use the following technical results.

\begin{lemma}\label{L:rec}
Suppose the sequence $z_j=(x_j,y_j)$ satisfies either \eqref{E:rec1} or \eqref{E:rec2}. Then for all $j\ge 1$,
\begin{enumerate}
\item $z_{j+2}=86z_{j+1}+z_{j}$,
\item $|z^2_{j+2}|=7398\,|z^2_{j+1}|-|z^2_{j}|$ (i.e.,  $|z^2_j|$ satisfies \eqref{E:exrec}),
\item $
z^2_{a,j+2}=7398\,z^2_{a,j+1}-z^2_{a,j}+ (-1)^j (17500-24500i)$,
\item $z^2_{b,j+2}=7398 \,z^2_{b,j+1}-z^2_{b,j}+ (-1)^j (2500-3500i)$,
\item $z^2_{c,j+2}=7398\,z^2_{c,j+1}-z^2_{c,j}- (-1)^j (700-500i)$,
\item $z^2_{d,j+2}=7398\,z^2_{d,j+1}-z^2_{d,j}- (-1)^j (4900-3500i)$.
  \end{enumerate}
\end{lemma}

\begin{proof}(a). It is easy to see that if $z_j$ (regarded as a 2-tuple) satisfies the relation $z_{j+1}=Mz_j$, where
\[
M=\begin{pmatrix}
\alpha&\beta\\
\gamma&\delta
\end{pmatrix},
\]
then $z_{j+2}=\tr(M)z_{j+1}-\det(M)z_{j}$. 
The result follows as $68+18=78+8=86$ and $68\cdot18-35^2=78\cdot 8-25^2=-1$.

(b). From (a) we have $|z^2_{j+2}|=(86x_{j+1}+x_{j})^2+(86y_{j+1}+y_{j})^2$, so
\[
|z^2_{j+2}|=86^2\,|z^2_{j+1}| + |z^2_{j}|+172(x_{j+1}x_{j}+y_jy_{j}).
\]
Assume first that \eqref{E:rec1} holds. Then using \eqref{E:rec1} twice,
\begin{align*}
|z^2_{j+1}| - |z^2_{j}| &=(x_{j+1}^2+y_{j+1}^2)-(x_j^2+y_j^2)\\
&=(68^2+35^2-1) x_j^2 + (35^2+18^2-1) y_j^2 +2\cdot(68+18)\cdot35 x_jy_j\\
&=86(68 x_j^2 + 70 x_jy_j +18 y_j^2)\\
&=86(x_{j+1}x_j +y_{j+1} y_j),
\end{align*}
so, from above, 
$|z^2_{j+2}|=(86^2+2)|z^2_{j+1}| - |z^2_{j}|=7398\,|z^2_{j+1}| - |z^2_{j}|$, as required.
A similar argument applies when \eqref{E:rec2} holds.

(c). From (a) we have $z^2_{a,j+2}=7396\,z^2_{a,j+1}+172z_{a,j+1}z_{a,j} +z^2_{a,j}$. So we are required to show that
$172z_{a,j+1}z_{a,j} +2z^2_{a,j}=2z^2_{a,j+1}+ (-1)^j (17500-24500i)$, or equivalently, using \eqref{E:rec1},
$5 x_{a,j}^2 - 14 x_{a,j} y_{a,j} - 5 y_{a,j}^2= 35 (-1)^j $,
for all $j$.  In matrix notation, the condition is $z_jQz_j^t=35 (-1)^j $, where $z_j=(x_{a,j},y_{a,j})$, $z^t$ denotes the transpose of $z$, and
\[
Q=\begin{pmatrix}
5&-7\\
-7&-14
\end{pmatrix}.\]
One verifies readily that this condition holds for $j=1$, where $z_1=(5,2)$. Now using \eqref{E:rec1} and setting $
M=\begin{pmatrix}
78&25\\
25&8
\end{pmatrix}
$,
we have $z_{j+1}Qz_{j+1}^t=(z_jM)Q(Mz_j^t)=-z_jQz_j^t$, since
$MQM = -Q$. So the required result follows by induction. This proves (c).

Part (d) is proven in the same manner; only the initial condition is different. For parts (e) and (f), one repeats the argument using the matrices
\[
Q=\begin{pmatrix}
7&-5\\
-5&-7
\end{pmatrix},\qquad M=\begin{pmatrix}
68&35\\
35&18
\end{pmatrix}.
\]
Once again, the argument works because $MQM = -Q$.
\end{proof}

\begin{remark} It is well known that if a sequence $r_j$ satisfies a second-order homogenous recurrence relation, then
$r_j^2$  satisfies a third-order recurrence relation, but in certain exceptional circumstances,  $r_j^2$ may satisfy a second-order recurrence relation, which is typically non-homogeneous \cite{BS}. In this respect, Lemma~\ref{L:rec} is perhaps somewhat surprising.
\end{remark}

(i). We have $B_1=(5+2i)^2+i^2 = 20+20i$ and $B'_1=5\rho((2+i)^2+1^2)=B_1$, while 
$B_2=\rho((78\cdot 5+25\cdot 2+(25\cdot 5+8\cdot 2)i)^2+(25 +8 i)^2)= 124480 +  174280 i$,
and $B'_2=5((68\cdot 2+35+(35\cdot 2+18)i)^2+(68 +35 i)^2)= B_2$.
From Lemma~\ref{L:rec}(c) and (d), we have for all $j$,
\begin{align*}
B_{j+2}&=7398 \rho (B_{j+1})-B_{j}+(-1)^j \rho^{j+1}(17500-24500i+ 2500-3500i)\\
&=7398 \rho (B_{j+1})-B_{j}+(-1)^j \rho^{j+1}(20000-28000i),
\end{align*}
and from Lemma~\ref{L:rec}(e) and (f), we have 
\begin{align*}
B'_{j+2}&=7398 \rho (B'_{j+1})-B_{j}-(-1)^j 5\rho^{j}(700-500i+ 4900-3500i)\\
&=7398 \rho (B_{j+1})-B_{j}-(-1)^j \rho^j(28000-20000i)=B_{j+2}.
\end{align*}
Hence $B_{j}=B'_{j}$ for all $j$.

(ii). Regarding $A_j$ and $B_j$ as complex numbers, one has $K_{A_j}=\frac{i}4(A_j\overline{B_j}-\overline{A_j}B_j)$. Hence when $j$ is odd,
\[
K_{A_j}=\frac{i}4(z^2_{a,j}\bar z^2_{b,j}-\bar z^2_{a,j}z^2_{b,j})=( x_{a,j} y_{b,j}-x_{b,j} y_{a,j} ) (x_{a,j} x_{b,j} + y_{a,j} y_{b,j}),
\]
while when $j$ is even,
\begin{align*}
K_{A_j}&=\frac{i}4(\rho(z^2_{a,j})\rho(\bar z^2_{b,j})-\rho(\bar z^2_{a,j})\rho(z^2_{b,j})=\frac{i}4(\bar z^2_{a,j}z^{2}_{b,j}-z^{2}_{a,j}\bar z^2_{b,j})\\
&=-( x_{a,j} y_{b,j}-x_{b,j} y_{a,j} ) (x_{a,j} x_{b,j} + y_{a,j} y_{b,j}).
\end{align*}
It is easy to see that $x_{a,j} x_{b,j} + y_{a,j} y_{b,j} > 0$.
Moreover, $x_{a,j} y_{b,j}-x_{b,j} y_{a,j}$ is the area $K_j$ of the parallelogram $P_j$ spanned by $z_{a,j} $ and $z_{b,j}$. So we are required to show that $(-1)^{j+1}K_j>0$ for all $j$. For $j=1$, one has $K_1=5$. Furthermore, $P_{j+1}$ is the image of $P_j$ under the linear transformation with matrix 
$M=\begin{pmatrix}
78&25\\
25&8
\end{pmatrix}$, which has determinant $-1$. Hence $K_j=(-1)^{j+1}5$, giving the required result.

A similar reasoning applies for $K_{C_j}$. Here 
\[
K_{C_j}=(-1)^j
25\,\frac{i}4(z^2_{c,j}\bar z^2_{d,j}-\bar z^2_{c,j}z^2_{d,j})=(-1)^j
25\,( x_{c,j} y_{d,j}-x_{d,j} y_{c,j} ) (x_{c,j} x_{d,j} + y_{c,j} y_{d,j}).
\]
Setting $K_j=x_{c,j} y_{d,j}-x_{d,j} y_{c,j}$, one argues as before, using $K_{1}=-4$ and the fact that $\begin{pmatrix}
68&35\\
35&18
\end{pmatrix}$ has determinant $-1$.

(iii). As observed above, the lengths $(a_j,b_j,c_j,d_j)$ are determined by the recurrence relation \eqref{E:exrec} with $(a_1,b_1,c_1,d_1)=(29,1,25,5)$ and $(a_2,b_2,c_2,d_2)=(213481  , 689  , 184925  , 29245)$.
The side $OA_j$ has length $|z_{a,j}^2|$. Hence, as  $|z_{a,1}^2|=29$ and from \eqref{E:rec1},
$|z_{a,2}^2|=(78\cdot 5+25\cdot 2)^2+(25\cdot 5+8\cdot 2)^2=213481$, so Lemma~\ref{L:rec}(b) gives $|OA_j|=a_j$ for all $j$. By the same reasoning, the sides $A_jB_j, B_jC_j, C_jO$ have lengths $b_j,c_j,d_j$ respectively.

(iv). We have $a_j+b_j=|z_{a,j}^2|+|z_{b,j}^2|$ and $c_j+d_j=(|z_{c,j}^2|+|z_{d,j}^2|)$. So  $a_j+b_j$ and $c_j+d_j$ satisfy the recurrence relation \eqref{E:exrec}, by Lemma~\ref{L:rec}(b). Hence, since $a_1+b_1=30=c_1+d_1$ and $a_2+b_2=214170=c_2+d_2$, we have $a_j+b_j=c_j+d_j$ for all $j$, as required.

(v). As we saw above, $a_j+b_j+c_j+d_j$ satisfies recurrence relation \eqref{E:exrec} and $a_1+b_1+c_1+d_1=60$ and 
$a_2+b_2+c_2+d_2=428340$. From above,  \[
K_{A_j}=(-1)^{j+1}\frac{i}4(z^2_{a,j}\bar z^2_{b,j}-\bar z^2_{a,j}z^2_{b,j}),\qquad K_{C_j}=(-1)^j
25\,\frac{i}4(z^2_{c,j}\bar z^2_{d,j}-\bar z^2_{c,j}z^2_{d,j}).
\]
One finds easily that $K_{A_1}+K_{C_1}=60$ and $K_{A_1}+K_{C_1}=428340$. So it remains to show that $K_{A_j}+K_{C_j}$ satisfies  \eqref{E:exrec}. In fact, we will show that $K_{A_j}$ and $K_{C_j}$ both satisfy  \eqref{E:exrec}.
As we saw in the proof of property (ii),
\[
K_{A_j}=(-1)^{j+1}( x_{a,j} y_{b,j}-x_{b,j} y_{a,j} ) (x_{a,j} x_{b,j} + y_{a,j} y_{b,j}),
\]
and $(-1)^{j+1}( x_{a,j} y_{b,j}-x_{b,j} y_{a,j} )=5$. So we will show that $r_j:=x_{a,j} x_{b,j} + y_{a,j} y_{b,j}$ satisfies  \eqref{E:exrec}. Using \eqref{E:rec1} three times, we have
\begin{align*}
r_{j+2}&=6709 x_{a,j+1} x_{b,j+1} + 2150 x_{b,j+1} y_{a,j+1} + 2150 x_{a,j+1} y_{b,j+1} + 689 y_{a,j+1} y_{b,j+1}\\
&=49633181 x_{a,j} x_{b,j}  + 15905700 x_{b,j} y_{a,j}  + 15905700 x_{a,j}  y_{b,j} + 5097221 y_{a,j}  y_{b,j}\\
&=7398r_{j+1}-r_j,
\end{align*}
as required. The proof that $K_{C_j}$ satisfies  \eqref{E:exrec} is obtained in exactly the same manner, using \eqref{E:rec2}.

\begin{remark} We have just determined all extangential  LEQs with side lengths $a,b,c,d$ that have $(\Sigma,T)=(45,50)$ and $c$ is divisible by 5. However, there are extangential  LEQs with $(\Sigma,T)=(45,50)$ for which $c$ is not divisible by 5. Here are two examples: the LEQ with vertices $A=(6300, 4505),\ B=(6320, 4520),\ C=(861, 620)$ and side lengths $7745,25,6709,1061$, and the LEQ with vertices $A=(33303495,46624900),\   B=(33410980,46775380), \ C=(4562280,6387199)$ and side lengths $57297505,\ 184925,\ 49633181,\ 7849249$.
\end{remark}

\begin{remark}
By Remark~\ref{R:conc}, an extangential LEQ $OABC$ is concave if and only if 
$\Sigma>8\frac{a+b}{a-c}=8h$. As we saw above, for extangential LEQs with $(\Sigma,T)=(45,50)$, we have 
$h=\frac{15}2$. So all extangential LEQs with $(\Sigma,T)=(45,50)$ are convex.
\end{remark}


\section{Theorem \ref{T:main} from Theorem \ref{T:nt}}\label{S:thms}

As in the previous section, let  $OABC$ be a non-kite extangential LEQ with sides $a,b,c,d$ and with its excircle outside the vertex $B$. As explained in Remark~\ref{R:a>b}, we may assume  $a=\max\{a,b,c,d\}$ and $b=\min\{a,b,c,d\}$.
We introduce some new variables: 

\begin{definition}\label{D:xyz} Let
$x=a+b,y=a-c,z=c-b$. By Lemma~\ref{L:sigmadivide}, $\Sigma$ and $T$ both divide $8x^2$. Define $k$ by $8x^2=kT$. 
By Remark~\ref{R:exposi}, $\Sigma<T$. Let $\Sigma'=T-\Sigma$.
 \end{definition}


From our previous observations we now extract four important consequences.
\begin{enumerate}
\item By Lemma~\ref{L:exsum}(b), $16\Sigma x^2  = k x^2 (\Sigma -8)- 8x^2z^2$, so $\Sigma   = \frac{8(k+z^2)}{k -16}$. In particular, $k>16$.
\item By Lemma~\ref{L:exsum}(a), $8x^2\Sigma'= y^2\Sigma T =\frac1k8x^2y^2\Sigma$, so $\Sigma'=\frac{y^2\Sigma}k$ and thus $T=\Sigma+\Sigma'=\frac{(k+y^2)\Sigma}k$.
\item From the definitions, $x=\sqrt{\frac{k(\Sigma+\Sigma')}{ 8}}$.
\item By Lemma \ref{L:nonselfint},  
$(c-b)T<(a-b)\Sigma$, so by (b),  $z(k+y^2)<k(y+z)$ and hence $yz<k$.
\end{enumerate}
It follows that the hypotheses of Theorem \ref{T:nt} are satisfied. So one of the following holds:
\begin{enumerate}
\item  $(\Sigma,\Sigma')=(9,9),(12,24),(16,16),(24,12),(10,40),(40,10)$ or $(18,32)$, 
\item  $(\Sigma,\Sigma')=(5m^2,5)$ for some integer $m$ for which there exists integers $n,Y,Z$ such that $m^2-10n^2=-1$ and $ (5m^2-8) Y^2= 5+8Z^2$,
\item  $(\Sigma,\Sigma')=(m^2,1)$ for some integer $m$ for which there exists integers $n,Y,Z$ such that $m^2-2n^2=-1$ and $ (m^2-8) Y^2= 1+8Z^2$.
\end{enumerate}
In case (a), we have
\[
(\Sigma,T)=(9,18),(12,36),(16,32),(24,36),(10,50),(40,50) \ \text{or}\ (18,50).
\]
Note that in the cases $(\Sigma,T)=(12,36),(16,32),(24,36),(10,50),(40,50)$, we have
$\Sigma(T-\Sigma)=8T$, so these are all degenerate cases which are excluded by Remark~\ref{R:conc}.
Thus $(\Sigma,T)=(9,18)$ or $(18,50)$, as required.

In case (b), we have
$(\Sigma,T)=(5m^2,m^2+5)$, and in case (c), we have
$(\Sigma,T)=(m^2,m^2+1)$, as required.


\section{Proof of Theorem \ref{T:nt}}\label{S:pfs}

The proof of Theorem \ref{T:nt} will occupy us for most of the rest of this paper.
The general strategy is to analyse the different possibilities for the ratio $\Sigma'/\Sigma$.
Suppose $\Sigma'=\frac{u}v \Sigma$, where $\gcd(u,v)=1$. So hypothesis (b) gives $vy^2=uk$. Then the assumption $k>yz$ gives $vy^2>u yz$ so $vy>u z$.
From (c) we have
\begin{equation}\label{E:x0}
x=\sqrt{\frac{k(\Sigma+\Sigma')}{ 8}}=\sqrt{\frac{y^2(u+v)\Sigma}{ 8u}}. 
\end{equation}
Hypothesis (a) gives $\Sigma(vy^2- 16u)=8( u z^2+vy^2)$. 

Throughout this section, we use the following notation.

\begin{definition}
For an integer $n$, we let $f(n)$ denote the square-free part of $n$, and write $n=f(n) s^2(n)$. 
\end{definition}

Since $vy^2=uk$ and $\gcd(u,v)=1$, 
we have that $f(u) s(u)$ divides $y$, say $y=f(u) s(u)y'$.  So $vy>u z$ gives
\begin{equation}\label{E:yz}
vy'>s(u) z.
\end{equation}
  Furthermore,  $y^2=f(u)uy'^2$ and hypothesis (a) can now be rewritten as 
 \begin{equation}\label{E:a1}
 \Sigma(vf(u)y'^2- 16)=8( z^2+vf(u)y'^2),
\end{equation}
and from \eqref{E:x0} we have
 \begin{equation}\label{E:x}
x =\sqrt{\frac{y'^2  f(u)(u+v)\Sigma}{ 8}}.
\end{equation}

We split the problem up into 6 cases:
\begin{enumerate}
\item[1.]   $u$ is  odd, $v$ is even and the 2-adic order of $v$ is even.
\item[2.]   $u$ is  odd, $v$ is even and the 2-adic order of $v$ is odd.
\item[ 3.]   $u$ is  even, $v$ is odd and the 2-adic order of $u$ is even.
\item[ 4.]   $u$ is  even, $v$ is odd and the 2-adic order of $u$ is odd.
\item[5.]   $u$ and $v$ are both odd and the 2-adic order of $u+v$ is even.
\item[6.]   $u$ and $v$ are both odd and the 2-adic order of $u+v$ is odd.
\end{enumerate} 
In each case, we make several change of variables. These will be introduced as we go along, but for the convenience of the reader, we summarize the main variables in Table~\ref{T:vars}.
\begin{table}[H]
\begin{tabular}{c||c|c|c|c}
  \hline
   Case&$\Sigma$&$w$&$y$ & $z$  \\\hline
1   & $ 2f(u+v)f(u) w^2$& $f(v)s(v)w'$& $f(u) s(u)f(v)y''$&$f(v)s(v)f(u)z''$\\
2   & $2 f(u+v)f(u) w^2$& $\frac12f(v)s(v)w'$& $\frac12f(u) s(u)f(v)y''$&$\frac12f(v)s(v)f(u)z''$\\
3   & $2 f(u+v)f(u) w^2$& $f(v)s(v)w'$& $f(u) s(u)f(v)y''$&$\frac14f(v)s(v)f(u)z'$ \\
4   & $ \frac12f(u+v)f(u) w^2$& $f(v)s(v)w'$& $f(u) s(u)f(v)y''$&$\frac12f(v)s(v)f(u)z'$ \\
5   & $2 f(u+v)f(u) w^2$& $f(v)s(v)w'$& $f(u) s(u)f(v)y''$&$f(v)s(v)f(u)z'$ \\
6   & $ \frac12f(u+v)f(u) w^2$& $f(v)s(v)w'$& $f(u) s(u)f(v)y''$&$f(v)s(v)f(u)z'$ \\
 \hline
 \end{tabular}
\caption{Variable changes}\label{T:vars}
\end{table}

\bigskip
\noindent
{\bf Case 1.  Assume $u$ is  odd, $v$ is even and the 2-adic order of $v$ is even.}\ 

We will show that in this case, $(\Sigma,\Sigma')=(40,10)$ is the only possibility.

As $x$ is an integer, and as $u,u+v$ are relatively prime odd integers,
from \eqref{E:x} we can write $\Sigma= 2f(u+v)f(u) w^2$, for some $w$. 
Note $v$ divides $\Sigma$ as $v\Sigma'=u \Sigma$ and $\gcd(u,v)=1$. So $v$ divides $2w^2$, and thus as the 2-adic order of $v$ is even, $v$ divides $w^2$. Hence  $f(v)s(v)$ divides $w$. Thus, setting $w=f(v)s(v)w'$ we have $\Sigma= 2f(u+v)f(u) f(v)v w'^2$ 
and  \eqref{E:a1} gives
\begin{equation}\label{E:oee1}
f(u+v)f(u)f(v)v w'^2 (vf(u) y'^2- 16)=4( z^2+vf(u) y'^2).
\end{equation}
Thus $vf(u)$ divides $4z^2$, so $f(v)s(v)f(u)$ divides $2z$, say $2z=f(v)s(v)f(u)z'$. So \eqref{E:yz} gives $2vy' >s(u)f(v)s(v)f(u)z'$ 
and hence
\begin{equation}\label{E:oee2}
2s(v)y' > f(u)s(u)z', 
\end{equation}
and \eqref{E:oee1} gives
\begin{equation}\label{E:oee3}
f(u+v) f(v) w'^2 (vf(u) y'^2- 16)=f(v) f(u)z'^2+ 4y'^2.
\end{equation}
Hence $f(v)$ divides $4y'^2$. Since the 2-adic order of $v$ is even, $f(v)$ is odd, so $f(v)$ divides  $y'$. Let $y'=f(v)y''$.
Then \eqref{E:oee2} gives
$2f(v)s(v)y'' > f(u)s(u)z'$, 
and \eqref{E:oee3} gives
$f(u+v) w'^2 (vf^2(v)f(u) y''^2- 16)= f(u)z'^2+ 4f(v)y''^2$.
From this last equation, notice that as $v$ is even and $f(u)$ is odd, $z'$ must be even, say $z'=2z''$. Then \eqref{E:oee2} gives
\begin{equation}\label{E:oee4}
f(v)s(v)y'' > f(u)s(u)z''\quad \text{and so}\quad vf(v)y''^2>uf(u)z''^2, 
\end{equation}
and \eqref{E:oee3} gives
\begin{equation}\label{E:oee5}
f(u+v) w'^2 (vf^2(v)f(u) y''^2- 16)= 4(f(u)z''^2+ f(v)y''^2).
\end{equation}
Note that from the left-hand side of  \eqref{E:oee5}, we have
\begin{equation}\label{E:oee5i}
vf^2(v)f(u) y''^2> 16.
\end{equation}
Furthermore,  \eqref{E:oee4} and \eqref{E:oee5} give
\[
uf(u+v) w'^2 (vf^2(v)f(u) y''^2- 16)< 4(v +u)f(v)y''^2.
\]
Hence 
\begin{equation}\label{E:oee7}
 w'^2<
 \frac{4}{f(v+u)}\left(\frac1{uf(u)f(v)}+\frac1{vf(v)f(u)}\right)\left(1+\frac{16 }{vf^2(v)f(u) y''^2- 16}\right).
\end{equation}
A slightly weaker but useful consequence is
\begin{equation}\label{E:oee6}
 w'^2<
4\left(\frac1{uf(u)f(v)}+\frac1{vf(v)f(u)}\right)\left(1+\frac{16 }{vf^2(v)f(u) y''^2- 16}\right).
\end{equation}
We will use \eqref{E:oee7} and  \eqref{E:oee6} repeatedly to derive contradictions with the fact that, being a positive integer, $w'\ge 1$.
Note that \eqref{E:oee7} is only useful when we know something about $f(u+v)$, so \eqref{E:oee6} will be more commonly applied. Even so, we sometimes only have information about $uf(u)$, and not about $f(u)$, for which we use the trivial bound $f(u)\ge 1$.

\begin{remark}\label{R:ineq}
The  first twelve possible values of $v$ are shown in Table~\ref{F:even}. Notice that the sequence $vf(v)$ is not monotonically increasing in $v$.  By hypothesis, $v$ is divisible by 4, so $v\ge 4$. 
It is easy to verify that the following hold:
\begin{enumerate}
\item if $v>4$, then  $v\ge 12, vf(v)\ge 16$ and $vf^2(v)\ge 16$, 
\item if $vf^2(v)\ge16$, then either $v=16$ or $vf(v)\ge 36$ and $vf^2(v)\ge 36$,
\item if $vf^2(v)\ge 36$, then either $v=36$ or $vf(v)\ge 64$ and $vf^2(v)\ge 64$,
\item if $vf^2(v)\ge 64$, then either $v=64$ or $vf(v)\ge 100$ and $vf^2(v)\ge 100$. 
\end{enumerate}
Note also that if $u>1$, then $u\ge 3$  and $uf(u)\ge 9$. 
\begin{table}
\begin{tabular}{c||c|c|c|c|c|c|c|c|c|c|c|c}
  \hline
   $v$&4& 12& 16& 20& 28& 36& 44& 48& 52& 64& 80& 100 \\\hline
   $f(v)$&1& 3& 1& 5& 7& 1& 11& 3& 13& 1& 5& 1 \\
   $vf(v)$&4& 36& 16& 100& 196& 36& 484& 144& 676& 64& 400& 100 \\
   $vf^2(v)$&4& 108& 16& 500& 1372& 36& 5324& 432& 8788& 64& 2000& 100
\\\hline
 \end{tabular}
\bigskip
\caption{The first twelve even positive integers with even  2-adic order}\label{F:even}
\end{table}

\end{remark}

\begin{lemma}\label{L:3poss}  Either $u=1$ or $v=4$ or $y''=1$. 
\end{lemma}

\begin{proof} 
It suffices to note that if $u>1,v>4$ and $y''>1$, then \eqref{E:oee6} would give
\[
w'^2<4\left(\frac1{9}+\frac1{16}\right)\left(1+\frac{16 }{16\cdot 4- 16}\right)=\frac{25}{27}< 1. \hskip1cm \cont
\]
\end{proof} 

\begin{lemma}\label{L:ynot1case1}  $y''\not=1$. 
\end{lemma}

\begin{proof} Suppose $y''=1$. Then \eqref{E:oee4} gives $vf(v)  > uf(u)z''^2\ge uf(u)$. Also, from \eqref{E:oee5i},  $vf^2(v)f(u) > 16$. So  if $v=4$, then  $vf(v)=vf^2(v)=4$ and hence $f(u)>4$, contradicting the fact that $vf(v)>uf(u)$. Hence  $v>4$ and so, by Remark~\ref{R:ineq}, $vf^2(v)\ge 16$. By Remark~\ref{R:ineq} again, if $vf^2(v)= 16$, then $v=16$, in which case $vf^2(v)f(u) > 16$ gives  $f(u)\ge 3$, so $u\ge 3$ and $uf(u)\ge 9$.   Then \eqref{E:oee6} would give
\begin{align*}
w'^2&<4\left(\frac1{9}+\frac1{48}\right)\left(1+\frac{16 }{48 - 16}\right)=\frac{19}{24}<1. \hskip1cm \cont
\end{align*}
Thus $vf(v)> 16$ and hence $vf^2(v)\ge 36$ and $vf(v)\ge 36$, by Remark~\ref{R:ineq}. Then if $u>1$ one would  have 
$u\ge 3$ and $uf(u)\ge 9$ and \eqref{E:oee6} would give
\begin{align*}
w'^2&<4\left(\frac1{9}+\frac1{36}\right)\left(1+\frac{16 }{36 - 16}\right)=1. \hskip1cm \cont
\end{align*}
So $u=1$. Then \eqref{E:oee6} gives
\begin{align*}
w'^2&<4\left(\frac1{1}+\frac1{36}\right)\left(1+\frac{16 }{36 - 16}\right)=\frac{37}{5},
\end{align*}
and hence $w'=1$ or $2$. Now \eqref{E:oee5} gives 
\begin{equation}\label{E:a2oddevenb}
f(1+v) w'^2 (vf^2(v)- 16)= 4(z''^2+ f(v)).
\end{equation}
By Remark~\ref{R:ineq}, if $vf^2(v)\ge 36$, then either $v=36$ or $vf(v)\ge 64$ and $vf^2(v)\ge 64$. 
If $v=36$, then \eqref{E:oee7} gives
\begin{align*}
w'^2&<\frac{4}{37}\left(\frac1{1}+\frac1{36}\right)\left(1+\frac{16 }{36 - 16}\right)=\frac{1}{5}. \hskip1cm \cont
\end{align*}
So $vf(v)\ge 64$ and $vf^2(v)\ge 64$. 
Now if $vf^2(v)= 64$, then \eqref{E:oee7} gives
\begin{align*}
w'^2&<\frac{4}{65}\left(\frac1{1}+\frac1{64}\right)\left(1+\frac{16 }{64 - 16}\right)=\frac{1}{12}. \hskip1cm \cont
\end{align*}
So by Remark~\ref{R:ineq},  $vf(v)\ge 100$ and $vf^2(v)\ge 100$. 
Notice that $f(v)=1,3$ or  $f(v)\ge 5$. If $f(v)\ge 5$, then \eqref{E:oee6} gives
\begin{align*}
w'^2&<4\left(\frac1{5}+\frac1{100}\right)\left(1+\frac{16 }{100 - 16}\right)=1. \hskip1cm \cont
\end{align*}
So $f(v)=1$ or $3$. If $f(v)=3$, then \eqref{E:a2oddevenb} gives
  $f(1+v)  w'^2(9v- 16)= 4z''^2+ 12$, with $w'=1$ or $2$, so modulo 3, $-f(1+v)  \equiv z''^2$. But if $f(v)=3$, then $v+1\equiv 1\pmod 3$. Hence, since $v+1=f(v+1)s^2(v+1)$, we have $f(v+1)\equiv 1\pmod 3$. But then $-1\equiv z''^2\pmod 3$, which is impossible. So $f(v)=1$ and hence $v$ is an even square, $v=4n^2$ say. Notice that as $vf^2(v)\ge 100$ and $f(v)=1$, we have $n\ge 5$. Equation \eqref{E:a2oddevenb} gives $f(1+4n^2) w'^2 (4n^2- 16)= 4z''^2+ 4$, hence
  \[
f(1+4n^2) w'^2 (n^2- 4)= z''^2+ 1.
\]
This is impossible modulo 4 if $w'=2$, so $w'=1$. So
\begin{equation}\label{E:w1}
f(1+4n^2)  (n^2- 4)= z''^2+ 1.
\end{equation}
Note that $f(1+4n^2)\not=1$ since otherwise $1+4n^2$ would be a square, which is impossible. So, as the prime divisors of $1^2+(2n)^2$ are all congruent to 1 modulo 4, we have $f(1+4n^2)\ge 5$.
Now \eqref{E:oee4} gives $4n^2>z''^2$.  Hence \eqref{E:w1} gives
\[
5(n^2- 4)\le f(1+4n^2)  (n^2- 4)= z''^2+ 1<4n^2+1.
\]
Thus $n^2<21$, but this is impossible as $n\ge 5$. Hence $y''=1$ is not possible.
\end{proof}

\begin{lemma}\label{L:ynot1}  If $u=1$, then $v=4$. 
\end{lemma}

\begin{proof}  Suppose $u=1$ and $v>4$. So $vf(v) \ge 16$ and $vf^2(v) \ge 16$, by Remark~\ref{R:ineq}. From Lemma~\ref{L:ynot1case1}, $y''\ge 2$. 
Then  \eqref{E:oee6} gives
\begin{align*}
w'^2&< 4\left(\frac1{1}+\frac1{16}\right)\left(1+\frac{16 }{16\cdot 4 - 16}\right)=\frac{17}{3},
\end{align*}
so $w'=1$ or $2$.
Note that if $f(v)\ge 5$, then $v\ge 20, vf(v)\ge 100$ and $vf^2(v)\ge 500$, so \eqref{E:oee6} gives
\begin{align*}
w'^2&< 4\left(\frac1{5}+\frac1{100}\right)\left(1+\frac{16 }{500\cdot 4 - 16}\right)=\frac{105}{124}<1. \hskip1cm \cont
\end{align*}
So $f(v)=1$ or $3$. First suppose that $f(v)=3$. Then $v\ge 12, vf(v)\ge 36$ and $vf^2(v)\ge 108$, so \eqref{E:oee6} gives
\begin{align*}
w'^2&< 4\left(\frac1{3}+\frac1{36}\right)\left(1+\frac{16 }{108\cdot 4 - 16}\right)=\frac{3}{2},
\end{align*}
so $w'=1$. As $f(v)=3$, $v$ has the form $12n^2$, for some $n$. 
Then \eqref{E:oee5} gives $f(1+12n^2)  (108n^2 y''^2- 16)= 4z''^2+ 12y''^2$, so
\[
f(1+12n^2)  (27n^2 y''^2- 4)= z''^2+ 3y''^2.\]
But then modulo 3, since $f(1+12n^2) \equiv 1$, we have $-1\equiv z''^2$, which is impossible. So $f(v)=1$. In this case, 
$v$ is an even square; i.e., it has the form $4n^2$, for some $n$. 
Then \eqref{E:oee5} gives $f(1+4n^2) w'^2 (4n^2 y''^2- 16)= 4z''^2+ 4y''^2$, so
\begin{equation}\label{E:w12}
f(1+4n^2) w'^2 (n^2 y''^2- 4)= z''^2+ y''^2,
\end{equation}
where from above, $w'=1$ or $2$. First suppose that $w'=2$. Then modulo 4 we have $0\equiv z''^2+ y''^2$, so $y''$ and $z''$ are both even, say $y''=2y'''$ and $z''=2z'''$. So we have
$f(1+4n^2)(4n^2 y'''^2- 4)= z'''^2+ y'''^2$, which modulo 4 gives $0\equiv z'''^2+ y'''^2$. So $y'''$ and $z'''$ are both even, say $y'''=2y_4$ and $z'''=2z_4$. So we have
\[
f(1+4n^2)(4n^2 y_4^2- 1)= z_4^2+ y_4^2.\]
But now, arguing modulo 4 again, we have $-1\equiv z_4^2+ y_4^2$, which is impossible. Hence $w'=1$. Thus \eqref{E:w12} gives
\begin{equation}\label{E:w11}
f(1+4n^2) (n^2 y''^2- 4)= z''^2+ y''^2.
\end{equation}

Note that we can write $1+4n^2=f(1+4n^2)m^2$, for $m:=s(1+4n^2)$. Notice also from \eqref{E:oee4} and \eqref{E:w11},
$f(1+4n^2)(n^2 y''^2- 4)<  (4n^2+1) y''^2$,
so 
\[
n^2 - \frac{4}{y''^2}<  \frac{4n^2+1}{f(1+4n^2)}=m^2.
\]
Thus, as $y''\ge 2$, we have $n^2 - 1\le n^2 - \frac{4}{y''^2} < m^2$, so $n^2 \le m^2$, so $n\le m$. Then we have
\[
f(1+4n^2)=\frac{1+4n^2}{m^2}\le \frac{1}{m^2}+4.
\]
If $m=1$ then $n=1$ and so $v{\not>}4$ a contradiction. So $f(1+4n^2)<5$. Thus, as $f(1+4n^2)$ is odd, $f(1+4n^2)=3$. But this would imply that $1+4n^2$ is divisible by $3$ and hence $1+n^2\equiv0\pmod3$, which is impossible.
\end{proof}

From Lemma~\ref{L:3poss}, either $u=1,v=4$ or $y''=1$. We saw in Lemma~\ref{L:ynot1case1} that $y''\not=1$, and in Lemma~\ref{L:ynot1} that if $u=1$, then $v=4$. 
So it remains to consider the situation where $v=4$. Assume for the moment that $y''=2$. From \eqref{E:oee5i}, $vf^2(v)f(u)  y''^2> 16$, which gives  $f(u)>1$, 
so $f(u)\ge 3$. 
From \eqref{E:oee4}, $16 >u f(u) z''^2 $ which implies $f(u)<5$. Hence $f(u)=3$ and  again \eqref{E:oee4} gives $z''= 1$.
But then substituting $f(u)=3,v=4,y''=2,z''=1$ in \eqref{E:oee5} gives  
\[
f(u+4) w'^2(4 \cdot 3 \cdot 4 - 16) = 4(3 + 4)  \iff  8f(u+4) w'^2= 7,
\]
which is impossible.
Hence $y''\ge 3$. 

If 
$f(u)\ge 3$, \eqref{E:oee6} would give
\begin{align*}
w'^2&< 4\left(\frac1{9}+\frac1{4\cdot 3}\right)\left(1+\frac{16 }{4 \cdot 27\cdot 3^3 - 4}\right)=\frac{21}{23}<1. \hskip1cm \cont
\end{align*}
So $f(u)=1$. Then \eqref{E:oee6} gives
\begin{align*}
w'^2&<4\left(\frac1{1}+\frac1{4}\right)\left(1+\frac{16 }{4\cdot 9 - 16}\right)=9,
\end{align*}
so $w'=1$ or $2$.
As $f(u)=1$, $u$ is a square. First suppose $u>1$. So $u\ge 9$. Since  $u$ is a square, $u+4$ cannot be a square and thus $f(u+4)\ge 3$. Then \eqref{E:oee7} gives
\begin{align*}
w'^2&< \frac{4}3\left(\frac1{9}+\frac1{4}\right)\left(1+\frac{16 }{4 \cdot 3^3 - 4}\right)=\frac{13}{15}<1. \hskip1cm \cont
\end{align*}
Hence $u=1$.
Now \eqref{E:oee5} gives 
$5w'^2 (4 y''^2- 16)= 4z''^2+ 4y''^2$, so $5w'^2 (y''^2- 4)= z''^2+ y''^2$. 
If $w'=2$ we have $19 y''^2= z''^2+ 80$. But modulo 19 this gives $ z''^2\equiv -4$, which is impossible.
So $w'=1$ and we have $4 y''^2= z''^2+ 20$ and so $z''$ is even, say $z''=2z'''$, and
then $ y''^2= z'''^2+ 5$. It follows that $z'''=2$ and $y''=3$. This is the required solution: 
$\Sigma=2f(u+v)f(u) f(v)v w'^2=40$ and $\Sigma'=u\Sigma/v=10$.

\bigskip
\noindent
{\bf Case 2. Assume $u$ is odd, $v$ is even and the 2-adic order of $v$ is odd.}\

We will show that in this case, $(\Sigma,\Sigma')=(24,12)$ is the only possibility.

As $x$ is an integer,
from \eqref{E:x}  we can write $\Sigma= 2f(u+v)f(u) w^2$, for some $w$. 
Since $v$ divides $\Sigma$, so $v$ divides $2w^2$, and as the 2-adic order of $v$ is odd,   $f(v)s(v)$ divides $2w$. Thus, setting $2w=f(v)s(v)w'$ we have $2\Sigma= f(u+v)f(u) f(v)v w'^2$ and
\eqref{E:a1} gives
\begin{equation}\label{E:oeo1}
f(u+v)f(u)f(v)v w'^2 (vf(u) y'^2- 16)=16( z^2+vf(u) y'^2).
\end{equation}
Thus $vf(u)$ divides $16z^2$, so $f(v)s(v)f(u)$ divides $4z$, say $4z=f(v)s(v)f(u)z'$. 
So \eqref{E:yz} gives $4vy' >s(u)f(v)s(v)f(u)z'$ and hence
\begin{equation}\label{E:oeo2}
4s(v)y' > f(u)s(u)z', 
\end{equation}
and \eqref{E:oeo1} gives
\begin{equation}\label{E:oeo3}
f(u+v) f(v) w'^2 (vf(u) y'^2- 16)=f(v) f(u)z'^2+ 16y'^2.
\end{equation}
Hence $f(v)$ divides $16y'^2$. Since the 2-adic order of $v$ is odd, $f(v)$  divides $2y'$. Let $2y'=f(v)y''$.
Then \eqref{E:oeo2} gives
$2f(v)s(v)y'' > f(u)s(u)z'$, 
and \eqref{E:oeo3} gives
$f(u+v) w'^2 (vf^2(v)f(u) y''^2- 64)= 4f(u)z'^2+ 16f(v)y''^2$.
From this last equation, notice that as $v,f(v)$ are even and $f(u)$ is odd, $z'$ must be even, say $z'=2z''$. So we have
\begin{equation}\label{E:oeo4}
f(v)s(v)y'' > f(u)s(u)z''\quad \text{and so}\quad vf(v)y''^2>uf(u)z''^2, 
\end{equation}
and 
\begin{equation}\label{E:oeo5}
f(u+v) w'^2 (vf^2(v)f(u) y''^2- 64)= 16(f(u)z''^2+ f(v)y''^2).
\end{equation}
Note that from the left-hand side of \eqref{E:oeo5},  we have
\begin{equation}\label{E:oeo5i}
vf^2(v)f(u) y''^2> 64.
\end{equation}
Furthermore, \eqref{E:oeo4} and \eqref{E:oeo5} give
\[
uf(u+v) w'^2 (vf^2(v)f(u) y''^2- 64)< 16(v +u)f(v)y''^2.
\]
Hence 
 \begin{equation}\label{E:oeo7}
 w'^2<
 \frac{16}{f(u+v)}\left(\frac1{uf(u)f(v)}+\frac1{vf(v)f(u)}\right)\left(1+\frac{64 }{vf^2(v)f(u) y''^2- 64}\right)
 \end{equation}
 and consequently
 \begin{equation}\label{E:oeo6}
 w'^2<
16\left(\frac1{uf(u)f(v)}+\frac1{vf(v)f(u)}\right)\left(1+\frac{64 }{vf^2(v)f(u) y''^2- 64}\right).
\end{equation}

\begin{lemma}\label{L:Cs}  $f(v)=2$.
\end{lemma}

\begin{proof} Assume that $f(v)>2$. So $f(v)\ge 6,vf(v)\ge 36, vf^2(v)\ge 216$.
We first show that $f(u)=1$. Indeed, if $f(u)>1$, then $f(u)\ge 3, uf(u)\ge 9$ and for all $y''$, 
\eqref{E:oeo6} would give
\begin{align*}
w'^2&<16\left(\frac1{9\cdot 6}+\frac1{36\cdot 3}\right)\left(1+\frac{64 }{216\cdot 3- 64}\right)=\frac{36}{73}<1. \hskip1cm \cont
\end{align*}
So $f(u)=1$. Hence $u$ is a square. Moreover, \eqref{E:oeo5} gives
\begin{equation}\label{E:oeo5fu1}
f(u+v) w'^2 (vf^2(v) y''^2- 64)= 16(z''^2+ f(v)y''^2).
\end{equation}

Since $f(v)$ is even and square-free, the five smallest possible values of $f(v)$ are $2,6,10,14,22$. If $f(v)>14$, then $f(v)\ge 22$, so $v\ge 22$ and for all $u,y''$,
\eqref{E:oeo6} would give
\begin{align*}
w'^2&<16\left(\frac1{22}+\frac1{22^2}\right)\left(1+\frac{64 }{22^3- 64}\right)=\frac{1012}{1323}<1. \hskip1cm \cont
\end{align*}
So $f(v)\le14$. Now suppose $f(v)=14$. Here $v\ge 14$ and  \eqref{E:oeo6}  gives
\begin{align*}
w'^2&<16\left(\frac1{14}+\frac1{14^2}\right)\left(1+\frac{64 }{14^3- 64}\right)=\frac{84}{67}<2,
\end{align*}
so $w'=1 $. Furthermore, $v$ is divisible by $7$, so as $\gcd(u,v)=1$ and $u$ is a square, $u+v\equiv 1,2$ or $4\pmod 7$, and thus 
$f(u+v)\equiv 1,2$ or $4\pmod 7$. Modulo 7, \eqref{E:oeo5fu1} gives
$ -f(u+v) \equiv  2 z''^2$,
so $z''^2\equiv 3,6,5$ respectively, but these congruences have no solutions.  So  $f(v)\not=14$.  

Now suppose $f(v)=10$.  Here  \eqref{E:oeo6}  gives
\begin{align*}
w'^2&<16\left(\frac1{10}+\frac1{10^2}\right)\left(1+\frac{64 }{10^3- 64}\right)=\frac{220}{117}<2,
\end{align*}
so $w'=1 $. 
Then \eqref{E:oeo5fu1} gives $f(u+v)  (100vy''^2- 64)= 16(z''^2+ 10y''^2)$. 
We have $v=10m^2$, for some $m$.  So
\begin{equation}\label{E:oeov10}
f(u+10m^2)  (125(my'')^2- 8)= 2(z''^2+ 10y''^2).
\end{equation}
 As $f(u)=1$, we have $u=n^2$, for some odd $n$. First suppose $f(n^2+10m^2)=1$.  So $n^2+10m^2=r^2$, for some odd $r$. Then 
 $1+2m^2\equiv 1\pmod 4$, and hence $m$ must be even, say $m=2m'$. So \eqref{E:oeov10} gives
 $250(m'y'')^2- 4= z''^2+ 10y''^2$ and hence $z''$ is even, say $z''=2z'''$. So $125(m'y'')^2- 2= 2z'''^2+ 5y''^2$. But one readily verifies that modulo 16, this equation 
has no solution for $m',y'',z'''$. 
Thus  $f(n^2+10m^2)\ge3$. 
Then \eqref{E:oeo7}  gives
\begin{align*}
w'^2&<\frac{16}3\left(\frac1{10}+\frac1{10^2}\right)\left(1+\frac{64 }{10^3- 64}\right)=\frac{220}{351}<1. \hskip1cm \cont
\end{align*}
So  $f(v)\not=10$.

Now suppose $f(v)=6$. Here  \eqref{E:oeo6}  gives
\begin{align*}
w'^2&<16\left(\frac1{6}+\frac1{6^2}\right)\left(1+\frac{64 }{6^3- 64}\right)=\frac{84}{19}<5,
\end{align*}
so $w'=1 $ or $2$. Furthermore, $v$ is divisible by $3$, so as $\gcd(u,v)=1$ and $u$ is a square, $u+v\equiv 1\pmod 3$,  and thus 
$f(u+v)\equiv  1\pmod 3$. Hence modulo 3, \eqref{E:oeo5fu1} gives
$-1 \equiv   z''^2$,
which is impossible. So $f(v)\not=6$.
\end{proof}


\begin{remark}\label{R:ineq2}
By the previous lemma, $f(v)=2$. So $v$ has the form $v=2m^2$ for some $m$. In particular, $v\ge 2, vf(v)\ge 4$ and $vf^2(v)\ge 8$. Furthermore, here are some obvious useful facts:
\begin{enumerate}
\item If $v>32$, then $vf(v)\ge 100,vf^2(v)\ge 200$, while $v\le 32$  only for $v=2,8,18$ and $32$. 
\item If $v>98$, then $vf(v)\ge 256,vf^2(v)\ge 512$, while $v\le 98$  only for $v=2,8,18,32,50,72$ and $98$. 
\item If $u>1$, then $u\ge 3$  and $uf(u)\ge 9$. Furthermore, if $u>3$ and $u\not=9$, then $uf(u)\ge 25$. 
And if $u>5$ and $u\not=25$, then $uf(u)\ge 49$. \end{enumerate}
\begin{table}
\begin{tabular}{c||c|c|c|c|c|c|c|c|c|c}
  \hline
   $v$&2& 8& 18& 32& 50& 72& 98& 128& 162& 200 \\\hline
   $vf(v)$&4& 16& 36& 64& 100& 144& 196& 256& 324& 400 \\
   $vf^2(v)$&8& 32& 72& 128& 200& 288& 392& 512& 648& 800
\\\hline
 \end{tabular}
\bigskip
\caption{The first ten positive integers $v$ with $f(v)=2$}\label{F:even2}
\end{table}
\end{remark}


\begin{lemma}\label{L:damn}  Either $u=1$ or $v=2$ or $y''=1$.
\end{lemma}

\begin{proof}
Suppose $u>1$ and $y''\ge 2$. We will show that $v=2$.  Let us assume for the moment that $v>98$, so $vf(v)\ge 256,vf^2(v)\ge 512$  by Remark~\ref{R:ineq2}(b). We also have $uf(u)\ge 9$ by Remark~\ref{R:ineq2}(c).  Then using $f(u)\ge 1$, \eqref{E:oeo6} gives 
\begin{align*}
w'^2&<16\left(\frac1{9\cdot 2}+\frac1{256}\right)\left(1+\frac{64}{512\cdot 4- 64}\right)=\frac{274}{279}<1. \hskip1cm \cont
\end{align*}
So $v\le 98$, and thus, as mentioned in the above remark,  $v=2,8,18,32,50,72$ or $98$. Our goal is to exclude the last 6 of these 7 possible $v$-values. We will first consider the cases $u=3$ and $u=9$. Note that of our 6 $v$-values of interest, we need only consider the ones relatively prime to $3$; that is, 8,32,50,98. For these four $v$-values, if $u=3$, then $\frac{16}{f(3+v)}(\frac1{9f(v)}+\frac1{3vf(v)})(1+\frac{64 }{3vf^2(v) y''^2- 64})$ takes the respective values
$\frac{2}{15},  \frac{2}{69},\frac{4}{219},\frac{4}{435}$,
and as these values are all less than 1, we obtain a contradiction from \eqref{E:oeo7}. Similarly, if $u=9$, then $\frac{16}{f(9+v)}(\frac1{9f(v)}+\frac1{vf(v)})(1+\frac{64 }{vf^2(v) 2^2- 64})$ takes the respective values
$\frac{2}{9},  \frac{2}{63},\frac{4}{207},\frac{4}{423}$,
and as these values are also all less than 1, we again obtain a contradiction from \eqref{E:oeo7}.
So we may assume that  $u>3$ and $u\not=9$. Then by Remark~\ref{R:ineq2}(c), we have $uf(u)\ge 25$ and for the five $v$-values $v=18,32,50,72,98$ respectively, one finds that $16(\frac1{25f(v)}+\frac1{vf(v)})(1+\frac{64 }{vf^2(v) 2^2- 64})$ takes the values 
$ \frac{172}{175},  \frac{114}{175},\frac{12}{23},\frac{194}{425},\frac{492}{1175}$.
 As these values are all less than 1, we obtain a contradiction from \eqref{E:oeo6}. It remains to treat the case $v=8$, with $u>3$ and $u\not=9$.
First note that in this case, $uf(u)\ge 25$ by Remark~\ref{R:ineq2}(c), and if $f(u)>1$, then $f(u)\ge 3$. But then \eqref{E:oeo6} gives
\[
 w'^2<
16\left(\frac1{25\cdot 2}+\frac1{16\cdot 3}\right)\left(1+\frac{64 }{32\cdot 3 \cdot 4 - 64}\right)=\frac{98}{125}<1. \hskip1cm \cont
\]
So we may assume $f(u)=1$, i.e., $u$ is a square. But then, as $u>1$, $u+8$ is not a square and so $f(u+8)\ge 3$. Then \eqref{E:oeo7} gives
\[
 w'^2<
\frac{16}{3}\left(\frac1{25\cdot 2}+\frac1{16}\right)\left(1+\frac{64 }{32 \cdot 4 - 64}\right)=\frac{22}{25}<1. \hskip1cm \cont
\]
Hence $v=8$ is impossible. Thus $v=2$.
\end{proof}


\begin{lemma}\label{L:y1u1}  If $y''=1$, then $u=1$. 
\end{lemma}

\begin{proof} Suppose $y''=1$, and arguing by contradiction, suppose $u>1$. Note that \eqref{E:oeo4} gives $vf(v)  > uf(u)z''^2\ge uf(u)$. Also, by \eqref{E:oeo5i},  $vf^2(v)f(u) > 64$. So if $v=2$, then  $vf^2(v)=8$ and hence $f(u)>8$, contradicting the fact that $vf(v)>uf(u)$. Hence  $v>2$.
By Lemma~\ref{L:Cs}, $f(v)=2$. So, as $v>2$, we have $v=2m^2$ for some $m\ge 2$.
Notice also that if $m=2$, then $vf^2(v)=32$, and so $ vf^2(v)f(u) > 64$ gives $f(u)>1$.

If $m\ge 11$, then as $uf(u)\ge 9$ by Remark~\ref{R:ineq2}(c), \eqref{E:oeo6} gives
\begin{align*}
w'^2&<16\left(\frac1{9\cdot 2}+\frac1{4m^2}\right)\left(1+\frac{64 }{8m^2- 64}\right)=\frac{1004}{1017}<1. \hskip1cm \cont
\end{align*}
So we need only consider $2\le m\le 10$. First consider $u=3$ and $u=9$. The numbers $v=2m^2$  with $2\le m\le 10$ and $\gcd(u,v)=1$ are given by $m=2,4,5,7,8,10$.
For $u=3$ and $m= 2,4,5,7,8,10$, the values of $\frac{16}{f(3+2m^2)}(\frac1{9\cdot 2}+\frac1{3\cdot 4m^2})(1+\frac{64 }{8m^2\cdot 3- 64})$ are respectively
\[
\frac{1}{3},\frac{1}{30},\frac{4}{201},\frac{4}{417},\frac{1}{138},\frac{1}{219}.\]
 As these values are all less than 1, we obtain a contradiction from   \eqref{E:oeo7}.
Now let $u=9$. Here $f(u)=1$ and so, as we observed at the beginning of this proof, $m>2$. For  $m= 4,5,7,8,10$, the values of $\frac{16}{f(9+2m^2)}(\frac1{9\cdot 2}+\frac1{ 4m^2})(1+\frac{64 }{8m^2- 64})$ are respectively
\[
\frac{1}{18},\frac{4}{153},\frac{4}{369},\frac{1}{126},\frac{1}{207}.\]
 As these values are all less than 1, we again obtain a contradiction from   \eqref{E:oeo7}.

From what we have just shown, we may suppose that  $u\ge 5$ and $u\not=9$, so $uf(u)\ge 25$, by Remark~\ref{R:ineq2}(c). If $m\ge 5$, then  \eqref{E:oeo6} gives
\begin{align*}
w'^2&<16\left(\frac1{25\cdot 2}+\frac1{4\cdot 25}\right)\left(1+\frac{64 }{8\cdot 25- 64}\right)=\frac{12}{17}<1. \hskip1cm \cont
\end{align*}
It remains to treat the cases $m=2,3,4$ for $u\ge 5$ and $u\not=9$.

If $u=5$, then for  $m= 2,3,4$, the values of $\frac{16}{f(5+2m^2)}(\frac1{25\cdot 2}+\frac1{5\cdot 4m^2})(1+\frac{64 }{8m^2 \cdot5  - 64})$ are respectively
$\frac{1}{15},\frac{4}{185},\frac{1}{90}$, which is impossible by   \eqref{E:oeo7}.
 Similarly, if $u=25$, then for  $m= 3,4$, the values of $\frac{16}{f(25+2m^2)}(\frac1{25\cdot 2}+\frac1{ 4m^2})(1+\frac{64 }{8m^2- 64})$ are respectively
$\frac{4}{25},\frac{1}{50}$, which is again impossible by  \eqref{E:oeo7}. For $m=2$ we don't need to consider $u=25$ as $f(25)=1$ and  as we observed at the beginning of this proof, $f(u)>1$ for $m=2$. Thus for 
 $m= 2,3,4$, we may assume that $u>5$ and $u\not=25$.
 
 For $m=4$ with $u>5$ and $u\not=25$,  we have $uf(u)\ge 49$ by Remark~\ref{R:ineq2}(c), so \eqref{E:oeo6} gives
 \[
 w'^2<16\left(\frac1{49\cdot 2}+\frac1{64}\right)\left(1+\frac{64 }{64\cdot 2- 64}\right)=\frac{81}{98}<1. \hskip1cm \cont
\]
 
It therefore remains to treat the cases $m=2,3$ for $u\ge 7$ and $u\not=25$. First let $m=2$. Then, as we saw at the beginning of the proof,  $f(u)\ge 3$. If $f(u)\ge 5$, then as $uf(u)\ge 49$, \eqref{E:oeo6} gives
 \[
 w'^2<16\left(\frac1{49\cdot 2}+\frac1{16\cdot 5}\right)\left(1+\frac{64 }{32\cdot 5- 64}\right)=\frac{89}{147}<1. \hskip1cm \cont
\]
If $f(u)=3$, then $u$ is divisible by 3, so $u+v=u+8\equiv 2\pmod3$, and hence $u+v$ is not a square, and neither is it divisible by $3$. So $f(u+v)\not=1$, and consequently $f(u+v)\ge 5$. Thus, using $uf(u)\ge 49$ and $f(u)=3$, \eqref{E:oeo7} gives
 \[
 w'^2<\frac{16}{5}\left(\frac1{49\cdot 2}+\frac1{16\cdot 3}\right)\left(1+\frac{64 }{32\cdot 3- 64}\right)=\frac{73}{245}<1. \hskip1cm \cont
\]
So the case $m=2$ is also impossible.

Finally let $m=3$.  If $f(u)\ge 3$, then as $uf(u)\ge 49$, \eqref{E:oeo6} gives
 \[
 w'^2<16\left(\frac1{49\cdot 2}+\frac1{36\cdot 3}\right)\left(1+\frac{64 }{72\cdot 3- 64}\right)=\frac{412}{931}<1. \hskip1cm \cont
\]
If $f(u)=1$, then $u$ is a square, so  $u+v=u+18$ is not a square. Thus $f(u+v)\not=1$. Furthermore, $u$ is not divisible by $3$ since $v=18$ and $\gcd(u,v)=1$. So $u+18$ is not divisible by $3$. Hence  $f(u+v)\ge 5$. Furthermore, as $uf(u)\ge 49$ and $f(u)=1$, we have $u\ge 49$.
Thus \eqref{E:oeo7} gives
 \[
 w'^2<\frac{16}{5}\left(\frac1{u\cdot 2}+\frac1{36}\right)\left(1+\frac{64 }{72- 64}\right)=\frac{4 (u+18)}{5 u}.
\]
As $w'\ge 1$, we obtain $u<72$, and so as $u$ is an odd square with $u\ge 49$, we have $u=49$. But then $f(u+v)=57$ and 
\eqref{E:oeo7} gives a contradiction, as
 \[
 w'^2<\frac{16}{67}\left(\frac1{49\cdot 2}+\frac1{36}\right)\left(1+\frac{64 }{72- 64}\right)=\frac{4}{49}<1. \hskip1cm \cont
\]
This completes the proof of the lemma.
\end{proof}

\begin{lemma}\label{L:y''not1}  $y''\not=1$. 
\end{lemma}

\begin{proof} Suppose $y''=1$, so by the above lemma, $u=1$. In this case,  \eqref{E:oeo5} gives
\begin{equation}\label{E:oeo52}
f(1+v) w'^2 (vf^2(v)- 64)= 16(z''^2+ f(v)).
\end{equation}
From Lemma~\ref{L:Cs}, $f(v)=2$, so $v$ has the form $v=2m^2$ for some $m$. Then \eqref{E:oeo52} gives
$f(1+v) w'^2 (8m^2- 64)= 16z''^2+32$, so
 \begin{equation}\label{E:oeo11}
f(1+2m^2) w'^2 (m^2- 8)= 2z''^2+4.
 \end{equation}
Notice that the 2-adic order is exactly 2, since on the right hand  side it is 1 or 2 and on the left hand side at least 2. 
So $z''$ is even, say $z''=2z'''$. Suppose that $m$ is even and write $m=2m'$. After replacing and dividing by 4 one gets  $f(1+8m'^2)w'^2(m'^2-2) = 2z'''^2+1$, so $w'$ and $m'$ are both odd. Observed modulo 8 this gives a contradiction, since the LHS is $-1$  and the RHS is 1 or 3 modulo 8. Hence $m$ must be odd and thus from \eqref{E:oeo11}  $w'$ is necessarily even, say $w'=2w''$,  giving 
\begin{equation}\label{E:eospbi}
f(1+2m^2) w''^2 (m^2- 8)= 2z'''^2+1.
 \end{equation}
From \eqref{E:oeo4} we have $m > z'''$. Thus 
\eqref{E:eospbi} gives
 \begin{equation}\label{E:oeosp}
 w''^2 <
 \frac{2m^2+1}{f(1+2m^2)(m^2- 8)}.
 \end{equation}
 Notice that from \eqref{E:eospbi} we have $m^2>8$, so  $m\ge 3$ and is odd. 
 If $m=3$, then \eqref{E:oeosp}
 gives
 $w''^2 <\frac{19}{19}=1$, a contradiction.  If $f(1+2m^2)=3$ then  $\frac{2m^2+1}{3(m^2- 8)}\le 1$ for $m\ge 5$, absurd. Finally, if  $f(1+2m^2)=1$ then $w''^2<\frac{2m^2+1}{m^2- 8}\le 3$ for $m\ge 5$ and so $w''=1$. Replacing in \eqref{E:eospbi} gives $m^2 = 2z''^2 +9$, which modulo 3 implies either $1\equiv 0$ or $1\equiv 2$.
\end{proof}


\begin{lemma}\label{L:v2u1}  If $v=2$, then $u=1$. 
\end{lemma}

\begin{proof} Suppose that $v=2$ and $u> 1$. So by Remark~\ref{R:ineq2}(c), $uf(u)\ge 9$.
From the previous lemma,  $y''\ge 2$. Assume for the moment that $y''=2$. By \eqref{E:oeo5i}, $vf^2(v)f(u)  y''^2> 64$, which gives  $f(u)>2$, so $f(u)\ge 3$. 
Now from \eqref{E:oeo4}, we have $uf(u)z''^2<  vf(v)y''^2 =16$. So  $ uf(u)\ge 9$ gives $z'=1$ and  then $uf(u)<  16$ and $f(u)\ge 3$ give $u=3$.
Then $f(u+v)=5$ and so \eqref{E:oeo7} gives
\begin{align*}
w'^2&< \frac{16}{5}\left(\frac1{9\cdot 2}+\frac1{4\cdot3}\right)\left(1+\frac{64 }{8\cdot3\cdot 4 - 64}\right)=\frac{4}{3}<2,
\end{align*}
so $w'=1$. But then, 
by \eqref{E:oeo5}, $5 (8\cdot3\cdot 4 - 64)= 16(3z''^2+ 8)$, which has no integer solution for $z''$. So $y''\ge 3$.

Note that if $f(u)\ge 7$, then $u\ge 7$ and by \eqref{E:oeo6}
 \[
 w'^2<
16\left(\frac1{49\cdot 2}+\frac1{4\cdot 7}\right)\left(1+\frac{64 }{8\cdot 7\cdot 9- 64}\right)=\frac{324}{385}<1. \hskip1cm \cont
\]
So it suffices to deal with the three cases $f(u)=1,3,5$. First suppose $f(u)=5$. So $uf(u)\ge 25$. As $f(u)=5$ and $v=2$, we have $u+v\equiv 2\pmod5$ and hence $u+v$ is not a square. So   $f(u+v)\ge 3$.  Hence by \eqref{E:oeo7}
 \[
 w'^2<
\frac{16}3\left(\frac1{25\cdot 2}+\frac1{4\cdot 5}\right)\left(1+\frac{64 }{8\cdot 5\cdot 9- 64}\right)=\frac{84}{185}<1. \hskip1cm \cont
\]
So $f(u)\not=5$. Now suppose $f(u)=3$. So $uf(u)\ge 9$. We have $u+v\equiv 2\pmod3$ and hence $u+v$ is not a square. So   $f(u+v)\ge 3$. But  $f(u+v)\not= 3$, because $u+v\equiv 2\pmod3$. So $f(u+v)\ge5$. Hence by \eqref{E:oeo7}
 \[
 w'^2<
\frac{16}5\left(\frac1{9\cdot 2}+\frac1{4\cdot 3}\right)\left(1+\frac{64 }{8\cdot 3\cdot 9- 64}\right)=\frac{12}{19}<1. \hskip1cm \cont
\]
So $f(u)\not=3$. Finally, suppose $f(u)=1$. Then 
\eqref{E:oeo5} gives
 \begin{equation}\label{E:oeo521}
f(u+v) w'^2 ( y''^2- 8)= 2(z''^2+ 2y''^2).
\end{equation}
As $f(u)=1$, so $u$ is an odd square, say $u=n^2$. Thus, as $u>1$ by hypothesis, $u\ge 9$.
As $u$ is a square, $u+2$ is not a square, so $f(u+v)\ge3$. By \eqref{E:oeo7},
 \[
 w'^2<
\frac{16}{3}\left(\frac1{9\cdot 2}+\frac1{4}\right)\left(1+\frac{64 }{8\cdot 9- 64}\right)=\frac{44}{3}<15.
\]
So $w'=1,2$ or $3$. Suppose for the moment that $y''=3$. Then as $f(u+v)$ is odd, $w'$ must be even, by \eqref{E:oeo521}, so $w'=2$. Then \eqref{E:oeo521} gives
$2f(u+v)  = z''^2+ 18$.
It follows that $z''$ must be even and hence $f(u+v)\ge11$. But then  \eqref{E:oeo7} gives
 \[
 w'^2<
\frac{16}{11}\left(\frac1{9\cdot 2}+\frac1{4}\right)\left(1+\frac{64 }{8\cdot 9- 64}\right)=4,
\]
contradicting $w'=2$. Hence $y''\ge 4$.

Note that for $y''\ge 4$, if $f(u+v)\ge11$, then \eqref{E:oeo7} would give
 \[
 w'^2<
\frac{16}{11}\left(\frac1{9\cdot 2}+\frac1{4}\right)\left(1+\frac{64 }{8\cdot 16- 64}\right)=\frac{8}{9}<1. \hskip1cm \cont
\]
So, as   $f(u+v)$ is square-free, $f(u+v)=3,5$ or $7$. But then $u+v$ would be divisible by $3,5,7$ respectively. Since there is no $n$ for which $n^2+2 \equiv 0$ modulo 5 or 7, we conclude that $f(u+v)=3$. Thus $n^2+2=u+v$ is divisible by 3 and hence $n^2\equiv 1\pmod 3$. Thus, for $u>1$  we have $n>3$ and thus $u\ge 25$.
Then  \eqref{E:oeo7} gives
 \[
 w'^2<
\frac{16}{3}\left(\frac1{25\cdot 2}+\frac1{4}\right)\left(1+\frac{64 }{8\cdot 16- 64}\right)=\frac{72}{25}<3.
\]
So $w'=1$. But then
\eqref{E:oeo521} would give $3 ( y''^2- 8)= 2(z''^2+ 2y''^2)$, so $0=24+2z''^2+ y''^2$, which is obviously impossible. Hence $u=1$.
\end{proof}


\begin{lemma}\label{L:u1v2}  If $u=1$, then $v=2$. 
\end{lemma}

\begin{proof} Suppose that $u=1$ and $v> 2$. So by Remark~\ref{R:ineq2}, $v\ge8,vf(v) \ge 16$ and $vf^2(v) \ge 32$. By Lemma~\ref{L:y''not1}, $y''\ge 2$.
Then  \eqref{E:oeo6} gives
\begin{align*}
w'^2&<16\left(\frac1{2}+\frac1{16}\right)\left(1+\frac{64 }{32\cdot 4 - 64}\right)=18,
\end{align*}
so $w'=1,2,3$ or $4$. 
 Lemma~\ref{L:Cs} and Equation \eqref{E:oeo5} give
 \begin{equation}\label{E:oeou1}
f(1+v) w'^2 (v y''^2- 16)= 4(z''^2+ 2y''^2),
\end{equation}
while, setting $v=2m^2$, \eqref{E:oeo4} gives
\begin{equation}\label{E:oeo4u1}
 (2m)^2y''^2>z''^2. 
\end{equation}
Let us first dispense with the case $v=8$. Suppose $v=8$. Then $f(1+v)=1$ and substituting in \eqref{E:oeou1}, the four possibilities for $w'$ give:
\begin{enumerate}
\item $w'=1$: $ 2( y''^2- 2)= z''^2+ 2y''^2$, so $-4=z''^2$, which is obviously impossible.
\item $w'=2$: $ 8( y''^2- 2)= z''^2+ 2y''^2$,  so $6y''^2-16=z''^2$, which is  impossible modulo 3.
\item $w'=3$: $ 18( y''^2- 2)= z''^2+ 2y''^2$,  so $(4y)''^2=z''^2+6^2$, which is  impossible as there is no such Pythagorean triple. 
\item $w'=4$: $32( y''^2- 2)= z''^2+ 2y''^2$,  so $30y''^2-64=z''^2$, which is  impossible modulo 3.
\end{enumerate}
So $v>8$ and hence $v\ge 18, vf(v) \ge 36$ and $vf^2(v) \ge 72$.
Then  \eqref{E:oeo6} gives
\begin{align*}
w'^2&<16\left(\frac1{2}+\frac1{36}\right)\left(1+\frac{64 }{72\cdot 4 - 64}\right)=\frac{76}7<11,
\end{align*}
so $w'=1,2$ or $3$. Substituting $w'=3$ in \eqref{E:oeou1}  for  $v=18$ and $32$ gives respectively
\[
19\cdot 9 (9 y''^2- 8)= 2(z''^2+ 2y''^2) \ \text{and}\ 33\cdot 9 (8 y''^2- 4)= z''^2+ 2y''^2.
\]
However, neither of these equations has a solution modulo 64.  So $w'=1$ or $2$   for  $v=18$ and $32$. For $v>32$ we have $v\ge 50$ and 
\eqref{E:oeo6} gives
\begin{align*}
w'^2&<16\left(\frac1{2}+\frac1{100}\right)\left(1+\frac{64 }{200\cdot 4 - 64}\right)=\frac{204}{23}<9,
\end{align*}
so $w'=1$ or $2$. Thus we have $w'=1$ or  $2$ for all $v\ge 18$.

First suppose  $w'=2$.  Note that if $f(1+2m^2)\ge 3$, then  \eqref{E:oeo4u1} and \eqref{E:oeou1} give
\[
12 (m^2y''^2 -8) \le 4 f(1+2m^2) (m^2y''^2 -8)=2(z''^2+ 2y''^2)<(8m^2 +4)y''^2,
\]
so $(m^2-1)y''^2< 24$. But for $v\ge 18$, we have $m\ge 3$. So $(m^2-1)y''^2< 24$ gives $y''^2< 3$, hence $y''=1$, contrary to Lemma~\ref{L:y''not1}. We conclude that  $f(1+2m^2)=1$. Thus   \eqref{E:oeou1} gives $4(m^2y''^2 -8)=2(z''^2+ 2y''^2)$, so
\[
2(m^2-1)y''^2 =z''^2+16.
\]
Since $f(1+2m^2)=1$, we have that $1+2m^2$ is a square, say $1+2m^2=n^2$. But investigations show that the simultaneous equations $2(m^2-1)y''^2 =z''^2+16$ and $1+2m^2=n^2$ have no integer solution modulo 128. Hence $w'=2$ is impossible.

Finally, suppose $w'=1$. Note that if $f(1+2m^2)\ge 11$, then  \eqref{E:oeo4u1} and \eqref{E:oeou1} give:
\[
11 (m^2y''^2 -8) \le  f(1+2m^2) (m^2y''^2 -8)=2(z''^2+ 2y''^2)<(8m^2 +4)y''^2,
\]
so $(3m^2-4)y''^2< 88$. But for $v\ge 18$, we have $m\ge 3$ and so $(3m^2-4)y''^2< 88$ gives $y''^2< 88/23<4$, hence $y''=1$, contrary to Lemma~\ref{L:y''not1}. So $f(1+2m^2)=1,3,5$ or~$7$. 

For $f(1+2m^2)=7$, \eqref{E:oeou1} gives $7 (m^2y''^2 -8) =2(z''^2+ 2y''^2)$, which has no solution modulo $49$. So $f(1+2m^2)\not=7$.

For $f(1+2m^2)=5$, \eqref{E:oeou1} gives $5 (m^2y''^2 -8) =2(z''^2+ 2y''^2)$, which has no solution modulo $25$. So $f(1+2m^2)\not=5$.

For $f(1+2m^2)=1$ and $3$, the calculation is slightly more complicated. For $f(1+2m^2)=1$
we consider the pair of the simultaneous equations $ m^2y''^2 -8 =2(z''^2+ 2y''^2)$ and $1+2m^2=n^2$, while for $f(1+2m^2)=3$
we consider the pair of the simultaneous equations $3 (m^2y''^2 -8) =2(z''^2+ 2y''^2)$ and $1+2m^2=3n^2$. In both cases one finds that the pair of equations has no solution modulo 64. Thus $w'=1$ is also impossible.
\end{proof}


Given the above lemmas, it remains to treat the case where $u=1,v=2$ and $y''\ge 2$. By \eqref{E:oeo5i}, $vf^2(v)f(u) y''^2- 64>0$, so $y''^2 > 8$. Thus $y''\ge 3$. Then  \eqref{E:oeo7} gives
 \[
 w'^2<
\frac{16}{3}\left(\frac1{ 2}+\frac1{4}\right)\left(1+\frac{64 }{8\cdot 9- 64}\right)=36,
\]
so $w'\le 5$. Equation \eqref{E:oeo5} gives
 \begin{equation}\label{E:oeou1v2}
3 w'^2 ( y''^2- 8)= 2(z''^2+ 2y''^2),
\end{equation}
One finds that for $w'=1,3$ and $5$, \eqref{E:oeou1v2} has no solution modulo 64. So $w'=2$ or~$4$.

For $w'=4$,   \eqref{E:oeo4} and \eqref{E:oeou1v2} give
$48 (y''^2 -8) =2(z''^2+ 2y''^2)<12y''^2$,
so $3y''^2< 32$. Thus, as $y''\ge 3$, we have $y''= 3$, and \eqref{E:oeou1v2} gives $z''^2=6$, which is obviously impossible. So $w'=2$. 

Finally, for $w'=2$, Equation \eqref{E:oeou1v2} has a unique positive integer solution: $y''=z''=4$.
From the definitions, for $u=1,v=2$, we have
$\Sigma= 2f(u+v)f(u)vf(v)w'^2/4=6w'^2=24$. Consequently, $\Sigma'=u\Sigma/v=12$. This is the required case 2 solution.

\bigskip
\noindent
{\bf Case 3. Assume $u$ is even, $v$ is odd, and the 2-adic order of $u$ is even.}\ 

We will show that in this case, $(\Sigma,\Sigma')=(10,40)$ and $(18,32)$ are the only two possibilities.

As $x$ is an integer,
from \eqref{E:x}  we can write $\Sigma= 2f(u+v)f(u) w^2$. 
Note $v$  divides $\Sigma$, so $v$ divides $w^2$, and hence  $f(v)s(v)$ divides $w$. Thus, setting $w=f(v)s(v)w'$ we may write $\Sigma= 2f(u+v)f(u) f(v)v w'^2$. Then
\eqref{E:a1} gives
\begin{equation}\label{E:eoe1}
f(u+v)f(u)f(v)v w'^2 (vf(u) y'^2- 16)=4( z^2+vf(u) y'^2).
\end{equation}
Thus $vf(u)$ divides $16z^2$, so $f(v)s(v)f(u)$ divides $z$, say $z=f(v)s(v)f(u)z'$. So \eqref{E:yz} gives $vy' >s(u)f(v)s(v)f(u)z'$ and hence
\begin{equation}\label{E:eoe2}
s(v)y' > f(u)s(u)z', 
\end{equation}
and \eqref{E:eoe1} gives
\begin{equation}\label{E:eoe3}
f(u+v) f(v) w'^2 (vf(u) y'^2- 16)=4(f(v) f(u)z'^2+ y'^2).
\end{equation}
Hence $f(v)$  divides $y'$. Let $y'=f(v)y''$.
Then \eqref{E:eoe2} gives 
\begin{equation}\label{E:eoe4}
f(v)s(v)y'' > f(u)s(u)z'\quad \text{and so}\quad vf(v)y''^2>uf(u)z'^2, 
\end{equation}
and \eqref{E:eoe3} gives
\begin{equation}\label{E:eoe5}
f(u+v) w'^2 (vf^2(v)f(u) y''^2- 16)= 4(f(u)z'^2+ f(v)y''^2).
\end{equation}

\begin{remark}\label{R:even}
Note that $v,f(v),f(u)$ and $f(u+v)$ are all odd. It follows from \eqref{E:eoe5} that $w'y''$ is even. 
\end{remark}

Note that from the left-hand side of  \eqref{E:eoe5}, we have
\begin{equation}\label{E:eoe5i}
vf^2(v)f(u) y''^2>16.
\end{equation}
Furthermore,   \eqref{E:eoe4} and \eqref{E:eoe5} give
\[
uf(u+v) w'^2 (vf^2(v)f(u) y''^2- 16)< 4(v +u)f(v)y''^2.
\]
Hence 
\begin{equation}\label{E:eoe7}
w'^2<\frac{4}{f(u+v)}\left(\frac1{uf(u)f(v)}+\frac1{vf(v)f(u)}\right)\left(1+\frac{16 }{vf^2(v)f(u) y''^2- 16}\right)
\end{equation}
 and consequently
\begin{equation}\label{E:eoe6}
 w'^2<
4\left(\frac1{uf(u)f(v)}+\frac1{vf(v)f(u)}\right)\left(1+\frac{16 }{vf^2(v)f(u) y''^2- 16}\right).
\end{equation}

\begin{remark}
Note that as $u$ is even and the 2-adic order of $u$ is even, we have $u\ge 4$. The first twelve possible values of $u$ are the same  as the $v$ values shown  in Table~\ref{F:even}. 
\end{remark}

\begin{lemma}\label{L:5poss}  The following conditions hold.
\begin{enumerate}
\item If $f(u)>1$, then $v=1$ and $y''\ge 7$.
\item $f(u)\le3$.
\item If $y''\ge 2$, then $f(v)=1$.
\end{enumerate}
\end{lemma}

\begin{proof}(a). Suppose $f(u)>1$, so  $f(u)\ge 3$, $u\ge12$ and $uf(u)\ge36$. Suppose also that $v\ge 3$, so $vf(v)\ge 9$. Then for $y''\ge 1$, \eqref{E:eoe6} would give
\[
w'^2<4\left(\frac1{9\cdot 3}+\frac1{36}\right)\left(1+\frac{16 }{9\cdot3- 16}\right)=\frac{7}{11}< 1. \hskip1cm \cont
\]
So $v=1$. Then \eqref{E:eoe4} gives
$y''^2>uf(u)z'^2\ge 36$, so $y''\ge 7$.

(b). If $f(u)>3$, then $f(u)\ge 5$, $u\ge20$ and $uf(u)\ge100$.  Then for $y''\ge 7$, \eqref{E:eoe6} would give
\[
w'^2<4\left(\frac1{20}+\frac1{5}\right)\left(1+\frac{16 }{5\cdot7^2- 16}\right)=\frac{1029}{1145}< 1. \hskip1cm \cont
\]

(c). Suppose $y''\ge 2$ and $f(v)\ge 3$, so $vf(v)\ge 9$. For all $u\ge 4$, \eqref{E:eoe6}  gives
\[
w'^2<4\left(\frac1{9}+\frac1{4\cdot 3}\right)\left(1+\frac{16 }{27\cdot 4- 16}\right)=\frac{21}{23}< 1. \hskip1cm \cont
\]
\end{proof}

\begin{lemma}\label{L:51}  $f(u)=1$.
\end{lemma}

\begin{proof} Suppose $f(u)>1$. By Lemma~\ref{L:5poss}, $y''\ge 7, v=1,f(u)=3$ and so $u\ge 12$.
Now \eqref{E:eoe6}  gives
\[
w'^2<4\left(\frac1{3\cdot 12}+\frac1{3}\right)\left(1+\frac{16 }{3\cdot7^2- 16}\right)=\frac{637}{393}< 2,
\]
so $w'=1$. Hence \eqref{E:eoe5} gives
\begin{equation}\label{E:eoe8}
f(u+v)  (3 y''^2- 16)= 4(3z'^2+ y''^2).
\end{equation}
Modulo 3 we have $y''^2\equiv -f(u+v)$. In particular, $f(u+v)=1$ is impossible. So as $u+v$ is odd and $\gcd(u,u+v)=1$, we have $f(u+v)\ge 5$. Then \eqref{E:eoe7}  gives
\[
w'^2<\frac4{5}\left(\frac1{3\cdot 12}+\frac1{3}\right)\left(1+\frac{16 }{3\cdot7^2- 16}\right)=\frac{637}{1965}< 1. \hskip1cm \cont
\]
\end{proof}

\begin{lemma}\label{L:52} The following conditions hold.
\begin{enumerate}
\item Either $u=4$ or $u=16$.
\item If $u=4$, then $v=1$.
\item If $u=16$, then $v=9$.
\end{enumerate}
\end{lemma}

\begin{proof}(a). By Lemma~\ref{L:51}, $u$ is an even square, say $u=4n^2$. Suppose $n\ge 3$, so $u\ge 36$. First suppose that $y''=1$. Then \eqref{E:eoe4} gives
$vf(v)>uf(u)z'^2\ge 36$. 
Then \eqref{E:eoe6}  gives
\[
w'^2<4\left(\frac1{36}+\frac1{36}\right)\left(1+\frac{16 }{36- 16}\right)=\frac{2}{5}< 1. \hskip1cm \cont
\]
Hence $y''\ge 2$ and so by Lemma~\ref{L:5poss}(c), $f(v)=1$. Suppose for the moment that $v=1$. Then \eqref{E:eoe4} gives $y''^2>uf(u)z'^2\ge 36$. So $y''\ge 7$. Moreover, $u+v=4n^2+1$ is not a square, and is not divisible by $3$. So $f(u+v)\ge 5$.
Then \eqref{E:eoe7}  gives
\[
w'^2<\frac4{5}\left(\frac1{36}+\frac1{1}\right)\left(1+\frac{16 }{7^2- 16}\right)=\frac{1813}{1485}< 2,
\]
so $w'=1$. Then \eqref{E:eoe7}  gives
\[
f(u+v)<4\left(\frac1{36}+\frac1{1}\right)\left(1+\frac{16 }{7^2- 16}\right)=\frac{1813}{297}< 7,
\]
so, as $f(u+v)$ is odd and $f(u+v)\ge 5$, we have $f(u+v)=5$.
Then \eqref{E:eoe5} gives 
$y''^2= 4z'^2+ 80$. But we saw above that \eqref{E:eoe4} gives $y''^2>uf(u)z'^2 \ge 36z'^2$. So we have
$ 4z'^2+ 80\ge 36z'^2$ and hence $2z'^2<5$. Thus $z'^2=1$. But then $y''^2= 4z'^2+ 80$ has no integer solution for $y''$. 
Consequently, $v=1$ is not possible.

As $f(v)=1$, we now have $v\ge 9$. But then as $y''\ge 2$, \eqref{E:eoe6}  gives
\[
w'^2<4\left(\frac1{36}+\frac1{9}\right)\left(1+\frac{16 }{9 \cdot 4- 16}\right)=1. \hskip1cm \cont
\]
We conclude that $u=4n^2$ with $n\le 2$.

(b). Let $u=4$ and assume $v>1$.  First suppose that $y''=1$. Then \eqref{E:eoe4} gives
$vf(v)>uf(u)z'^2\ge 4$. Hence $vf(v)\ge 9$.
Also, \eqref{E:eoe5i} gives $vf^2(v)f(u) y''^2> 16$, so $vf^2(v)> 16$. So $v\not=9$, and consequently either $f(v)\ge 3$ or $v$ is an odd square with $v\ge 25$. In either case, $vf^2(v)\ge 25$.

First suppose that $v$ is an odd square with $v\ge 25$. Then $u+v=4+v$ is not a square, so $f(u+v)>1$. Moreover, as $v$ is a square $4+v\not\equiv 0\pmod 3$, so $f(u+v)\not=3$. Hence $f(u+v)\ge 5$.
Then \eqref{E:eoe7}  gives
\[
w'^2<\frac4{5}\left(\frac1{4}+\frac1{25}\right)\left(1+\frac{16 }{25- 16}\right)=\frac{29}{45}< 1. \hskip1cm \cont
\]

Now suppose  $f(v)\ge 3$. Then \eqref{E:eoe6}  gives
\[
w'^2<4\left(\frac1{4\cdot 3}+\frac1{9}\right)\left(1+\frac{16 }{3^3- 16}\right)=\frac{21}{11}< 2.
\]
So $w'=1$. But as $y''=1$, this contradicts Remark~\ref{R:even}. We conclude that $y''=1$ is not possible. 

We now consider $y''\ge 2$. By Lemma~\ref{L:5poss}(c), $f(v)=1$. Suppose $v>1$.
Since $f(v)=1$, $v$ is an odd square. So $4+v$ is not a square, and hence $f(u+v)\ge 3$.
Then, as $v\ge 9$, \eqref{E:eoe7}  gives
\[
w'^2<\frac4{3}\left(\frac1{4}+\frac1{9}\right)\left(1+\frac{16 }{9\cdot 4- 16}\right)=\frac{13}{15}< 1. \hskip1cm \cont
\]
Hence $v=1$, as required.

(c). Let $u=16$ and assume $v\not=9$. First suppose that $y''=1$ and that $f(v)=1$. Then \eqref{E:eoe4} gives
$vf(v)>uf(u)z'^2\ge 16$. As $f(v)=1$, $v$ is an odd square, it follows that $v\ge 25$.
Then \eqref{E:eoe6}  gives
\[
w'^2<4\left(\frac1{16}+\frac1{25}\right)\left(1+\frac{16 }{25- 16}\right)=\frac{41}{36}< 2,
\]
so $w'=1$. But as $y''=1$, this contradicts Remark~\ref{R:even}.

Now suppose that $y''=1$ and that $f(v)>1$. So $f(v)\ge 3$.
If $v=3$, then \eqref{E:eoe6}  gives
\[
w'^2<4\left(\frac1{16\cdot 3}+\frac1{9}\right)\left(1+\frac{16 }{27- 16}\right)=\frac{57}{44}< 2,
\]
so $w'=1$. Once again, this contradicts Remark~\ref{R:even}.

If $v>3$ then we have $v\ge 5$ and so for $f(v)\ge 3$
\eqref{E:eoe6}  gives
\[
w'^2<4\left(\frac1{16}+\frac1{15}\right)\left(1+\frac{16 }{45- 16}\right)=\frac{93}{116}< 1. \hskip1cm \cont
\]
We conclude that $y''=1$ is not possible. 

We now consider $y''\ge 2$. By Lemma~\ref{L:5poss}(c), $f(v)=1$. So $v$ is an odd square. Note that if $v\ge 25$, then 
 \eqref{E:eoe6}  gives
\[
w'^2<4\left(\frac1{16}+\frac1{25}\right)\left(1+\frac{16 }{25\cdot 4- 16}\right)=\frac{41}{84}< 1. \hskip1cm \cont
\]
So it remains to eliminate the possibility that $v=1$.

Let $v=1$. Then \eqref{E:eoe4} gives
$y''^2>uf(u)z'^2\ge 16$, so $y''\ge 5$. Also $f(u+v)=17$.
Then  \eqref{E:eoe7}  gives
\[
w'^2<\frac4{17}\left(\frac1{16}+\frac1{1}\right)\left(1+\frac{16 }{25- 16}\right)=\frac{25}{36}< 1. \hskip1cm \cont
\]
Hence $v=9$, as required.
\end{proof} 

By the previous lemma, we have $u=4,v=1$ or $u=16,v=9$. Consider the first case. Here \eqref{E:eoe5i} gives $ y''^2> 16$, so $ y''\ge  5$. Then  as $f(u+v)=5$, \eqref{E:eoe7}  gives
\[
w'^2<\frac4{5}\left(\frac1{4}+\frac1{1}\right)\left(1+\frac{16 }{25- 16}\right)=\frac{25}{9}< 3,
\]
so $w'=1$. 
Equation \eqref{E:eoe5} gives $5 ( y''^2- 16)= 4(z'^2+ y''^2)$, so $y''^2= 4z'^2+ 80$, which has the solution $y''=12,z'=4$.
From the definitions, $w=f(v)s(v)w'=w'$. Then $\Sigma= 2f(u+v)f(u) w^2=10$, and $\Sigma'= \frac{u}v\Sigma=4\Sigma=40$, which is one of the  desired solutions.

Now consider the second case,  $u=16,v=9$. Here \eqref{E:eoe5i} gives $ 9y''^2> 16$, so $ y''\ge  2$. Then  \eqref{E:eoe6}  gives
\[
w'^2<4\left(\frac1{16}+\frac1{9}\right)\left(1+\frac{16 }{9\cdot 4- 16}\right)=\frac{5}{4}< 2,
\]
so $w'=1$. Equation \eqref{E:eoe5} gives $ (9y''^2- 16)= 4(z'^2+ y''^2)$, so $5y''^2= 4z'^2+ 16$. This has infinitely many solutions.
From the definitions, $w=f(v)s(v)w'=3w'=3$.
Then $\Sigma= 2f(u+v)f(u) w^2=18$, and $\Sigma'= \frac{u}v\Sigma=16\Sigma/9=32$, which is the other desired Case 3 solution.

\bigskip
\noindent
{\bf Case 4. Assume $u$ is even, $v$ is odd, and the 2-adic order of $u$ is odd.}\

We will show that in this case, $(\Sigma,\Sigma')=(12,24)$ is the only possibility.

As $x$ is an integer,
from \eqref{E:x}  we can write $2\Sigma= f(u+v)f(u) w^2$. 
Note $v$  divides $\Sigma$, so $v$ divides $w^2$, and hence  $f(v)s(v)$ divides $w$. Thus, setting $w=f(v)s(v)w'$ we may write $2\Sigma= f(u+v)f(u) f(v)v w'^2$. Then
 \eqref{E:a1} gives
\begin{equation}\label{E:eoo1}
f(u+v)f(u)f(v)v w'^2 (vf(u) y'^2- 16)=16( z^2+vf(u) y'^2).
\end{equation}
Thus $vf(u)$ divides $16z^2$, so $f(v)s(v)\frac{f(u)}2$ divides $z$, say $2z=f(v)s(v)f(u)z'$. So \eqref{E:yz} gives $2vy' >s(u)f(v)s(v)f(u)z'$ and hence
\begin{equation}\label{E:eoo2}
2s(v)y' > f(u)s(u)z', 
\end{equation}
and \eqref{E:eoo1} gives
\begin{equation}\label{E:eoo3}
f(u+v) f(v) w'^2 (vf(u) y'^2- 16)=4(f(v) f(u)z'^2+ 4y'^2).
\end{equation}
Hence $f(v)$  divides $y'$. Let $y'=f(v)y''$.
Then \eqref{E:eoo2} gives 
\begin{equation}\label{E:eoo4}
2f(v)s(v)y'' > f(u)s(u)z'\quad \text{and so}\quad 4vf(v)y''^2>uf(u)z'^2, 
\end{equation}
and \eqref{E:eoo3} gives
\begin{equation}\label{E:eoo5}
f(u+v) w'^2 (vf^2(v)f(u) y''^2- 16)= 4(f(u)z'^2+4 f(v)y''^2).
\end{equation}
Note that from the left-hand side of  \eqref{E:eoo5}, we have
\begin{equation}\label{E:eoo5i}
vf^2(v)f(u) y''^2>16.
\end{equation}
Furthermore,   \eqref{E:eoo4} and \eqref{E:eoo5} give
\[
uf(u+v) w'^2 (vf^2(v)f(u) y''^2- 16)< 16(v +u)f(v)y''^2.
\]
Hence 
\begin{equation}\label{E:eoo7}
w'^2<\frac{16}{f(v+u)}\left(\frac1{uf(u)f(v)}+\frac1{vf(v)f(u)}\right)\left(1+\frac{16 }{vf^2(v)f(u) y''^2- 16}\right)
\end{equation}
 and consequently
\begin{equation}\label{E:eoo6}
 w'^2<
16\left(\frac1{uf(u)f(v)}+\frac1{vf(v)f(u)}\right)\left(1+\frac{16 }{vf^2(v)f(u) y''^2- 16}\right).
\end{equation}


\bigskip
\begin{lemma}\label{L:mult6}  The following conditions hold.
\begin{enumerate}
\item $f(v)\le 5$.
\item If $f(v)>1$, then $f(u)=2$.
\item If $f(v)=3$, then either $v=3$ or $u=2$.
\end{enumerate}
\end{lemma}

\begin{proof}(a). Suppose that  $f(v)\ge 7$. So $vf(v)\ge 7^2$ and $vf^2(v)\ge 7^3$. Also, as the 2-adic order of $u$ is odd, we have $f(u)\ge 2$, so $u\ge 2$ and $uf(u)\ge 4$. Then for all $y''\ge 1$,
 \eqref{E:eoo6} gives
\begin{align*}
w'^2&<16\left(\frac1{4\cdot 7}+\frac1{ 7^2\cdot2}\right)\left(1+\frac{16 }{7^3\cdot 2- 16}\right)=\frac{252}{335}<1. \hskip1cm \cont
\end{align*}

(b). Suppose that $f(v)\ge 3$ and $f(u)>2$. Then $f(u)\ge 6$, and so $uf(u)\ge 36$. And $vf(v)\ge 9$ and $vf^2(v)\ge 27$. Then for all $y''\ge 1$,
 \eqref{E:eoo6} gives
\begin{align*}
w'^2&<16\left(\frac1{36\cdot 3}+\frac1{9\cdot6}\right)\left(1+\frac{16 }{27\cdot 6- 16}\right)=\frac{36}{73}<1. \hskip1cm \cont
\end{align*}

(c).  Suppose that $f(v)=3$ and that $v>3$ and $u>2$. As $f(v)=3$, $v$ has the form $v=3m^2$, for some odd $m$. By Part (b),  $f(u)=2$, so $u$ has the form $u=2n^2$, for some integer $n$. So $v\ge 27$ and $u\ge 8$. 
Then  \eqref{E:eoo6} gives
\begin{align*}
w'^2&<16\left(\frac1{16\cdot 3}+\frac1{3^4\cdot2}\right)\left(1+\frac{16 }{3^5\cdot 2- 16}\right)=\frac{21}{47}<1. \hskip1cm \cont
\end{align*}
\end{proof}

\begin{lemma}\label{L:fv136}   $f(v)=1$.
\end{lemma}

\begin{proof} 
By Lemma \ref{L:mult6}(a), $f(v)\le 5$. First suppose $f(v)=5$. By Lemma~\ref{L:mult6}(b),  $f(u)=2$. If $u>2$, then $u\ge 8$ and $uf(u)\ge 16$.
Then \eqref{E:eoo6} gives
\begin{align*}
w'^2&<16\left(\frac1{16\cdot 5}+\frac1{25\cdot 2}\right)\left(1+\frac{16 }{5^3\cdot 2- 16}\right)=\frac{5}{9}<1. \hskip1cm \cont
\end{align*}
So $u=2$. But  $v$ has the form $v=5m^2$, for some odd $m$, so $u+v=2+5m^2$, and this cannot be a square as $2+5m^2\equiv 3\pmod 4$. Hence $f(u+v)\ge 3$. Then by \eqref{E:eoo7} gives
\begin{align*}
w'^2&<\frac{16}3 \left(\frac1{4\cdot 5}+\frac1{25\cdot 2}\right)\left(1+\frac{16 }{5^3\cdot 2- 16}\right)=\frac{140}{351}<1. \hskip1cm \cont
\end{align*}
So $f(v)\not=5$. 

Now suppose $f(v)=3$, so $v$ has the form $v=3m^2$, for some odd $m$. By Lemma~\ref{L:mult6}(b),  $f(u)=2$,  so $u$ has the form $u=2n^2$, for some integer $n$. By Lemma~\ref{L:mult6}(c),  either $v=3$ or $u=2$. We claim that in both cases, $u+v$ is not a square. Indeed, if $v=3$, then $u+v=2n^2+3$ isn't a square since modulo 8, $2n^2+3$ is either $3$ or $5$, according to whether $n$ is even or odd, but the quadratic residues modulo 8 are $0,1$ and $4$. Similarly, if $u=2$, then $u+v=2+3m^2$ isn't a square as $2+3m^2\equiv 2\pmod 3$. 
Thus, in both cases, $f(u+v)>1$ and so, as $\gcd(u+v,v)=1$, we have $f(u+v)\ge 5$. Then by \eqref{E:eoo7} gives
\begin{align*}
w'^2&<\frac{16}5 \left(\frac1{4\cdot 3}+\frac1{2\cdot 3^2}\right)\left(1+\frac{16 }{3^3\cdot 2- 16}\right)=\frac{12}{19}<1. \hskip1cm \cont
\end{align*}
\end{proof}

\begin{lemma}\label{L:y6}  If $y''>1$, then $v=1$.
\end{lemma}

\begin{proof} By the previous lemma, $f(v)=1$, so $v$ is an odd square, say $v=m^2$. 
Suppose $y''\ge 2$ and that  $v\ge 9$. First suppose that $f(u)>2$. Then $f(u)\ge 6$ and \eqref{E:eoo6} gives
\begin{align*}
w'^2&<16\left(\frac1{36}+\frac1{9\cdot 6}\right)\left(1+\frac{16 }{9\cdot 6\cdot 4- 16}\right)=\frac{4}{5}<1. \hskip1cm \cont
\end{align*}
So $f(u)=2$. Hence $u$ has the form $u=2n^2$, for some integer $n$. 

Now suppose for the moment that $v\ge 25$ and that $u\ge 18$. Then $uf(u)\ge 36$ and \eqref{E:eoo6} gives
\begin{align*}
w'^2&<16\left(\frac1{36}+\frac1{25\cdot 2}\right)\left(1+\frac{16 }{25\cdot 2\cdot 4- 16}\right)=\frac{172}{207}<1. \hskip1cm \cont
\end{align*}
So either $v=9$ or  $u< 18$. First consider the case where $u< 18$ and $v\ge 25$. There  are two possibilities: either $u=2$ or $u=8$.
If $u=8$, then $u+v=8+m^2$ cannot be a square for $m>1$. So $f(u+v)\ge 3$. Then $uf(u)=16$, and
\eqref{E:eoo7} gives
\begin{align*}
w'^2&<\frac{16}3 \left(\frac1{16}+\frac1{25\cdot 2}\right)\left(1+\frac{16 }{25\cdot 2\cdot 4- 16}\right)=\frac{11}{23}<1. \hskip1cm \cont
\end{align*}
If $u=2$, then $uf(u)=4$, and $u+v=2+m^2$ cannot be a square. So $f(u+v)\ge 3$. Then 
\eqref{E:eoo7} gives
\begin{align*}
w'^2&<\frac{16}3 \left(\frac1{4}+\frac1{25\cdot 2}\right)\left(1+\frac{16 }{25\cdot 2\cdot 4- 16}\right)=\frac{36}{23}<2,
\end{align*}
so $w'=1$. But then \eqref{E:eoo6} gives
\begin{align*}
f(u+v)< 16\left(\frac1{4}+\frac1{25\cdot 2}\right)\left(1+\frac{16 }{25\cdot 2\cdot 4- 16}\right)=\frac{108}{23}<5. 
\end{align*}
So $f(u+v)=3$. Then \eqref{E:eoo5} gives
$3 (2v y''^2- 16)= 4(2z'^2+4 y''^2)$, so 
\[(3 v-8) y''^2= 4z'^2+24.
\]
So as $v$ is odd, $y''$ must be even, say $y''=2y'''$, so $(3 v-8) y'''^2= z'^2+6$. But it is easy to see that as $v$ is an odd square, this equation has no solution modulo 8. 

We conclude from the above that $v=9$. In this case, for $u\ge 2$, \eqref{E:eoo6} gives
\[
f(u+v)w'^2< 16 \left(\frac1{4}+\frac1{9\cdot 2}\right)\left(1+\frac{16 }{9\cdot 2\cdot 4- 16}\right)=\frac{44}{7}<7.
\]
Notice that $f(u+v)\not=3$ since $\gcd(u+v,v)=1$. So we have three possibilities:
\begin{enumerate}
\item $f(u+v)=1$ and $w'=1$,
\item $f(u+v)=1$ and $w'=2$,
\item $f(u+v)=5$ and $w'=1$.
\end{enumerate}
In these cases, \eqref{E:eoo5} gives respectively
\begin{align}
   y''^2&= 4z'^2+8,\label{E:cs1}\\
  7 y''^2&= z'^2+8,\label{E:cs2}\\
  37 y''^2&= 4z'^2+40.\label{E:cs3}
 \end{align} 
However, one finds that Equation~\eqref{E:cs1} has no solution modulo 16, 
Equation~\eqref{E:cs2} has no solution modulo 32, and 
Equation~\eqref{E:cs3} has no solution modulo 25.
This completes the proof of the lemma.
\end{proof}

\begin{lemma}\label{L:116}  If $y''=1$, then $v=1$.
\end{lemma}

\begin{proof} By Lemma \ref{L:fv136}, $f(v)=1$, so $v$ is an odd square, say $v=m^2$. 
Suppose $y''=1$ and that $v\ge 9$.
First note that if  $f(u)>2$, then $f(u)\ge 6$, and so $uf(u)\ge 36$.  Then by \eqref{E:eoo4},
$4v>uf(u)z'^2\ge uf(u)\ge 36$, so $v>9$. Thus $v\ge 25$ and then 
 \eqref{E:eoo6} gives
\begin{align*}
w'^2&<16\left(\frac1{36}+\frac1{25\cdot 6}\right)\left(1+\frac{16 }{25\cdot 6- 16}\right)=\frac{124}{201}<1. \hskip1cm \cont
\end{align*}
So $f(u)=2$. 
Then \eqref{E:eoo5} gives 
\begin{equation}\label{E:eoo8}
f(u+v) w'^2 (m^2 - 8)= 4(z'^2+2).
\end{equation}
In particular, as $m$ and $f(u+v)$ are odd, $w'$ must be even.
For all $u\ge 2$, \eqref{E:eoo6} gives
\[
f(u+v)w'^2< 16 \left(\frac1{4}+\frac1{9\cdot 2}\right)\left(1+\frac{16 }{9\cdot 2- 16}\right)=44.
\]
So, as $w'$ is even and $f(u+v)$ is odd and square-free, we have three possibilities:
\begin{enumerate}
\item $w'=6$ and $f(u+v)=1$; here \eqref{E:eoo8} gives $9m^2= z'^2+74$.
\item $w'=4$ and $ f(u+v)=1$; here \eqref{E:eoo8} gives $ 4 (m^2 - 8)= z'^2+2$.
\item $w'=2$ and $ f(u+v)=1,3,5,7$; here \eqref{E:eoo8} gives $f(u+v)  (m^2 - 8)= z'^2+2$.
\end{enumerate}
However, in the first two cases, the equation has no solution modulo 4.
In the third case we find that for $f(u+v)=1$ and $5$, the equation $f(u+v)  (m^2 - 8)= z'^2+2$ also has no solution modulo 4,
while for $f(u+v)=7$, the equation $f(u+v)  (m^2 - 8)= z'^2+2$ has no solution modulo 8.
So it remains to treat the case where $f(u+v)=3$ and $w'=2$.
So it remains to treat the case where $f(u+v)=3$ and $w'=2$.
Here the equation is $3m^2= z'^2+26$, which actually does have integer solutions. However, notice that for $f(u+v)=3$, we have $v\not=9$, since $\gcd(u+v,v)=1$, so $v\ge 25$.  Hence for $u\ge 2$, 
\eqref{E:eoo7} gives
\begin{align*}
w'^2&<\frac{16}3 \left(\frac1{4}+\frac1{25\cdot 2}\right)\left(1+\frac{16 }{25\cdot 2- 16}\right)=\frac{36}{17}<3. 
\end{align*}
But this contradicts the assumption that $w'=2$.
\end{proof}

By the two preceding lemmas, $v=1$. 

\begin{lemma}\label{L:y16}   $y''>1$.
\end{lemma}

\begin{proof} If $y''=1$, then by \eqref{E:eoo4}, we have $uf(u)z'^2<4$. But this is impossible as $uf(u)\ge 4$.
\end{proof}

\begin{lemma}\label{L:hum6}   If  $f(u+1)>1$, then $f(u)=2$.
\end{lemma}

\begin{proof} By the previous lemma, we have $y''\ge 2$. Suppose $f(u+1)>1$.
Note that if $f(u)\ge 10$, then  $f(u+1)\ge 3$ and
\eqref{E:eoo7} gives
\begin{align*}
w'^2&<\frac{16}3 \left(\frac1{10^2}+\frac1{10}\right)\left(1+\frac{16 }{10\cdot  4- 16}\right)=\frac{44}{45}<1. \hskip1cm \cont
\end{align*}
Furthermore,  if $f(u)=6$ and $y''\ge 3$, then as $\gcd(u,u+1)=1$, we have $f(u+1)\not=3$, so $f(u+1)\ge 5$, and hence
\eqref{E:eoo7} gives
\begin{align*}
w'^2&<\frac{16}5 \left(\frac1{6^2}+\frac1{6}\right)\left(1+\frac{16 }{6\cdot  9- 16}\right)=\frac{84}{95}<1. \hskip1cm \cont
\end{align*}
Finally, suppose that $f(u)=6$ and $y''=2$. Then $f(u+1)\ge 5$ and \eqref{E:eoo7} gives
\begin{align*}
w'^2&<\frac{16}5 \left(\frac1{6^2}+\frac1{6}\right)\left(1+\frac{16 }{6\cdot  4- 16}\right)=\frac{28}{15}<2. 
\end{align*}
So $w'=1$. Then applying \eqref{E:eoo7} again gives
\begin{align*}
f(u+1)&<16 \left(\frac1{6^2}+\frac1{6}\right)\left(1+\frac{16 }{6\cdot  4- 16}\right)=\frac{28}{3}<10, 
\end{align*}
So as $f(u+1)$ is odd and square-free and $f(u+1)\ge 5$, we have $f(u+1)=5$ or $7$. Hence from \eqref{E:eoo5} we have
$f(u+1) = 3z'^2+8$, which is impossible for $f(u+1)=5$ and $7$.
\end{proof}

\begin{lemma}\label{L:fu6}   If  $f(u+1)=1$, then $f(u)=2$.
\end{lemma}

\begin{proof} Suppose $f(u+1)=1$, so $u+1=r^2$ for some odd $r$. So $u\equiv0\pmod8$. 
Suppose that $f(u)>2$. Then $f(u)\ge 6$ and as $u$ is divisible by 8,
$u\ge 24$. So, as $y''\ge 2$ by Lemma~\ref{L:y16},
\eqref{E:eoo6} gives
\begin{align*}
w'^2&< 16  \left(\frac1{24\cdot 6}+\frac1{6}\right)\left(1+\frac{16 }{6\cdot  4- 16}\right)=\frac{25}{3}<9. 
\end{align*}
So $w'=1$ or $2$. First suppose $w'=1$. Then \eqref{E:eoo5} gives
$ (f(u) -16)y''^2= 4f(u)z'^2+16$.
In particular, $f(u)>16$. Thus, as $f(u)$ is even and square-free, $f(u)\ge 22$. So, as $u$ is divisible by 8, $u\ge 88$. But then 
\eqref{E:eoo6} gives
\begin{align*}
w'^2&< 16  \left(\frac1{88\cdot 22}+\frac1{22}\right)\left(1+\frac{16 }{22\cdot  4- 16}\right)=\frac{89}{99}<1. \hskip1cm \cont
\end{align*}
So $w'=2$. Then \eqref{E:eoo5} gives
\begin{equation}\label{E:eoo9}
 (f(u) -4)y''^2= f(u)z'^2+16.
\end{equation}
By \eqref{E:eoo4}, we have $4y''^2>uf(u)z'^2\ge 24f(u)z'^2$. So \eqref{E:eoo9} gives 
$ (f(u) -4)y''^2< \frac16y''^2+16$ and hence
$(6f(u) -25)y''^2 <96$. But $f(u)\ge 6$ and $y''\ge 2$, so in fact, as $f(u)$ is even and square-free, the only possibility is 
 $f(u)=6$ and $y''=2$. Then  \eqref{E:eoo9} gives  $8=6z'^2+16$, which is impossible.
\end{proof}

\begin{lemma}\label{L:fu62}   $u=2$.
\end{lemma}

\begin{proof} From the two preceding lemmas, we have $v=1$ and $f(u)=2$. So $u$ has the form $u=2n^2$ for some $n$. Suppose $n>1$, so $uf(u)\ge 16$.
Equation \eqref{E:eoo5} gives
\begin{equation}\label{E:u26}
f(u+1) w'^2 ( y''^2- 8)= 4(z'^2+2 y''^2).
\end{equation}
In particular, $y''^2>8$, so $y''\ge 3$.
By \eqref{E:eoo4}, we have $4y''^2>uf(u)z'^2\ge16 z'^2$, so 
$y''>2 z'$. 
So when $y''=3$ or $4$, we obtain $z'=1$. Then when $y''=4$,
\eqref{E:u26} gives $2f(u+1) w'^2 = 33$, which is impossible modulo 2.
When $y''=3$,
\eqref{E:u26} gives $f(u+1) w'^2 = 4\cdot 19$, which implies necessarily $w'=2$ and $f(u+1) =19$. Moreover the smallest value of $u$ with $f(u)=2$ and $f(u+1) =19$ is $u=18$. But then $uf(u)\ge 36$ and with $y''=3$,
\eqref{E:eoo7} gives
\begin{align*}
f(u+1) w'^2&< 16  \left(\frac1{36}+\frac1{2}\right)\left(1+\frac{16 }{2\cdot  9- 16}\right)=76,
\end{align*}
which gives a contradiction. So we have $y''\ge 5$.  Then for $u\ge 8$, \eqref{E:eoo7} gives
\begin{align*}
f(u+1) w'^2&< 16  \left(\frac1{16}+\frac1{2}\right)\left(1+\frac{16 }{2\cdot  25- 16}\right)=\frac{225}{17}<14.
\end{align*}
Notice also that \eqref{E:u26}  can be rearranged to give $(f(u+1) w'^2-8) y''^2= 4z'^2+8 f(u+1) w'^2$, so $f(u+1) w'^2>8$.
So $9\le f(u+1) w'^2\le 13$, and since $f(u+1)$ is odd and square-free, we have only the following possibilities:
\begin{enumerate}
\item $w'=1$ and $f(u+1)=11,13$; here \eqref{E:u26} gives $(f(u+1)-8) y''^2= 4z'^2+8 f(u+1)$.
\item $w'=2$ and $ f(u+1)=3$; here \eqref{E:u26} gives $y''^2= z'^2+24$.
\item $w'=3$ and $ f(u+1)=1$; here \eqref{E:u26} gives $y''^2= 4z'^2+72$.
\end{enumerate}
In the first case, with  $f(u+1)=11$, the equation is $3y''^2=4z'^2+88$, which has no solution modulo 32.   
In the first case, with  $f(u+1)=13$, the equation is $5y''^2=4z'^2+104$, which has no solution modulo 16. 
In the third case,  the equation is $y''^2= 4z'^2+72$ has no solution modulo 16.

It remains to deal with the second case, where the equation $y''^2= z'^2+24$ has the solution $y''=5,z'=1$. Note that in this case $ f(u+1)=3$. But $u=2n^2$ and  the smallest value of $n>1$ for which $ f(u+1)=3$ is $n=11$. Here $u=242$ and 
\eqref{E:eoo7} gives
\begin{align*}
f(u+1) w'^2&< 16  \left(\frac1{242 \cdot 2}+\frac1{2}\right)\left(1+\frac{16 }{2\cdot  25- 16}\right)=\frac{24300}{2057}<12,
\end{align*}
contradicting the assumption that $w'=2$ and $ f(u+1)=3$.
 \end{proof}

From the preceding lemmas, we have $v=1$ and $u=2$.  By \eqref{E:eoo5i}, we have $2y''^2> 16$, so $y''\ge 3$. We have $ f(u+v)=3$ and so 
\eqref{E:eoo7} gives
\begin{align*}
w'^2&< \frac{16}3  \left(\frac1{4}+\frac1{2}\right)\left(1+\frac{16 }{2\cdot  9- 16}\right)=36,
\end{align*}
so $w'\le 5$. Moreover, \eqref{E:eoo5} gives
\begin{equation}\label{E:6end}
(3 w'^2 -8) y''^2= 4z'^2+24w'^2 ,
\end{equation}
so $w'\ge2$. So there are four possibilities.

If $w'=5$, \eqref{E:6end} gives $67y''^2= 4z'^2+600$, which has no solutions modulo 32. 

If $w'=4$, \eqref{E:6end} gives $10y''^2= z'^2+96$ (which has the solution $y''=4,z'=8$). But by \eqref{E:eoo4}, we have $y''>z'$, so $10y''^2= z'^2+96$ gives
$9y''^2<96$, giving $y''\le 3$. So, as $y''\ge 3$, from above, we have $y''= 3$. But then $10y''^2= z'^2+96$ has no integer solution for $z'$.

If $w'=3$, \eqref{E:6end} gives $19y''^2= 4z'^2+216$,  which has no solution modulo 32. 

Finally, if $w'=2$, \eqref{E:6end} gives $y''^2= z'^2+24$, which has the solution $y''=5,z'=1$. Note that in this case 
\[
\Sigma= \frac12f(u+v)f(u) vf(v)w'^2=3 w'^2=12, 
\]
and $\Sigma'=\frac{u}v \Sigma=24$. This is our desired Case 4 solution.

\bigskip
\noindent
{\bf Case 5. Assume $u,v$ are both odd and the 2-adic order of $u+v$ is even.}\

We will show that there are no solutions in this case.

As $x$ is an integer,
from \eqref{E:x}  we can write $\Sigma= 2f(u+v)f(u) w^2$, for some $w$. 
Note $v$  divides $\Sigma$, so $v$ divides $w^2$, and hence  $f(v)s(v)$ divides $w$. Thus, setting $w=f(v)s(v)w'$ we may write $\Sigma= 2f(u+v)f(u) f(v)v w'^2$. Then
  \eqref{E:a1} gives
\begin{equation}\label{E:ooe1}
f(u+v)f(u)f(v)v w'^2 (vf(u) y'^2- 16)=4( z^2+vf(u) y'^2).
\end{equation}
Thus $vf(u)$ divides $4z^2$, so $f(v)s(v)f(u)$ divides $z$, say $z=f(v)s(v)f(u)z'$. So \eqref{E:yz} gives $vy' >s(u)f(v)s(v)f(u)z'$ and hence
\begin{equation}\label{E:ooe2}
s(v)y' > f(u)s(u)z', 
\end{equation}
and \eqref{E:ooe1} gives
\begin{equation}\label{E:ooe3}
f(u+v) f(v) w'^2 (vf(u) y'^2- 16)=4(f(v) f(u)z'^2+ y'^2).
\end{equation}
Hence $f(v)$  divides $y'$. Let $y'=f(v)y''$.
Then \eqref{E:ooe2} gives 
\begin{equation}\label{E:ooe4}
f(v)s(v)y'' > f(u)s(u)z'\quad \text{and so}\quad vf(v)y''^2>uf(u)z'^2, 
\end{equation}
and \eqref{E:ooe3} gives
\begin{equation}\label{E:ooe5}
f(u+v) w'^2 (vf^2(v)f(u) y''^2- 16)= 4(f(u)z'^2+ f(v)y''^2).
\end{equation}
Now   \eqref{E:ooe4} and \eqref{E:ooe5} give
\[
uf(u+v) w'^2 (vf^2(v)f(u) y''^2- 16)< 4(v +u)f(v)y''^2.
\]
Hence
\begin{equation}\label{E:ooe7}
w'^2<\frac{4}{f(v+u)}\left(\frac1{uf(u)f(v)}+\frac1{vf(v)f(u)}\right)\left(1+\frac{16 }{vf^2(v)f(u) y''^2- 16}\right)
\end{equation}
 and consequently
\begin{equation}\label{E:ooe6}
 w'^2<
4\left(\frac1{uf(u)f(v)}+\frac1{vf(v)f(u)}\right)\left(1+\frac{16 }{vf^2(v)f(u) y''^2- 16}\right).
\end{equation}

\begin{remark}\label{R:mod4} Note that using the  hypothesis that the 2-adic order of $u+v$ is even, one has 
\[
f(u)+f(v)\equiv u+v\equiv 0 \ \pmod 4.
\]
In particular, $f(u)$ and $f(v)$ are not both 1. 
\end{remark}

\begin{lemma}\label{L:mult}  The following conditions hold.
\begin{enumerate}
\item $f(u)\le 19$.
\item $f(v)\le 5$.
\item If $y''>1$, then 
$f(u)\le 7$.
\end{enumerate}
\end{lemma}

\begin{proof}(a). For $f(u)\ge 21$,  we have $uf(u)\ge 21^2$, so for all $y''\ge 1$, \eqref{E:ooe6} gives
\begin{align*}
w'^2&<4\left(\frac1{21^2}+\frac1{21}\right)\left(1+\frac{16 }{21- 16}\right)=\frac{88}{105}<1. \hskip1cm \cont
\end{align*}

(b). Suppose $f(v)\ge 7$. Then \eqref{E:ooe6} gives
\begin{align*}
w'^2&<4\left(\frac1{49}+\frac1{7}\right)\left(1+\frac{16 }{7^3- 16}\right)=\frac{224}{327}<1. \hskip1cm \cont
\end{align*}

(c). Suppose $y''\ge 2$  
and $f(u)> 7$. As $f(u)$ is a square-free odd number, $f(u)\ge 11$. Then \eqref{E:ooe6} gives
\begin{align*}
w'^2&<4\left(\frac1{11^2}+\frac1{11}\right)\left(1+\frac{16 }{11\cdot 4- 16}\right)=\frac{48}{77}<1. \hskip1cm \cont
\end{align*}
\end{proof}

\begin{lemma}\label{L:fv131}   $f(v)=1$
\end{lemma}

\begin{proof} 
Suppose $f(v)>3$. By Lemma \ref{L:mult}(b), $f(v)= 5$. Then by Remark~\ref{R:mod4}, $f(u)\equiv 3\pmod 4$, and in particular, $f(u)\ge 3$ and so $u\ge 3$. Then \eqref{E:ooe6} gives
\begin{align*}
w'^2&<4\left(\frac1{25\cdot 3}+\frac1{5\cdot 9}\right)\left(1+\frac{16 }{125\cdot 3- 16}\right)=\frac{160}{1077}<1. \hskip1cm \cont
\end{align*}
So $f(v)=1$ or $3$.
Suppose that $f(v)=3$.  So $v\ge 3$. By  Remark~\ref{R:mod4}, $f(u)\equiv 1\pmod 4$.
Suppose for the moment that $f(u)\ge 5$. Then $u\ge 5$ and 
 \eqref{E:ooe6} gives
\begin{align*}
w'^2&<4\left(\frac1{5^2\cdot 3}+\frac1{3^2\cdot 5}\right)\left(1+\frac{16 }{3^3\cdot 5- 16}\right)=\frac{96}{595}<1. \hskip1cm \cont
\end{align*}
So $f(u)=1$.

As  $u,v$ are relatively prime, and as $v$ is divisible by $3$ since $f(v)=3$, we have $u\not\equiv0\pmod 3$.
Furthermore, as $f(u)=1$, $u$ is a square. So $u\equiv1\pmod 3$, and thus $u+v\equiv1\pmod 3$, and hence $f(u+v)\equiv1\pmod 3$.
By \eqref{E:ooe6},
\begin{align*}
w'^2&<4\left(\frac1{3^2}+\frac1{3}\right)\left(1+\frac{16 }{3^3- 16}\right)=\frac{48}{11}<5,
\end{align*}
so $w'=1$ or $2$. In particular, $w'^2\equiv1\pmod 3$.
Then, using $f(u)=1$ in \eqref{E:ooe5}  gives $f(u+v) w'^2 (9v y''^2- 16)= 4(z'^2+ 3y''^2)$, and modulo 3 we have
$z'^2\equiv-1$, which is impossible. So $f(v)=1$.
\end{proof}

\begin{lemma}\label{L:fv1}   $v=1$.
\end{lemma}

\begin{proof}  
By the previous lemma, $f(v)=1$, so $v$ is an odd square, $v=m^2$ say. Suppose that $v>1$, so $v\ge 9$. By  Remark~\ref{R:mod4}, $f(u)\equiv 3\pmod 4$.
So $f(u)\ge 3$. Note that if $f(u)\ge 7$, then $u\ge 7$ and so
 \eqref{E:ooe6} gives
\begin{align*}
w'^2&<4\left(\frac1{7^2}+\frac1{7\cdot 9}\right)\left(1+\frac{16 }{9\cdot 7- 16}\right)=\frac{64}{329}<1. \hskip1cm \cont
\end{align*}
So if $v>1$, then $f(u)=3$. In this case, $u$ is divisible by 3 and so as $u,v$ are relatively prime, by hypothesis, we have $v\ge 25$. Then  \eqref{E:ooe6} gives
\begin{align*}
w'^2&<4\left(\frac1{3^2}+\frac1{3\cdot 25}\right)\left(1+\frac{16 }{25\cdot 3 - 16}\right)=\frac{112}{177}<1. \hskip1cm \cont
\end{align*}
Thus $v=1$.
\end{proof}

\begin{lemma}\label{L:fv13}   $f(u)=3$.
\end{lemma}

\begin{proof}  From the previous lemma, $v=1$.
From \eqref{E:ooe4} we have  $y''>f(u)z'\ge 1$. So from Lemma \ref{L:mult}(c) we have $f(u)\le 7$. 
Suppose $f(u)= 7$. Then $u\ge 7$ and from \eqref{E:ooe4} we have  $y''>7z'\ge 7$. So $y''\ge 8$. 
Then  \eqref{E:ooe6} gives
\begin{align*}
w'^2&<4\left(\frac1{7^2}+\frac1{7}\right)\left(1+\frac{16 }{7\cdot 8^2- 16}\right)=\frac{128}{189}<1. \hskip1cm \cont
\end{align*}
So $f(u)\le 5$ and since $f(u)\equiv 3\pmod 4$, we have $f(u)=3$.
\end{proof}

From the above lemmas, we have $v=1$ and $ f(u)=3$. 
From \eqref{E:ooe4}, we have $y'' > f(u)z'\ge 3$. Then by \eqref{E:ooe6},
\begin{align*}
w'^2&<4\left(\frac1{3^2}+\frac1{3}\right)\left(1+\frac{16 }{3^3- 16}\right)=\frac{48}{11}<5,
\end{align*}
so $w'=1$ or $2$. In particular, $w'^2\equiv1\pmod 3$.
As 
$f(u)=3$, we have that $u$ is divisible by 3. So $u+v\equiv1\pmod 3$, and hence $f(u+v)\equiv1\pmod 3$.
Substituting $v=1, f(u)=3$ in \eqref{E:ooe5}  gives $f(u+v) w'^2 (3 y''^2- 16)= 4(3z'^2+ y''^2)$, and thus modulo 3 we obtain
$y''^2\equiv-1$, which is impossible. So there there are no solutions in Case 5.

\bigskip
\noindent
{\bf Case 6. Assume $u,v$ are both odd and  the 2-adic order of $u+v$ is odd. }\

We will show that in this case, one of the following holds:
\begin{enumerate}
\item $(\Sigma,\Sigma')=(9,9)$ or $(16,16)$, 
\item $(\Sigma,\Sigma')=(m^2,1)$, for some integer $m$ satisfying the equations $m^2+1=2n^2$ and
$(m^2-8)Y^2=1+8Z^2$ for some integers $n,Y,Z$,
\item  $(\Sigma,\Sigma')=(5m^2,5)$, for some integer $m$ satisfying the equations $m^2+1=10n^2$ and
$(5m^2-8)Y^2=5+8Z^2$ for some integers $n,Y,Z$.
\end{enumerate}

As $x$ is an integer,
from \eqref{E:x}  we can write $2\Sigma= f(u+v)f(u) w^2$. 
Note $v$  divides $\Sigma$, so $v$ divides $w^2$, and hence  $f(v)s(v)$ divides $w$. Thus, setting $w=f(v)s(v)w'$ we may write $2\Sigma= f(u+v)f(u) f(v)v w'^2$. Then
\eqref{E:a1} gives
\begin{equation}\label{E:ooo1}
f(u+v)f(u)f(v)v w'^2 (vf(u) y'^2- 16)=16( z^2+vf(u) y'^2).
\end{equation}
Thus $vf(u)$ divides $16z^2$, so $f(v)s(v)f(u)$ divides $z$, say $z=f(v)s(v)f(u)z'$. So \eqref{E:yz} gives $vy' >s(u)f(v)s(v)f(u)z'$ and hence
\begin{equation}\label{E:ooo2}
s(v)y' > f(u)s(u)z',
\end{equation}
and \eqref{E:ooo1} gives
\begin{equation}\label{E:ooo3}
f(u+v) f(v) w'^2 (vf(u) y'^2- 16)=16(f(v) f(u)z'^2+ y'^2).
\end{equation}
Hence $f(v)$  divides $y'$. Let $y'=f(v)y''$.
Then \eqref{E:ooo2} gives 
\begin{equation}\label{E:ooo4}
f(v)s(v)y'' > f(u)s(u)z'\quad \text{and so}\quad vf(v)y''^2>uf(u)z'^2, 
\end{equation}
and \eqref{E:ooo3} gives
\begin{equation}\label{E:ooo5}
f(u+v) w'^2 (vf^2(v)f(u) y''^2- 16)= 16(f(u)z'^2+ f(v)y''^2).
\end{equation}
Note that from the left-hand side of  \eqref{E:ooo5}, we have
\begin{equation}\label{E:ooo5i}
vf^2(v)f(u) y''^2> 16.
\end{equation}
Furthermore,  \eqref{E:ooo4} and \eqref{E:ooo5} give
\[
uf(u+v) w'^2 (vf^2(v)f(u) y''^2- 16)< 16(v +u)f(v)y''^2.
\]
Hence 
\begin{equation}\label{E:ooo7}
w'^2<\frac{16}{f(u+v)}\left(\frac1{uf(u)f(v)}+\frac1{vf(v)f(u)}\right)\left(1+\frac{16 }{vf^2(v)f(u) y''^2- 16}\right)
\end{equation}
 and consequently, as $f(u+v)\ge 2$,
\begin{equation}\label{E:ooo6}
 w'^2<
8\left(\frac1{uf(u)f(v)}+\frac1{vf(v)f(u)}\right)\left(1+\frac{16 }{vf^2(v)f(u) y''^2- 16}\right).
\end{equation}

\begin{remark}\label{R:mod44} As the 2-adic order of $u+v$ is odd, one has 
\[
f(u)+f(v)\equiv u+v\equiv 2 \ \pmod 4.
\]
\end{remark}

\begin{lemma}\label{L:mult4} The following conditions hold.
\begin{enumerate}
\item $f(u)\le 23$.
\item $f(v)\le 7$.
\item If $y''>1$, then 
$f(u)\le 11$.
\end{enumerate}
\end{lemma}

\begin{proof}(a). If $f(u)> 23$, then as $f(u)$ is square-free, $f(u)\ge 29$, so for all $y''\ge 1$, \eqref{E:ooo6} gives
\begin{align*}
w'^2&<8\left(\frac1{29^2}+\frac1{29}\right)\left(1+\frac{16 }{29- 16}\right)=\frac{240}{377}<1. \hskip1cm \cont
\end{align*}

(b). If $f(v)> 7$, then as $f(v)$ is square-free, $f(v)\ge 11$, and \eqref{E:ooo6} gives
\begin{align*}
w'^2&<8\left(\frac1{121}+\frac1{11}\right)\left(1+\frac{16 }{11^3- 16}\right)=\frac{1056}{1315}<1. \hskip1cm \cont
\end{align*}

(c). If $y''\ge 2$  
and $f(u)\ge 13$, then \eqref{E:ooo6} gives
\begin{align*}
w'^2&<8\left(\frac1{13^2}+\frac1{13}\right)\left(1+\frac{16 }{13\cdot 4- 16}\right)=\frac{112}{117}<1. \hskip1cm \cont
\end{align*}
\end{proof}

\begin{lemma}\label{L:fv134}   $f(v)=1$.
\end{lemma}

\begin{proof} 
By Lemma \ref{L:mult4}(b), $f(v)\le 7$. First suppose $f(v)=7$. Then by Remark~\ref{R:mod44}, $f(u)\equiv 3\pmod 4$, and in particular, $f(u)\ge 3$ and so $u\ge 3$. Then \eqref{E:ooo6} gives
\begin{align*}
w'^2&<8\left(\frac1{49\cdot 3}+\frac1{7\cdot 9}\right)\left(1+\frac{16 }{7^3\cdot 3- 16}\right)=\frac{560}{3039}<1. \hskip1cm \cont
\end{align*}
So $f(v)\le 5$.

Now suppose that $f(v)=5$. Then by Remark~\ref{R:mod44}, $f(u)\equiv 1\pmod 4$. Suppose for the moment that $f(u)>1$. Then as $u,v$ are relatively prime and $v$ is divisible by $5$, we have $f(u)\ge 13$. So 
 \eqref{E:ooo6} gives
\begin{align*}
w'^2&<8\left(\frac1{13^2\cdot 5}+\frac1{5^2\cdot 13}\right)\left(1+\frac{16 }{5^3\cdot 13- 16}\right)=\frac{720}{20917}<1. \hskip1cm \cont
\end{align*}
So $f(u)=1$. Hence $u$ is an odd square. Suppose for the moment that $u>1$. Then $u\ge 9$ and 
 \eqref{E:ooo6} gives
\begin{align*}
w'^2&<8\left(\frac1{9\cdot 5}+\frac1{5^2}\right)\left(1+\frac{16 }{5^3- 16}\right)=\frac{560}{981}<1. \hskip1cm \cont
\end{align*}
So $u=1$. Notice that as $f(v)=5$ we have $v=5m^2$ for some odd $m$ and so $u+v=1+5m^2$. In particular, $f(u+v)\not=2$ as otherwise
$1+5m^2=2r^2$ for some $r$. But this equation has no solution modulo 5. So, as the 2-adic order of $u+v$ is odd, we have $f(u+v)\ge 6$. 
Then  \eqref{E:ooo7} gives
\begin{align*}
w'^2&<\frac{16}6\left(\frac1{5}+\frac1{5^2}\right)\left(1+\frac{16 }{5^3- 16}\right)=\frac{80}{109}<1. \hskip1cm \cont
\end{align*}
So $f(v)\not=5$.

Now suppose that $f(v)=3$. By  Remark~\ref{R:mod44}, $f(u)\equiv 3\pmod 4$. So as  $u,v$ are relatively prime, and as $v$ is divisible by $3$, we have $f(u)\not=3$, so $f(u)\ge 7$.
Then  \eqref{E:ooo6} gives
\begin{align*}
w'^2&<8\left(\frac1{7^2\cdot 3}+\frac1{3^2\cdot 7}\right)\left(1+\frac{16 }{3^3\cdot 7- 16}\right)=\frac{240}{1211}<1. \hskip1cm \cont
\end{align*}
So $f(v)=1$.
\end{proof}

 \begin{lemma}\label{L:fv11}  If $v=1$, then $f(u)=1$.
\end{lemma}

\begin{proof}  Suppose $v=1$.
From \eqref{E:ooo4} we have  $y''>f(u)z'\ge 1$. 
Thus from Lemma \ref{L:mult4}(c) we have $f(u)\le 11$.
Moreover, as $v=1$, we have  $f(u)\equiv 1\pmod 4$, by  Remark~\ref{R:mod44}.
Thus, as $f(u)$ is square-free, $f(u)\le 5$. 
Suppose $f(u)= 5$. Then  from \eqref{E:ooo4} we have  $y''>5z'\ge 5$. So $y''\ge 6$. 
Then  \eqref{E:ooo6} gives
\begin{align*}
w'^2&<8\left(\frac1{5^2}+\frac1{5}\right)\left(1+\frac{16 }{5\cdot 6^2- 16}\right)=\frac{432}{205}<3. 
\end{align*}
So $w'=1$. Then  \eqref{E:ooo5} gives $f(u+v)  (5 y''^2- 16)= 16(5z'^2+ y''^2)$, so
\[ (5f(u+v) -16) y''^2= 80z'^2+ 16f(u+v)>0.
\]
So $f(u+v) >16/5$. Thus, as $u+v$ is  even and square-free, $f(u+v)\ge 6$. Hence \eqref{E:ooo7} gives
\begin{align*}
w'^2&<\frac{16}6\left(\frac1{5^2}+\frac1{5}\right)\left(1+\frac{16 }{5\cdot 6^2- 16}\right)=\frac{144}{205}<1. \hskip1cm \cont 
\end{align*}
So $f(u)< 5$ and since $f(u)\equiv 1\pmod 4$, we have $f(u)=1$.
\end{proof}


 \begin{lemma}\label{L:fv111}  If $v=1$, then $u=1$ and $w'=3$ or $4$.
\end{lemma}

\begin{proof} 
Suppose $v=1$. By the previous lemma, $f(u)=1$. Then \eqref{E:ooo5}  gives $f(u+v) w'^2 ( y''^2- 16)= 16(z'^2+ y''^2)$, and so
\begin{equation}\label{E:ooo12}
(f(u+v) w'^2 -16)y''^2= 16(z'^2+f(u+v) w'^2).
\end{equation}
From \eqref{E:ooo5i}, we have  $ y''>4$. Thus $ y''\ge 5$.
So \eqref{E:ooo7} gives
\begin{equation}\label{E:ooo13}
f(u+v) w'^2< 16\left(\frac1u+1\right)\left(1+\frac{16}{25-16}\right)=\frac{16\cdot 25}{9}\left(\frac1u+1\right).
\end{equation}
In particular, as $w' \ge 1$ and $u\ge 1$, this gives $f(u+v)< \frac{16\cdot 25 \cdot2}{9}=\frac{800}9$, so $f(u+v)\le 88$. Moreover,  $f(u+v)$ is  even and square-free, and furthermore,
as $u$ is an odd square, say $u=n^2$, and $n^2+1 = f(u+v)m^2$, where $m=s(u+v)$, we have that $-1$ is a quadratic residue modulo $f(u+v)$. Hence $f(u+v)$ cannot be divisible by a prime congruent to 3 modulo 4. It follows that the only possible values of $f(u+v)$ are:
\[
2,10,26,34,58,74,82.
\]
Notice also that by \eqref{E:ooo12}, we have $f(u+v) w'^2 >16$, so $w'\ge 2$ for $f(u+v)=10$ and $w'\ge 3$ for $f(u+v)=2$.

Let us assume for the moment that $u>1$. So, as $ f(u)=1$, we have $u\ge 9$. Then  from \eqref{E:ooo4}, we have  $y''^2>u z'^2\ge 9z'^2$. 
Moreover, \eqref{E:ooo13}  gives $f(u+v)< \frac{16\cdot 25 }{9}(\frac19+1)=\frac{4000}{81}<50$, so $f(u+v)\le 34$. For the resulting four cases of $f(u+v)$ we have:
 
 \begin{enumerate}
\item If  $f(u+v) =34$, then \eqref{E:ooo13}  gives $w'^2< \frac{16\cdot 25 }{34\cdot 9}(\frac19+1)=\frac{2000}{1377}<2$, so $w'=1$.  Then  \eqref{E:ooo12} gives  $9 y''^2= 8(z'^2+ 34)$, which is impossible modulo 3.

\item If  $f(u+v) =26$, then \eqref{E:ooo13}  gives $w'^2< \frac{16\cdot 25 }{26\cdot 9}(\frac19+1)=\frac{2000}{1053}<2$, so $w'=1$.  Then  \eqref{E:ooo12} gives  $5 y''^2= 8(z'^2+ 26)$. So as  $y''^2>9z'^2$, we have
$45 z''^2< 8(z'^2+  26)$; i.e., $z''^2< \frac{8\cdot 26}{37}$, so $z''\le 2$. But  $5 y''^2= 8(z'^2+ 26)$ has no integer solution for $y''$ when $z'=1$
or $z'=2$.

\item If  $f(u+v) =10$, then \eqref{E:ooo13}  gives $w'^2< \frac{16\cdot 25 }{10\cdot 9}(\frac19+1)=\frac{400}{81}<5$, so $w'\le 2$. 
But $w'\ge 2$ for $f(u+v)=10$, as we observed above, so $w'= 2$. Then \eqref{E:ooo12} gives  $3 y''^2= 2(z'^2+ 40)$, which is impossible modulo 3.

\item If  $f(u+v) =2$, then \eqref{E:ooo13}  gives $w'^2< \frac{16\cdot 25 }{2\cdot 9}(\frac19+1)=\frac{2000}{81}<25$, so $w'\le 4$. 
But $w'\ge 3$ for $f(u+v)=2$, as we observed above, so $w'= 3$ or $4$. 
 \begin{enumerate}
\item If $w'= 3$, then \eqref{E:ooo12} gives  
$  y''^2= 8(z'^2+ 18)$. But then $y''^2> 9z'^2$ gives $z'^2< 8\cdot 18$, so $z'\le 11$. But for none of these values does $  y''^2= 8(z'^2+ 18)$ have an integer solution for $y''$. So 
this case is impossible.

 \item If $w'= 4$, then \eqref{E:ooo12} gives  
$  y''^2= z'^2+ 32$. But then $y''^2> 9z'^2$ gives $z'^2< 4$, so $z'=1$. However then $  y''^2= z'^2+ 32$ has no integer solution for $y''$. So this case is also impossible.
\end{enumerate}
 \end{enumerate}

We conclude from the above that $u=1$. So $f(u+v) =2$ and
 \eqref{E:ooo13}  gives $w'^2< \frac{16\cdot 25 \cdot2}{2\cdot 9}=\frac{400}{9}<45$, so $w'\le 6$. 
But $w'\ge 3$ for $f(u+v)=2$, as we observed above, so $w'= 3,4,5$ or $6$. We will now eliminate the possibilities that $w'= 5$ or $6$.
 \begin{enumerate}
\item If  $w'=5$, then  \eqref{E:ooo12} gives  $17 y''^2= 8(z'^2+ 50)$. From \eqref{E:ooo4}, $y''^2>u z'^2\ge z'^2$.  So
$9z''^2< 8\cdot  50$ and hence $z''\le 6$. But for none of these values does $17 y''^2= 8(z'^2+ 50)$ have an integer solution for $y''$. So 
this case is impossible.

\item If  $w'=6$, then  \eqref{E:ooo12} gives  $7 y''^2= 2(z'^2+ 72)$, which is impossible  mod 7.

\end{enumerate}

\end{proof}


\begin{lemma}\label{L:fv19}  If $v>1$, then $u=1$, $w'=1$ and either $f(u+v)=2$ or $f(u+v)=10$.
\end{lemma}

\begin{proof}  
By the previous lemma, $f(v)=1$, so $v$ is an odd square, $v=m^2$ say. Suppose $m\ge 3$.
Note that $f(u)\equiv 1\pmod 4$, by  Remark~\ref{R:mod44}. If $f(u)\ge 5$, then \eqref{E:ooo6} gives
\begin{align*}
w'^2&<8\left(\frac1{5^2}+\frac1{5\cdot 9}\right)\left(1+\frac{16 }{9\cdot 5- 16}\right)=\frac{112}{145}<1. \hskip1cm \cont 
\end{align*}
So $f(u)< 5$ and since $f(u)\equiv 1\pmod 4$, we have $f(u)=1$.

First suppose that $v=9$.
Substituting $v=9, f(u)=1$ in \eqref{E:ooo5}  gives $f(u+v) w'^2 ( 9y''^2- 16)= 16(z'^2+ y''^2)$, and so
\begin{equation}\label{E:ooo15}
(9f(u+v) w'^2 -16)y''^2= 16(z'^2+f(u+v) w'^2).
\end{equation}
From \eqref{E:ooo5i}, we have $ 9y''^2>16$. Thus $ y''\ge 2$.
So \eqref{E:ooo7} gives
\begin{equation}\label{E:ooo14}
f(u+v) w'^2< 16\left(\frac1u+\frac19\right)\left(1+\frac{16}{9\cdot 4-16}\right)=\frac{16\cdot 9}{5}\left(\frac1u+\frac19\right).
\end{equation}
Let us assume for the moment that $u>1$. So, as $ f(u)=1$ and $\gcd(u,v)=1$, we have $u\ge 25$. Then $w'\ge 1$ and \eqref{E:ooo14}  give $f(u+v)< \frac{16\cdot 9 }{5}(\frac1{25}+\frac19)=\frac{544}{125}<5$ so, as $f(u+v)$ is  even and square-free, $f(u+v)=2$. Then we have
$m^2+9 =2n^2$ for some $n$. But considering this equation modulo 3, it follows that $n,m$ are both divisible by 3, contradicting the hypothesis that $\gcd(u,v)=1$. So $u=1$.
Thus, $f(u+v)=u+v=10$. Furthermore, \eqref{E:ooo14}  gives $w'^2<  \frac{16\cdot 9 }{10\cdot 5}(1+\frac19)=\frac{16}{5}<4$, so $w'=1$, as required.  

Now suppose that $v>9$, so $v\ge 25$. As $f(u)=1$, so $u,v$ are both odd squares. 
First suppose that if $u\ge 49$. Then since $v\ge 25$,
 \eqref{E:ooo7} gives
\begin{align*}
f(u+v)w'^2&<16\left(\frac1{49}+\frac1{25}\right)\left(1+\frac{16 }{25- 16}\right)=\frac{1184}{441}<3.
\end{align*}
So necessarily $f(u+v)=2$ and $w'=1$. Then \eqref{E:ooo5}  gives $vy''^2- 16= 8(z'^2+ y''^2)$, so as $v$ is odd, $y''$ is divisible by 4.
But then
\eqref{E:ooo6} gives
\begin{align*}
w'^2&<8\left(\frac1{49}+\frac1{25}\right)\left(1+\frac{16 }{25\cdot 4^2- 16}\right)=\frac{74}{147}<1. \hskip1cm \cont
\end{align*}
Hence $u=1,9$ or $25$. If $u=25$, then since  $u,v$ are relatively prime, we have $v\ge 49$ and thus
 \eqref{E:ooo6} gives
\begin{align*}
w'^2&<8\left(\frac1{25}+\frac1{49}\right)\left(1+\frac{16 }{49- 16}\right)=\frac{592}{825}<1. \hskip1cm \cont
\end{align*}
Hence $u=1$ or $9$. 

Suppose $u=9$. Then $u+v=9+m^2$, and if $f(u+v)=2$, then $u+v=2r^2$, for some $r$, and thus $9+m^2=2r^2$. Modulo 3 this would give $m\equiv 0\pmod 3$, contradicting the assumption that $u,v$ are relatively prime. Hence $f(u+v)>2$. In this case, using again the fact that $u,v$ are relatively prime, we would have $f(u+v)\ge 10$. Thus, 
if $v\ge 49$, then \eqref{E:ooo7} would give
\begin{align*}
w'^2&<\frac{16}{10}\left(\frac1{9}+\frac1{49}\right)\left(1+\frac{16 }{49- 16}\right)=\frac{464}{1485}<1. \hskip1cm \cont
\end{align*}
Hence  $v=25$. But in this case,  $f(u+v)=34$ and \eqref{E:ooo7}  gives
\begin{align*}
w'^2&<\frac{16}{34}\left(\frac1{9}+\frac1{25}\right)\left(1+\frac{16 }{25- 16}\right)=\frac{16}{81}<1. \hskip1cm \cont
\end{align*}
So $u\not=9$.

Finally, suppose $u=1$.
Thus \eqref{E:ooo5} gives
\begin{equation}\label{E:ooo8}
f(u+v) w'^2 (v y''^2- 16)= 16(z'^2+ y''^2).
\end{equation}
First suppose that $v=25$. Then $f(1+v)=26$. Note that if $y''>1$, then \eqref{E:ooo7}  gives
\begin{align*}
w'^2&<\frac{16}{26}\left(\frac1{1}+\frac1{ 25}\right)\left(1+\frac{16 }{25\cdot 4- 16}\right)=\frac{16}{21}<1. \hskip1cm \cont
\end{align*}
So $y''=1$. Then \eqref{E:ooo7}  gives
\begin{align*}
w'^2&<\frac{16}{26}\left(\frac1{1}+\frac1{ 25}\right)\left(1+\frac{16 }{25- 16}\right)=\frac{16}{9}<2,
\end{align*}
so $w'=1$. Substituting in \eqref{E:ooo8} gives $8z''^2=109$, which has no integer solution. Hence $v>25$ and thus $v\ge 49$.

For $v\ge 49$, \eqref{E:ooo6} gives
\begin{align*}
w'^2&<8\left(\frac1{1}+\frac1{ 49}\right)\left(1+\frac{16 }{49- 16}\right)=\frac{400}{33}<13,
\end{align*}
so $w'\le 3$. Suppose for the moment that $y''$ is odd. Then working modulo 4, as $v$ is odd and $ f(u+v) \equiv 2\pmod 4$, we conclude from \eqref{E:ooo8} 
that $w'$ is even. So $w'\le 3$ gives $w'=2$. Then \eqref{E:ooo8}  gives
$f(u+v)  (v y''^2- 16)= 4(z'^2+ y''^2)$,
and working modulo 4 again gives a contradiction. Thus $y''$ is even, say $y''=2y'''$, and \eqref{E:ooo8}  gives
\begin{equation}\label{E:ooo10}
f(u+v) w'^2 (v y'''^2- 4)= 4(z'^2+ 4y'''^2).
 \end{equation}
As $y''=2y'''\ge 2$, \eqref{E:ooo6} gives
\begin{align*}
w'^2&<8\left(\frac1{1}+\frac1{ 49}\right)\left(1+\frac{16 }{49\cdot 4- 16}\right)=\frac{80}{9}<9,
\end{align*}
so $w'\le 2$. 

Note that if $f(u+v)\ge 18$, then for $y''\ge 2$, \eqref{E:ooo7} would give
\begin{align*}
w'^2&<\frac{16}{18}\left(\frac1{1}+\frac1{49}\right)\left(1+\frac{16 }{49\cdot 4- 16}\right)=\frac{80}{81}<1. \hskip1cm \cont
\end{align*}
So as $f(u+v)$ is even and square-free, $f(u+v)=2,6,10$ or $14$. But $u+v=1+m^2$ and $-1$ is not quadratic residue modulo 6 or 14. So 
$f(u+v)=2$ or $10$. If $f(u+v)=10$, then \eqref{E:ooo7}  gives
\begin{align*}
w'^2&<\frac{16}{10}\left(\frac1{1}+\frac1{49}\right)\left(1+\frac{16 }{49\cdot 4- 16}\right)=\frac{16}{9}<2,
\end{align*}
so $w'=1$.  We will show that one also has $w'=1$ when $f(u+v)=2$. Indeed, suppose $f(u+v)=2$ and $w'=2$. Then \eqref{E:ooo10} gives $2 (v y'''^2- 4)= z'^2+ 4y'''^2$. As $f(u+v)=2$, we have $v=m^2=2n^2-1$ for some $n$. So we have
\[
2 (2n^2-3) y'''^2- 8= z'^2.
\]
However, this equation has no solution for $n,y''',z'$ modulo 64. So $w'=1$.
\end{proof}

\bigskip
If $v=1$, then from Lemma \ref{L:fv111},  $u=1$ and $w'=3$ or $4$.
Then \eqref{E:ooo5}  gives $w'^2 ( y''^2- 16)= 8(z'^2+ y''^2)$.
 \begin{enumerate}
\item If  $w'=3$, then we have  $y''^2= 8(z'^2+ 18)$. This equation has infinitely many solutions. Here $\Sigma= \frac12f(u+v)f(u)vf(v)w'^2=9$, and $\Sigma'=\Sigma$. So this is one of our desired solutions.

\item If  $w'=4$, then we have    $y''^2= z'^2+ 32$. This has two solutions ($y''=6,z'=2$ and $y''=9,z'=7$). Here $\Sigma= \frac12f(u+v)f(u)vf(v)w'^2=16$, and $\Sigma'=\Sigma$. So this is another  one of our desired solutions.
\end{enumerate}

If $v>1$, then from Lemma \ref{L:fv134}, $f(v)=1$ so $v=m^2$ for some odd $m$, and from Lemma~\ref{L:fv19},  $u=1$, $w'=1$ and either $f(u+v)=2$ or $f(u+v)=10$. 
Then \eqref{E:ooo5}  gives 
$(f(u+v) m^2-16)y''^2= 16f(u+v)+ 16z'^2$.
Then as $m$ is odd and $f(u+v)=2$ or $10$, we have that $y''$ is divisible by 4, say $y''=4Y$. So
$(f(u+v) m^2-16)Y^2= f(u+v)+ z'^2$.
Working mod 8  we see that $Y$ is necessarily odd and  $z'^2\equiv 0$, so $z'$ is divisible by 4, say $z'=4Z$. So we have
\begin{equation}\label{E:oooex}
\big(\frac{f(u+v)}2 m^2-8\big)Y^2= \frac{f(u+v)}2+ 8Z^2.
 \end{equation}
 Furthermore, $\Sigma= \frac12f(u+v)f(u)vf(v)w'^2=\frac{f(u+v)}2m^2$, and $\Sigma'=\frac{u}v\Sigma=\frac{f(u+v)}2$. 
 
 Thus when $f(u+v)=2$, we have $(\Sigma,\Sigma')=(m^2,1)$. Furthermore, there exists $n$ such that 
  \begin{equation}\label{E:oooex1}
m^2+1=2n^2\ \text{and}\ (m^2-8)Y^2=1+ 8Z^2,
 \end{equation}
where the latter equation comes from \eqref{E:oooex}. Similarly, 
 $f(u+v)=10$, there exists $n$ such that 
 \begin{equation}\label{E:oooex5}
m^2+1=10n^2\ \text{and}\ (5m^2-8)Y^2=5+ 8Z^2.
 \end{equation}
  So these are the two desired families of  solutions. This completes the proof of  Theorem~\ref{T:nt}.

\begin{remark}\label{R:mnos} Consider the integers $m$ for which there exists $n$ with $m^2+1=10n^2$, as in \eqref{E:oooex5}. It is easy to see that $m$ is necessarily divisible by 3. The numbers $m/3$ are well known; see entry
A097314 of \cite{OEIS}. The first eight values of $m$ are: $3, 117, 4443, 168717, 6406803, 243289797, 3079535161, 116941239519$.
\end{remark}

\begin{remark}\label{R:exextan} Suppose that in the case $u=1,v=m^2$ at the end of the above proof, we have a solution $m,n,Y,Z$ to \eqref{E:oooex1} or \eqref{E:oooex5}.  From \eqref{E:x}, as $y=y'=y''=4Y$, and $u+v=1+m^2=f(u+v)n^2$ and $\Sigma=\frac{f(u+v)}2m^2$, 
\[
x =\sqrt{\frac{y'^2  f(u)(u+v)\Sigma}{ 8}}
=f(u+v)mnY.
\]
Moreover, $z=f(v)s(v)f(u)z'=mz'=4mZ$ and from Definition~\ref{D:xyz},
$x=a+b,y=a-c,z=c-b$.

For $f(u+v)=2$, we have $x=2mnY$ and solving gives, using $d=a+b-c$, 
\begin{align*}
a&
=(mn+2)Y+2mZ,\qquad
b
=(mn-2)Y-2mZ,\\
c&
=(mn-2)Y+2mZ,\qquad
d
=(mn+2)Y-2mZ.
\end{align*}
Notice that as $m<\sqrt2 n$, we have $\Sigma=m^2<4mn=8\frac{a+b}{a-c}$.
Hence by Remark~\ref{R:conc}, all such  extangential LEQ (if any exist) are necessarily convex.

Similarly, for $f(u+v)=10$, we have $x=10mnY$ and solving gives
\begin{align*}
a&
=(5mn+2)Y+2mZ,\qquad
b
=(5mn-2)Y-2mZ,\\
c&
=(5mn-2)Y+2mZ,\qquad
d
=(5mn+2)Y-2mZ.
\end{align*}
Notice that as $m<\sqrt{10} n$, we have $\Sigma=5m^2<20mn=8\frac{a+b}{a-c}$.
Hence by Remark~\ref{R:conc}, all such  extangential LEQ  are necessarily convex.
\end{remark}

\begin{remark}\label{R:coro}
Note that we now have all the ingredients for the proof of Corollary~\ref{C:main} from the introduction.
The proof for the LEQs of Theorem~\ref{C:main} parts (b) and (c) are given in the previous remark. 
The proof for the LEQs of Theorem~\ref{C:main} part (a) were given in Section~\ref{S:extanleqs}; see Remark~\ref{R:convex1} and the analysis of LEQs with $(\Sigma,T)=(18,50)$.
\end{remark}

\begin{remark} We mention that in the case $m=3$ of \eqref{E:oooex5}, the equation $(5m^2-8)Y^2=5+ 8Z^2$ gives
$37 Y^2= 5+8Z^2$, which is equivalent to Equation~\eqref{P22} in Section~\ref{S:extanleqs};
the connection is given by setting $W=111 Y + 52 Z$.
\end{remark}

\begin{example}\label{Ex:expl} We now exhibit an extangential LEQ corresponding to the case $m=117$ of \eqref{E:oooex5}.  This is the case $n=37$ in Theorem~\ref{T:main}(b). According to \cite{Alpern}, the smallest solution to $(5\cdot 117^2-8)Y^2=5+8Z^2$ is
\[
Y=34884218483995340806373,\quad
Z=3226483779786979759026161.
\]
The formulas from Remark \ref{R:exextan} give
\begin{align*}
a&=1510135881993200406047678005,\quad b= 1936178957897460209165\\
c&=1509996345119264424684452513,\quad d=141473052893878823434657.
\end{align*}
Let
\begin{align*}
A&=(640848245491383541211578005,1367415046112187810865469000), \\
B&=(640849067137238673279485480, 1367416799305572965277883040),\\
C&=(60036158873125939312368,128102631990427959679265).
\end{align*}
It is easy to verify that the points $O,A,B,C$ form the vertices of an extangential LEQ with side lengths $a,b,c,d$ as given above (see below).
This LEQ has  $(\Sigma,T)=(5\cdot 117^2,5+5\cdot 117^2)$.  Note the perimeter is $a+b+c+d\cong 3.0 \cdot 10^{27}$.

Let us briefly explain how the above vertices were determined. First factorize $a$ and note that each prime factor is congruent to 1 mod 4. Then for each factor $\alpha$ of $a$, consider all ways of writing $\alpha$ as a sum of two squares: $\alpha=\alpha_1^2+\alpha_2^2$, where $\alpha_1>\alpha_2>0$, and using the Pythagorean formula,  consider the points $A_\alpha=\frac{a}{\alpha}(\alpha_1^2-\alpha_2^2,2\alpha_1\alpha_2)$ and $A_\alpha^*=\frac{a}{\alpha}(2\alpha_1\alpha_2,\alpha_1^2-\alpha_2^2)$. Let $P_a$ denote the union over $\alpha$ of all the sets $\{A_\alpha,A_\alpha^*\}$. Similarly, construct $P_b,P_c$ and $P_d$. Then search for
 members $S_a,S_b,S_c,S_d$ in $P_a,P_b,P_c,P_d$ respectively such that $S_a+S_b=S_c+S_d$, and  set $A=S_a,B=S_a+S_b,C=S_c$. 
 By construction, the resulting quadrilateral $OABC$ has side lengths $a,b,c,d$. The extangential condition, $a+b=c+d$, is verified directly from the above values of $a,b,c,d$.
Finally, check that 
 \[
 \det[S_a,S_b]>0\ \text{and}\  \det[S_c,S_d]>0,
 \]
 which shows that $OABC$ is positively oriented and has no self-intersection, and  check that the equability condition is satisfied, i.e., 
 \[
a+b+c+d=\frac12 \big( \det[S_a,S_b]+ \det[S_c,S_d]\big).
 \]
 \end{example}


\section{Comments on the Open Problem}\label{S:open}

In this section we make some comments on the Open Problem stated in the Introduction.
Suppose we have integers $m,n,Y,Z$ such that  the following two equations hold:
\begin{align}
m^2&=2n^2-1,\label{E:key1}\\
(m^2-8) Y^2&= 1+8Z^2. \label{E:key2}
\end{align}
For convenience, set $M=m^2-8$. We first make some elementary observations:

\begin{enumerate}
\item[1.] From \eqref{E:key1}, $m$ is odd. Then working modulo 4, as $m^2=2n^2-1$ and $m$ odd,  $n$ is also odd, and hence from \eqref{E:key2},  $Y$ and $M$ are also odd. 

\item[2.]  Working modulo 3, $2n^2-1\equiv \pm1$. So from \eqref{E:key1}, $m$ is not divisible by 3. 
Thus $m^2\equiv 1$, so from \eqref{E:key1} again, $n$ is not divisible by 3. Hence $M\equiv 2$. Thus, from \eqref{E:key2},   $Y$ is necessarily divisible by 3, and $Z$ is not divisible by 3.

\item[3.]  Working modulo 7, the quadratic residues are $0,1,2$ and $4$. So $2n^2-1$ is $0,\pm1$ or $3$. So as $m^2=2n^2-1$, we conclude  $m^2$ is $0$ or $1$. Then $M\equiv -1$ or $0$. But if $M\equiv0$ then \eqref{E:key2} gives $Z^2\equiv -1$, which is impossible. So $m$ is divisible by 7.

\item[4.]   The prime divisors of $m^2-8$ are all congruent to 1 mod 8. 
Indeed, suppose $p$ is a prime divisor of $m^2-8$. Then $8\equiv m^2\pmod{p}$. But it is well known that $8$ is a quadratic residue mod an odd prime $p$ if and only if $p$ is congruent to 1 or 7 mod 8. So $p$ is congruent to 1 or 7 mod 8. But  by \eqref{E:key2}, $p$ is also a prime divisor of $1+8Z^2$, so $-2\equiv (4Z)^2\pmod{p}$. But it is well known that $-2$ is a quadratic residue mod an odd prime $p$ if and only if $p$ is congruent to 1 or 3 mod 8. Hence $p$ is congruent to 1 mod 8.
\end{enumerate}

Let $m=2r+1$. Then $m^2=2n^2-1$ gives $4r^2+4r+2=2n^2$, so $r^2+(r+1)^2=n^2$. So the solutions $(m,n)$ to \eqref{E:key1} correspond to Pythagorean triangles $(r,r+1,n)$ whose base and height differ by 1. These triangles are well known; see entry A001652 of \cite{OEIS}. In particular, it is well known that the solutions $r_0,r_1,r_2,\dots$ satisfy $r_{i}=6r_{i-1} - r_{i-2} + 2$ with $r_{0} = 0, r_{1} = 3$. Let us denote the correspond values of $m$ by $m_i=2r_i+1$. So $m_{i}=6m_{i-1} - m_{i-2}$ with $m_{0} = 1, m_{1} = 7$.  As we saw in observation 3 above, we are only interested in values of $m$ that are divisible by 7. Note that modulo 7,
$m_i\equiv -m_{i-1} - m_{i-2}$, so $m_{i+1}\equiv-m_{i} - m_{i-1} \equiv m_{i-2}$. So, as $m_{0} \equiv 1, m_{1} \equiv 0,m_{2} \equiv-1$, we are only interested in the values $m_{1+3i}$. Set $\mu_i:=\frac17m_{1+3i}$.
The sequence $\mu_0,\mu_1,\mu_2,\dots$ is also well known; see entry  A097732 of \cite{OEIS}. In particular, it is known to satisfy the relation $\mu_i = 198\mu_{i-1} - \mu_{i-2}$, with $\mu_0=1, \mu_1=199$.

Table~\ref{T:first7} shows the first 7 values of $\mu_i$ and the prime divisors of the corresponding values of $M_i=49\mu_i^2-8$.  Notice that only for $\mu_0=1$ and $\mu_4= 1544598001$ is every prime divisor of  $M_i$ congruent to 1 mod 8, as required by Observation~4 above.  So the other cases of Table~\ref{T:first7} cannot be solutions to Equation \eqref{E:key2}.

\begin{table}
\begin{tabular}{c|c|c|c}
  \hline
$i$&   $\mu_i$&Factorization of $M_i=49\mu_i^2-8$ & Factors mod 8  \\\hline
0&1   & $41$&$1$\\
1& 199    & $23{\cdot}  239{\cdot}  353$&$7{\cdot}  7{\cdot}  1$\\
2& 39401  & $79{\cdot}  103{\cdot}  599{\cdot}  15607$&$7{\cdot}  7{\cdot}  7{\cdot}  7$\\
3&7801199   & $47{\cdot}  6771937{\cdot}  9369319$&$7{\cdot}  1{\cdot}  7$\\
4&1544598001   & $41{\cdot}  45245801{\cdot}  63018038201$&$1{\cdot}  1{\cdot}  1$\\ 
5&305822602999  & $41{\cdot}  71{\cdot}  239{\cdot}  424577{\cdot}  865087{\cdot}  17934071$&$1{\cdot}  7{\cdot}  7{\cdot}  1{\cdot}  7{\cdot}  7$\\
6&60551330795801& $223{\cdot} 2297{\cdot} 37223{\cdot} 302663{\cdot} 3553471{\cdot} 8761009$& $ 7{\cdot}  1{\cdot}  7{\cdot}  7{\cdot}  7{\cdot}  1$ \\
\hline
 \end{tabular}

\caption{The first 7 solutions to $m^2=2n^2-1$ with $m\equiv0\pmod7$}\label{T:first7}
\end{table}


Notice that for $M Y^2= 1+8Z^2$ we have $2M Y^2=X^2+2$, for $X=4Z$. By 
\cite[Theorem 5]{Mollin},  $2MY^2=X^2+2$ has no solution if the continued fraction expansion of $\sqrt{2M}$  has odd period length. In fact, this is the case for $m=7\mu_0=7$ ($M=41$); the continued fraction expansion of $\sqrt{82}$ is $9,\overline{18}$, which has odd period length $\ell(\sqrt{82})=1$. This shows that 
for $m=7$, \eqref{E:key2} has no solutions. Another proof  that \eqref{E:key2} has no solutions for $m=7$ is given by using \cite[Theorem 8]{Walker} or \cite{Yuan}. According to these results,
 $M Y^2-2W^2= 1$ has no solution if $R^2-2MS^2=-1$ has a solution. And in fact, for $M=41$ ($m=7$), $R^2-2MS^2=-1$ has the solution $R=9,S=1$.
 
From Table~\ref{T:first7}, we see that the next potential solution would be for $m=7\mu_4$. Here, already, the numbers are very large, and we have been unable to determine the continued fraction expansion of $\sqrt{2M_4}$. 
To see that there are no solutions for $m=7\mu_4$, we require a deeper result, due to Wei. As we observed above, if we have a solution $M,Y,Z$ to \eqref{E:key2}, then
$2MY^2=X^2+2$, where $X=4Z$.

\begin{proposition}[{{\cite[Prop.~4.4]{Wei}}}] Suppose that $M=p_1p_2\dots p_j$, where $p_i\equiv 1\pmod8$ for each $i$. If the equation $2MY^2=X^2+2$ has an integer solution $X,Y$, then 
\[
\prod_{i=1}^j \genfrac{(}{)}{.5pt}{0}{2}{p_i}_4=1,\]
where $\big(\frac{\, \cdot\, }{\cdot}\big)_4$ denotes the quartic residue symbol (see \cite[Chap. 5]{Lem}).
\end{proposition}

Recall that $\big(\frac{2}{p_i}\big)_4=\pm1$ and $\left(\frac{2}{p_i}\right)_4\equiv2^{(p_i-1)/4}\pmod{p_i}$. From Table~\ref{T:first7}, we have
$M_4=p_1p_2p_3$, where $p_1=41,p_2=  45245801,p_3= 63018038201$. Calculations show that $\big(\frac{2}{p_i}\big)_4=-1$ for $i=1,2,3$. Hence by Wei's Proposition, there are no solutions to \eqref{E:key2} for $m=7\mu_4$.

In fact, calculations show that for $7\le i\le 155$,   $M_i$ has a prime divisor  congruent to 7 mod 8, so these cases also cannot be solutions to Equation \eqref{E:key2}. In establishing this, the only difficulty is in factorizing $M_i$. Once a factor congruent to 7 mod 8 has been found, it is easy to verify that it is indeed a factor. To substantiate our claim, for each $i$ with $7\le i\le 155$ we exhibit an explicit prime divisor of $M_i$  congruent to 7 mod 8. Consider the following set of 62 primes  congruent to 7 mod 8:
\begin{align*}
P&=\{23, 47, 71, 79, 103, 167, 191, 223, 239, 263, 311, 359, 431, 479, 607, 719, 887, \\
& 983,1031, 1103, 1279, 1399, 1487, 1511, 1823, 1879, 2671, 2767, 3271, 3559, 4903, \\
&4943, 6823, 7583, 8231, 23447, 39551, 53527, 72559, 153511, 167911, 255511, \\
&625111, 869951, 1471271, 2593399, 10808983, 13980671, 39556927, 108732031, \\
&125448527, 160812623, 209110079, 627025159, 9707524087, 181155438071, \\
&291814585319, 3072313317767, 15238519898992991, 39834495682679591, \\
&15327739968951498750119, 110095018941508669324502008759\}.
\end{align*}
Now consider the set 
\begin{align*}
&R=\{40, 9, 1, 4, 21, 1, 6, 5, 4, 15, 19, 55, 2, 1, 10, 9, 1, 48, 11, 2, 50, 4, 9, 8, 1, 41, 9, 1, \\
&13, 4, 34, 22, 14, 9, 4, 1, 9, 59, 1, 61, 9, 5, 2, 9, 56, 26, 1, 4, 43, 1, 9, 32, 16, 46, 9, 4, 33,\\
 &1, 2, 58, 1, 9, 6, 5, 9, 2, 17, 27, 1, 28, 54, 1, 7, 4, 18, 5, 49, 15, 9, 1, 5, 2, 1, 29, 20, 9, 4, 37, \\
&2, 6, 1, 30, 5, 1, 4, 36, 9, 5, 44, 4, 60, 1, 10, 3, 1, 62, 9, 4, 39, 5, 8, 2, 1, 9, 5, 1, 23, 9, 24, \\
&51, 4, 57, 11, 1, 9, 4, 1, 2, 38, 31, 35, 5, 42, 4, 1, 52, 53, 1, 3, 45, 47, 9, 12, 5, 25, 1, 4, 8, 1\},
\end{align*}
and let $r_i$ denote the $i$-th member of $R$.
The enthusiastic reader will easily verify that for each $1\le i\le 149$, the $r_i$-th member of $P$ is a divisor of $M_{i+6}$.

It follows from the above  that the smallest possible value of $m$ for which there could potentially be a solution to \eqref{E:key1} and \eqref{E:key2}
would have $m\ge 7\mu_{156}$. We don't know if there is a solution for $m=7\mu_{156}$. In particular, we have been unable to find any factors of $M_{156}$, which is unsurprising as $M_{156}\cong 1.8\cdot 10^{718}$.

\begin{remark}\label{R:big}
Note that if a solution to \eqref{E:key1} and \eqref{E:key2} exists, and there is an extangential LEQ corresponding to case (c) of Theorem~\ref{T:main}, with sides $a,b,c,d$, then by Remark~\ref{R:exextan}, its perimeter  would be
\[
2(a+b)=4mnY> 4 mn>2\sqrt{2}m^2.
\]
In particular, if there is an extangential LEQ corresponding to $m=7\mu_{156}$, then the perimeter would be at least $5.0\cdot 10^{718}$. \end{remark}

\bibliographystyle{amsplain}
{}


\end{document}